\documentclass[10pt,draft]{amsart}

\usepackage[margin=2.5cm,bottom=2.5cm,top=2.5cm]{geometry}
\usepackage{mathtools}
\mathtoolsset{showonlyrefs}
\usepackage{amsmath,amssymb,amsthm}
\usepackage{amsfonts}
\usepackage{latexsym}
\usepackage{amscd}
\usepackage{xcolor}

\begin{document}
\newtheorem{lemme}{Lemma}[section]
\newtheorem{proposition}{Proposition}[section]
\newtheorem{cor}{Corollary}[section]
\numberwithin{equation}{section}
\newtheorem{theoreme}{Theorem}[section]
\theoremstyle{remark}
\newtheorem{example}{Example}[section]
\newtheorem*{ack}{Acknowledgment}
\theoremstyle{definition}
\newtheorem{definition}{Definition}[section]
\theoremstyle{remark}
\newtheorem*{notation}{Notation}
\theoremstyle{remark}
\newtheorem{remark}{Remark}[section]
\newenvironment{Abstract}
{\begin{center}\textbf{\footnotesize{Abstract}}%
\end{center} \begin{quote}\begin{footnotesize}}
{\end{footnotesize}\end{quote}\bigskip}
\newenvironment{nome}
{\begin{center}\textbf{{}}%
\end{center} \begin{quote}\end{quote}\bigskip}

\newcommand{\triple}[1]{{|\!|\!|#1|\!|\!|}}
\newcommand{\xx}{\langle x\rangle}
\newcommand{\ep}{\varepsilon}
\newcommand{\al}{\alpha}
\newcommand{\be}{\beta}
\newcommand{\de}{\partial}
\newcommand{\la}{\lambda}
\newcommand{\La}{\Lambda}
\newcommand{\ga}{\gamma}
\newcommand{\del}{\delta}
\newcommand{\Del}{\Delta}
\newcommand{\sig}{\sigma}
\newcommand{\ome}{\Omega^n}
\newcommand{\Ome}{\Omega^n}
\newcommand{\C}{{\mathbb C}}
\newcommand{\N}{{\mathbb N}}
\newcommand{\Z}{{\mathbb Z}}
\newcommand{\R}{{\mathbb R}}
\newcommand{\T}{{\mathbb T}}
\newcommand{\Rn}{{\mathbb R}^{n}}
\newcommand{\Rnu}{{\mathbb R}^{n+1}_{+}}
\newcommand{\Cn}{{\mathbb C}^{n}}
\newcommand{\spt}{\,\mathrm{supp}\,}
\newcommand{\Lin}{\mathcal{L}}
\newcommand{\SSS}{\mathcal{S}}
\newcommand{\F}{\mathcal{F}}
\newcommand{\xxi}{\langle\xi\rangle}
\newcommand{\eei}{\langle\eta\rangle}
\newcommand{\xei}{\langle\xi-\eta\rangle}
\newcommand{\yy}{\langle y\rangle}
\newcommand{\dint}{\int\!\!\int}
\newcommand{\hatp}{\widehat\psi}
\renewcommand{\Re}{\;\mathrm{Re}\;}
\renewcommand{\Im}{\;\mathrm{Im}\;}
\title[ 
Modified energies for the periodic gKdV
 ]
{
Modified energies for the periodic generalized KdV equation and applications 
}
  \author{Fabrice Planchon}
  \address{ Sorbonne Universit\'e, CNRS, IMJ-PRG F-75005 Paris, France}
  \email{fabrice.planchon@sorbonne-universite.fr} 
\author{Nikolay Tzvetkov}
\address{
CY Cergy Paris
Universit\'e, Laboratoire AGM,  Cergy-Pontoise, F-95000,UMR 8088 du CNRS
}
\email{nikolay.tzvetkov@u-cergy.fr}
\author{Nicola Visciglia}
\address{Dipartimento di Matematica, Universit\`a di Pisa, Italy}
\email{nicola.visciglia@unipi.it}
\thanks{ The first author was supported by ERC grant ANADEL no 757996, 
the second author by ANR grant ODA (ANR-18-CE40-0020-01), the third author acknowledge the Gruppo Nazionale per l' Analisi Matematica, la Probabilit\`a e le loro Applicazioni (GNAMPA) 
of the Istituzione Nazionale di Alta Matematica (INDAM)} 
\date{\today}
 \maketitle
 \par \noindent
\centerline {\em Dedicated to the memory of Professor Jean Ginibre}
\begin{abstract}
We construct modified energies for the generalized KdV equation. As a consequence,  we obtain quasi-invariance of the high order Gaussian measures
along with $L^q$ regularity on the corresponding Radon-Nykodim density, as well as new bounds on the growth of the Sobolev norms of the solutions.
\end{abstract}
\section{Introduction}
\subsection{Integrable KdV models and their conservation laws}
The Korteweg de Vries (KdV) and modified Korteweg de Vries (mKdV) equations are canonical integrable partial differential equations which read 
\begin{equation}\label{KdV_mKdV}
\partial_t u - \partial_{x}^3 u + \partial_x(u^{p+1})= 0,\quad p=1,2,
\end{equation}
with $p=1$ for KdV equation and $p=2$ for mKdV. The well-known Miura transform may be used to connect solutions of KdV and mKdV.
Both equations \eqref{KdV_mKdV} have a Lax pair formulation and as a consequence, possess infinitely many conservation laws (see e.g.   \cite{Kappeler,schur,Zhidkov} and references therein).  
 One important consequence of these conservation laws is a priori global in time bound for the Sobolev norm $H^k$ of the solutions of  \eqref{KdV_mKdV}:  for every $k\in\N$, there exists a first integral that writes
 $$
 \int (\partial_x^k u)^2 dx +{\rm lower\,\, order\,\,terms}
 $$
where integration holds on the line $\R$ or on a torus $\R/(2\pi\Z)$ depending on which background \eqref{KdV_mKdV} is considered. 
{  Following \cite{Zhidkov}, another interesting consequence of the aforementioned conservation laws is the existence of invariant measures along the flow associated with \eqref{KdV_mKdV}; such measures are also absolutely continuous 
with respect to Gaussian measures (see the subsection \ref{quamuk} for their definition).}
 \subsection{Generalized KdV models and their almost conservation laws}
The Lax pair formulation is no longer known to hold if the power $p$ is replaced by $p\neq 1,2$ in \eqref{KdV_mKdV}. Therefore  a natural question is whether conservation laws for KdV and mKdV may still hold in the case of a general power $p$.
We introduce in this paper a family of energies which are not exact conservation laws, however will be useful to get
 new results for the long time behavior of the solutions to \eqref{gKdV}, as we shall specify in the sequel. %
 \\
 
 Consequently, we consider the Cauchy problem, which we call the generalized KdV (gKdV) equation,
\begin{equation}\label{gKdV}
\begin{cases}
\partial_t u - \partial_{x}^3 u + \partial_x(u^{p+1})= 0, \quad p \geq 3, \quad p\in 2\N,\\%
u(0, x)=\varphi(x)\in H^k(\mathbb T),
\end{cases}
\end{equation}
where $\T$ denotes the one dimensional torus $\R/(2\pi\Z)$. 
Even if \eqref{gKdV} does not have a Lax pair formulation it still has an Hamiltonian formulation, with conserved Hamiltonian ${\mathcal H(u)}$:
$$
{\mathcal H}(u)= \frac 12 \int_\T (\partial_x u)^2 dx + \frac 1{p+2} \int_\T u^{p+2} dx\,.
$$ 
That ${\mathcal H}$ is preserved along the flow of \eqref{gKdV} and some of its truncations will play a key role in our analysis. 
\\

Equation \eqref{gKdV} is the so called defocusing model. The focusing model can be obtained by a change of sign in front of the nonlinear term in \eqref{gKdV} (which also changes the sign in front of the second term defining ${\mathcal H}$).  
Most of the results that we will obtain below can be extended to either odd $p$ or to the focusing gKdV for even $p$, provided that a uniform bound on the $H^1$ norm of the solution is assumed. These extensions do not require any new significant argument and therefore we restrict ourselves to \eqref{gKdV}. 
\\

The Cauchy problem \eqref{KdV_mKdV} has a long and interesting history. The first results were mainly dealing with \eqref{gKdV} on the spatial domain $\R$ rather than the torus $\T$. A classical reference is the work by Saut \cite{saut}. In the fundamental paper \cite{kato}, Kato initiated the use of the dispersive properties in the analysis of \eqref{gKdV}.  In \cite{KPV1,KPV2,KPV3} Kenig-Ponce-Vega implemented tools from harmonic analysis allowing to have a deeper view on the dispersive effect. However, these tools were not efficient in the periodic case. The dispersive effect in the periodic case was understood by Bourgain in \cite{B1}. The methods developed by Bourgain had far reaching consequences in the field of dispersive PDE's. We will benefit of Bourgain's seminal work in this paper too.  In the case $p=4$ and focusing nonlinearity, when \eqref{gKdV} is posed on the real line, there are remarkable results by Martel-Merle-Rapha\"el on solutions developing singularities in finite time, see \cite{MMR1,MMR2,MMR3} and the references therein.  It is not clear how to extend the results of   \cite{MMR1,MMR2,MMR3}  to the periodic case  because on $\R$ the blow-up point may escape to infinity and therefore its localization seems a considerable challenge.  For a recent progress in this direction, we refer to \cite{MP}. 
\\

Along the paper we shall also need to analyze Fourier truncations of \eqref{gKdV}.  
For a given even integer $p\geq 3$ and $M\in \N$, we consider the following truncated version of the gKdV equation \eqref{gKdV},
\begin{equation}\label{gKdVtrunc}
\begin{cases}
\partial_t u -\partial_x^3 u+\pi_M \partial_x((\pi_M u)^{p+1}) =0, \\
u(0,x)=\varphi(x).
\end{cases}
\end{equation}
where for $M\in \N$ we denote by $\pi_M$, the Dirichlet projector defined by
$
\pi_M\Big(
\sum_{n\in\Z}c_n e^{inx}
\Big)
=
\sum_{|n|\leq M}c_n e^{inx}\,.
$
Thanks to the aforementioned references on the Cauchy problem for \eqref{gKdV},  solving \eqref{gKdVtrunc} will not be an issue here because we know from \cite{B1,B2,CKSTT,Staf1} that \eqref{gKdVtrunc} is globally well-posed in $H^s(\T)$, $s\geq 1$ (and even locally well-posed for $s\geq \frac {1} {2}$), { where $H^{s}(\T)$ is defined through the Fourier transform as $((1+|n|^{s}) c_{n})_{n} \in l^{2}(\Z)$}. In fact, for $M<\infty$ an easier argument can be applied by using that for frequencies $\leq M$ the equation \eqref{gKdVtrunc} becomes an ODE (with a Lipschitz vector field) for which global well-posedness holds thanks to the $L^2$ conservation law, while for frequencies $>M$ the equation \eqref{gKdVtrunc} becomes a linear equation. However, if one wishes to have $H^s(\T)$ bounds uniformly with respect to $M$, the analysis of \cite{B1,B2,CKSTT,Staf1}  cannot be avoided. Let us denote by $\Phi_M(t)$ the flow of  \eqref{gKdVtrunc}
and $\Phi(t)$ the flow of  \eqref{gKdV}. We shall note specify the dependence on $p$ in these notations. \\

{ Our aim is twofold: first, we provide new
a priori polynomial bounds on the growth of high order Sobolev norm for solutions to \eqref{gKdV} (see Theorem \ref{growth}); second, we gain knowlegde about transport of Gaussian measures along the flow associated with \eqref{gKdV} (see Theorem \ref{quasiinvariance}).}
\\

Let us introduce a few notations that will be used in the sequel: for every $s\in \R$ and for every $p\in [1,\infty]$ 
denote $H^s=H^s(\T)$  and $L^p=L^p(\T)$; $H^{s^-}$ will denote any Sobolev space $H^{s-\varepsilon}$ with $\varepsilon>0$ small enough.
Similarly, given a real number $\alpha$, denote by $\alpha^+$ any real number larger than $\alpha$. We shall denote $\N=\{0,1,2, \cdots\}$
and for every $k\in \N\setminus \{0\}$
we shall use the following specific Sobolev norm: $\|\varphi\|_{H^k}^2=\|\varphi\|_{L^2}^2+ \|\partial_x^k \varphi\|_{L^2}^2$.
We recall that $\Phi(t)$ and $\Phi_M(t)$ denote respectively the 
nonlinear flows associated with \eqref{gKdV} and \eqref{gKdVtrunc}. We shall also use sometimes the notation $\Phi(t)=\Phi_\infty(t)$.
The Dirichlet projector on frequencies smaller or equal than $M$ is denoted by $\pi_M$ and we also
use sometimes the notation $\pi_\infty$ to denote the identity operator.
\\

We can now state our main results.

\subsection{Growth of Sobolev norms}\label{grO}
The first main result of this paper is the following one.
\begin{theoreme}\label{growth} Let $k>1$ be an integer. Then for every  $\varphi \in H^k$ and for every $\varepsilon>0$ there exists a constant $C>0$ such that
\begin{equation}\label{grsh}\|\Phi(t) \varphi\|_{H^k}\leq \| \varphi\|_{H^k}+C\, t^{\frac{k-1}{2}+\varepsilon}, \quad \forall t>0.\end{equation}
\end{theoreme}
Although Theorem \ref{growth} is stated for $t>0$, it can be extended to every $t\in \R$ by the reversibility of the flow associated with \eqref{gKdV}. {  Moreover, the constant $C$ may be chosen uniformly with respect to bounded sets in $H^{k}$, and, even better, in $H^{1}$, once $\varepsilon$ has been chosen (through the use of Proposition \ref{Cauchy}).}

The line of research leading to results as the one in Theorem~\ref{growth} was initiated in \cite{B5}. 
We improve results obtained in \cite{Staf1} (where the growth had exponent $2k$) and \cite{Atanas} (where the growth was lowered to $k-1+\varepsilon$.)
\\

It would be very interesting to construct solutions of the defocusing gKdV such that the $H^k$ norms do not remain bounded in time for $k>1$. Unfortunately such results are rare in the context of canonical dispersive models 
(with the notable exception of \cite{HPTV}).
{  Our approach is restricted so far to 1d models and it would be very interesting to extend it to higher dimensions, however we believe it does apply to the 1d NLS modulo some additional difficulties related to complex valued functions. This will be addressed elsewhere.
\\

The proof of Theorem \ref{growth} stems from  constructing suitable modified energies at the level of the $H^{k}$ norms: while we are unable to obtain invariant energies,  we quantify how far they are from being exact conservation laws. The closer we are to conservation laws, the better are polynomial bounds for corresponding nonlinear solutions.}

\begin{theoreme}\label{basicA}
Let $k>1$ be an integer. Then for every $T>0$  there exist functionals
$${\mathcal E}_k:H^k\longrightarrow \R, \quad {\mathcal G}_{k}^T:H^k\longrightarrow \R,$$
such that:
\begin{itemize} 
\item for all $\varphi\in H^k$ and all $T>0$ we have:
\begin{equation}\label{timederA}
{\mathcal E}_k(\Phi(T) \varphi) - {\mathcal E}_k(\varphi)= {\mathcal G}_{k}^T(\varphi);
\end{equation}
\item  the energy ${\mathcal E}_k$ has the following structure:
\begin{equation}\label{rku}
{\mathcal E}_k(u)= \|\partial_x^k u\|_{L^2}^2 + \mathcal R_k(u),
\end{equation}
with ${\mathcal R}_k(u)$ satisfying for all $u\in H^k$:
\begin{equation}\label{RestimateA}
|\mathcal R_k(u)|\leq C +C  \|u\|_{H^k}^{\frac{2k-4}{k-1}}\|u\|^{p+\frac{2}{k-1}}_{H^1};
\end{equation}
\item for every fixed $R>0$ and $T>0$, there exists a constant $C>0$ such that, for all $\varphi\in H^k$ such that  $\|\varphi\|_{H^1}<R$ we have:
\begin{equation}\label{GestimateA}
|{\mathcal G}_{k}^T(\varphi)|
\leq C +C \|\varphi\|_{H^{k}}^{(\frac{2k-4}{k-1})^+}.
\end{equation}
\end{itemize}
\end{theoreme}

{ We emphasize that Theorem \ref{basicA} immediatly holds, in a stronger form, for $k=0,1$, with energies ${\mathcal E}_0(u)=\|u\|_{L^2}^2$, ${\mathcal E}_1(u)={\mathcal H}(u)$; those energies 
are indeed conserved along the flow associated with \eqref{gKdV}. However, for $k>1$, such exact conservation laws are not available and hence 
the r.h.s. in \eqref{timederA} is non zero.}
\\

{ The proof of Theorem~\ref{basicA}  relies on a key improvement on some ideas developed in \cite{PTV0} 
in the context of NLS,
together with bounds resulting from dispersive estimates such as Bourgain's $L^6$ Strichartz inequality for (linear) KdV.
Let us now briefly sketch the main point: adaptating the argument  in \cite{PTV0} would allow to get suitable energies $\tilde {\mathcal E}_k$ such that 
$\tilde {\mathcal E}_k(\Phi(T) \varphi)-\tilde{\mathcal E}_k(\varphi)= \tilde {\mathcal G}_{k}^T(\varphi)$, where $\tilde {\mathcal G}_{k}^T$ would satisfy $|\tilde {\mathcal G}_{k}^T(\varphi))| 
\leq C +C \|\varphi\|_{H^{k}}^{\alpha(k)}$, but with $\alpha(k)>(\frac{2k-4}{k-1})^+$. Here, we get a significantly smaller power of the $H^{k}$ norm of the initial datum in \eqref{GestimateA} for ${\mathcal G}_{k}^T$. %
This improvement on $\alpha(k)$, which in turn is crucial in order to obtain the growth
of the Sobolev norms in Theorem~\ref{growth}, follows from a refinement of the construction of  the energies ${\mathcal E}_k$ compared with $\tilde {\mathcal E}_k$. In particular, once we compute the variation of the energies introduced in this paper along solutions, we get that ${\mathcal G}_{k}^T(\varphi)$ can be expressed as the space-time integral of multilinear expression of densities in which the worse single terms (namely the ones  
that carry the maximal number of derivatives) have at least five factors involving at least one derivative. This key property
of distributing derivatives on several factors was out of reach with our previous constructions of modified energies. For details we refer to Section~\ref{N=infty}. Then the dispersive effect, through the $L^6$ Strichartz bound, allows us to transform the aforementioned distribution of derivatives in terms of powers of Sobolev norms of the initial datum, as discussed above.}
\\

We should point out that, in the context of gKdV, modified energies similar to the ones we get 
in Theorem \ref{basicA}
already appeared, at the level of $k=2$, in \cite{MartelAJM}, where they are used in connection with $N$-soliton asymptotics.\\

{ For details on how Theorem~\ref{basicA} implies Theorem~\ref{growth} we refer to \cite{PTV0}, however for the sake of completness
we briefly outline the argument. By local Cauchy theory, for any given $\varphi\in H^k$ with $k>1$,
there exists $T=T(\|\varphi\|_{H^1})$ and a constant $C=C(\|\varphi\|_{H^1})$ such that 
$\sup_{t\in (0, T)} \|\Phi(t)\varphi\|_{H^k}\leq C \|\varphi\|_{H^k}$.
From $\sup_t  \|\Phi(t)\varphi\|_{H^1}<\infty$ (recall we are considering the defocusing equation), we may use the estimate above  uniformly, selecting as initial condition $\Phi(s)\varphi$ for arbitrary $s$ and hence
\begin{equation}\label{localth}\sup_{t\in (s, s+T)}\|\Phi(t)\varphi\|_{H^k}\leq C \|\Phi(s)\varphi\|_{H^k}, \quad \forall s>0.\end{equation}
Combining \eqref{timederA} with \eqref{rku} and recalling conservation of the $L^2$ norm, we get, for our chosen $T$,
$$\|\Phi(t+T)\varphi\|_{H^k}^2 -\|\Phi(t)\varphi\|_{H^k}^2={\mathcal G}_{k}^T(\Phi(t) \varphi)-{\mathcal R}_k
(\Phi(t+T)\varphi)+ {\mathcal R}_k
(\Phi(t)\varphi)$$ 
In turn,
\eqref{GestimateA} and \eqref{RestimateA} (recalling that the $H^1$ norm on the r.h.s. is uniformly bounded) together imply,
\begin{align*}
\|\Phi(t+T)\varphi\|_{H^k}^2 - \|\Phi(t)\varphi\|_{H^k}^2 
\leq  & C +C \|\Phi(t+T)\varphi\|_{H^{k}}^{(\frac{2k-4}{k-1})^+} +C \|\Phi(t)\varphi\|_{H^{k}}^{(\frac{2k-4}{k-1})^+}\\
 \leq & C +C \|\Phi(t)\varphi\|_{H^{k}}^{(\frac{2k-4}{k-1})^+}\,,
\end{align*}
 for all $t>0$, where we used \eqref{localth} at the last step. Choosing $t=nT$ and defining 
$\alpha_n=\|\Phi_{nT}(\varphi)\|_{H^k}^2$ we get a discrete Gronwall-type inequality
$\alpha_{n+1}-\alpha_n\leq C + C \alpha_n^{(\frac{k-2}{k-1})^+}$ ; this yields $\alpha_n\leq C n^{(k-1)^+}$ which in turn implies \eqref{grsh} for $t=nT$. Using \eqref{localth}, this  extends to all $t$.}

\subsection{Quasi-invariant Gaussian measures}\label{quamuk}

Next we focus on the quasi-invariance of certain Gaussian measures under $\Phi(t)$. 
For $k\geq 1$, we denote by $\mu_k$ the measure induced by the map 
\begin{equation*}
\omega\longmapsto \sum_{n\in\Z}\frac{g_n(\omega)}{(1+n^2)^{k/2}}\, e^{inx} \,,
\end{equation*}
where  $g_n=\overline{g_{-n}}$, $g_0=0$ and $(g_n)_{n>0}$ is a sequence of independent, identically distributed complex Gaussian random variables.
We can see $\mu_k$ as a probability measure on $H^{(k-\frac 12)^-}$. 
We aim at understanding how $\mu_k$ is transported under the flow of \eqref{gKdV}. 
This is a delicate task because infinite dimensional measures become easily mutually singular. 
As ${\mathcal H}$ is preserved along the flow of \eqref{gKdV}, studying the transport of $ \chi_R ({\mathcal H}(u)) d \mu_k(u)$ is more natural, where $R>0$ is a fixed energy level and $\chi_R :\R\rightarrow\R$ is a continuous function vanishing outside $[-R,R]$. 
Such cut-offs were first introduced in the field of dispersive PDE's in \cite{LRS}. 
\\

For $M<\infty$,  verifying that the transport of the measure $ \chi_R ({\mathcal H}(\pi_M u)) d \mu_k(u)$
 by $\Phi_M(T)$ is absolutely continuous with respect to its initial value %
 will be relatively easy, with a Radon-Nikodym derivative given by 
\begin{equation}\label{FDFD}
g_{T, M}(u)=e^{\|\pi_M u\|_{H^k}^2-\|\pi_M (\Phi_M(T)u)\|_{H^k}^2  }\,.
\end{equation}
The main difficulty is passing to the limit $M\rightarrow\infty$ in $f_{T,M}$.
The issue is that each term in the exponential in \eqref{FDFD} strongly diverges in the limit $M\rightarrow\infty$ and therefore subtle cancellations should be exploited. 
A suitable adaptation of the modified energies introduced in Theorem~\ref{basicA} can be useful to overcome this difficulty,
more precisely we shall use the need the following result.

 \begin{theoreme}\label{basicB} 
Let $k>1$ be an integer. Then for every  $T>0$, $M\in \N\cup \{\infty\}$  there exist functionals
$$\bar {\mathcal E}_k:H^k\longrightarrow \R, \quad \bar {\mathcal G}_{k,M}^T:H^{(k-\frac 12)^-}\longrightarrow \R,$$
such that:

\begin{itemize}
\item for every $M\in \N\cup \{\infty\}$, $T>0$ and $\varphi\in H^k$ we have:
\begin{equation}
\label{timeder}
\bar {\mathcal E}_k(\pi_M(\Phi_M(T) \varphi)) - \bar {\mathcal E}_k(\pi_M \varphi)= \bar {\mathcal G}_{k,M}^T(\pi_M \varphi);
\end{equation}
\item the energy $\bar {\mathcal E}_k$ has the following structure:
\begin{equation}\label{rkubar}
\bar {\mathcal E}_k(u)= \|\partial_x^k u\|_{L^2}^2 + \bar{\mathcal R}_k(u),
\end{equation}
where $\bar {\mathcal R}_k: H^{(k-\frac 12)^-}\rightarrow \R$ satisfies for all $u\in H^{(k-\frac 12 )^-}$
the following bound:
\begin{equation}\label{RestimateB}
|\bar{\mathcal R}_k(u)|\leq C +C \|u \|_{H^{(k-\frac 12)^-}}^{(\frac{4k-8}{2k-3})^+} \|u\|_{H^1}^{p+\frac 2{2k-3}};
\end{equation}
\item
for every fixed $R>0$ and $T>0$, there exists a constant $C>0$, uniform w.r.t. to $M\in \N\cup \{\infty\}$, such that for all $\varphi\in H^{(k-\frac 12)^-}$ such that $\|\varphi\|_{H^1}<R$ we have:
\begin{equation}\label{Gestimate}
|\bar {\mathcal G}_{k,M}^T(\pi_M \varphi))|
\leq C +C 
\|\varphi \|_{H^{(k-\frac 12)^-}}^{(\frac{4k-8}{2k-3})^+};
\end{equation}
\item
for every $\varphi\in H^{(k-\frac12)^-}$and for every $T>0$ we have the following  convergence:
\begin{equation}\label{Gestimateprime}
\bar {\mathcal G}_{k,M}^T(\pi_M \varphi) \overset{M\rightarrow \infty}\longrightarrow \bar {\mathcal G}_{k,\infty}^T(\varphi).
\end{equation}
\end{itemize}
\end{theoreme}

{ 
Let us compare Theorems \ref{basicA} and \ref{basicB}: we point out that both energies ${\mathcal E}_k$ and
$\bar {\mathcal E}_k$ coincide; however we elected to use different notations, as
the corresponding lower order part of ${\mathcal E}_{k}$, namely ${\mathcal R}_k$, introduced along Theorem \ref{basicA},  is defined on $H^k$, while $\bar {\mathcal R}_k$ introduced in Theorem \ref{basicB} is defined on the larger space $H^{(k-\frac 12)^-}$. It will be clear along the proof that
${\mathcal R}_k(u)=\bar {\mathcal R}_k(u)$ for $u\in H^k$ but $\bar{\mathcal R}_k$ to be defined on $H^{(k-\frac 12)^-}$ is crucial for our purpose: when dealing with the analysis of transported Gaussian measures $\mu_k$,  working  at regularity $H^k$ is not sufficient and  one needs to go below at lower regularity $H^{(k-\frac 12)^-}$, as $\mu_k(H^{(k-\frac 12)^-})=1$ and 
$\mu_k(H^k)=0$. The same comment applies to functionals
$ {\mathcal G}_{k}^T$ and $\bar {\mathcal G}_{k,\infty}^T$ appearing in Theorems \ref{basicA} and \ref{basicB}. While they coincide on $H^k$, we consider the second one as defined on the larger space
$H^{(k-\frac 12)^-}$. Estimates \eqref{RestimateB} and \eqref{Gestimate} will be crucial in the sequel and should be compared with \eqref{RestimateA} and \eqref{GestimateA}. Notice how  the power of the $H^1$ norm differs in \eqref{RestimateA} and \eqref{RestimateB}, while, in \eqref{Gestimate} we loose more derivatives but we gain a smaller power compared to \eqref{GestimateA}: in the r.h.s. of \eqref{RestimateB} and \eqref{Gestimate} we have the $H^{(k-\frac 12)^-}$ norm of the initial datum to a power
that is less than two, and this is crucial for using  some standard Gaussian analysis. 
This is the key to proving, beyond quasi-invariance, $L^q$ regularity for the density of the transported Gaussian measure (see Theorem~\ref{quasiinvariance}).\

Let us also comment on the introduction, in Theorem \ref{basicB}, unlike in Theorem \ref{basicA}, of a family of functionals
$\bar {\mathcal G}_{k,M}^T$ for $M\in \N$. In order to study properties of the transported Gaussian measures
along the infinite dimensional flow associated with \eqref{gKdV}, we first need to analyze the variation of the energy $\bar {\mathcal E}_k$ along the flow
$\Phi_M(t)$ associated with \eqref{gKdVtrunc} (see \eqref{timeder}). The functionals $\bar {\mathcal G}_{k,M}^T$ turn out to be strongly related to $\bar {\mathcal G}_{k,\infty}^T$ (and hence to the functional ${\mathcal G}^T_{k}$ in Theorem \ref{basicA}). Actually they do look alike except for the Dirichlet projector $\pi_M$  that appears in $\bar {\mathcal G}_{k,M}^T$ for $M\in \N$. However the nice properties enjoyed  by projectors $\pi_M$ allow to estimate $\bar {\mathcal G}_{k,M}^T$ uniformly w.r.t. $M$.
\\

As Theorem \ref{basicA}, Theorem~\ref{basicB}  relies similarly on a key improvement on \cite{PTV}.
The energies that we introduce along Theorem \ref{basicB} allow to improve on the power that we get on the r.h.s. in \eqref{RestimateB} and \eqref{Gestimate} when compared to what  we would get by simply adapting the construction used in \cite{PTV} in the NLS context.\\

We can now give  the precise statement of our quasi-invariance result for gKdV. }
\begin{theoreme}\label{quasiinvariance} 
Let $k>1$ be an integer. The Gaussian measure $\mu_k$ is quasi invariant by the flow $\Phi(T)$ for every $T>0$. Moreover,  for every fixed $R>0$ 
and $T>0$ we have, for $A$ a Borel subset of $H^{(k-\frac 12)^-}$,
$$
\int_{\Phi(T)A} \chi_R({\mathcal H}(u)) d\mu_k(u)=\int_A  g_T( u) \chi_R ({\mathcal H}(u)) d \mu_k(u),
$$
where $g_T(u) \chi_R({\mathcal H}(u)) \in L^q(\mu_k)$ with $q\in [1,\infty)$ for $k>2$ and $q=\infty$ for $k=2$. In addition,  we have
$$
g_T(u)=\exp\Big(-\bar {\mathcal R}_k (u)+\bar {\mathcal R}_k (\Phi(T)u)-\bar {\mathcal G}_{k,\infty}^T (u)\Big)\, ,
$$ 
where $\bar {\mathcal R}_k$ and $\bar {\mathcal G}_{k,\infty}^T$, as introduced in  Theorem~\ref{basicB}, are well-defined quantities on the support of $\mu_k$. 
More precisely, with $g_{T,M}$ defined in \eqref{FDFD}, we have
$$
\lim_{M\rightarrow\infty}\|g_{T,M}(u)\chi_R ({\mathcal H}(\pi_M u))-g_T(u)\chi_R({\mathcal H}(u)) \|_{L^q(d\mu_k(u))}=0\,.
$$
\end{theoreme}
We point out that with such quasi invariance we also get $L^q$ regularity for the Radon-Nykodim density.
Exactly as for Theorem \ref{growth}, the result in Theorem \ref{quasiinvariance} can be extended to the case $T<0$ by using the reversibility of the flow
associated with \eqref{gKdV}. However for simplicity we shall focus on the case $T>0$.
\\

Theorem~\ref{quasiinvariance} fits in the line of research aiming to describe macroscopical (statistical dynamics) properties of Hamiltonian PDE's. {  In particular, it implies a stability property of the corresponding infinite dimensional Liouville equation (see \cite{STz}, Corollary 1.3)}. The earliest references we are aware of is \cite{F}, followed by \cite{LRS, Zhidkov,B3,B4}. Inspired by the work on invariant measures for the Benjamin-Ono equation  \cite{DTV,TV13a,TV13b,TV14}, quasi-invariance of Gaussian measure for several dispersive models was obtained in recent years, see \cite{deb, gauge, GLV, FS,forl, GOTW,OS,OT1,OT2,OT3,OTT,PTV,phil,sigma}.
The method to identify the densities in Theorem~\ref{quasiinvariance} is inspired by recent works \cite{deb,GLT_21}. In Theorem~\ref{quasiinvariance}, we provide  much more information on the densities when compared to \cite{PTV}, which used modified energies on the nonlinear Schr\"odinger equation. It should be underlined that a key novelty in the proof of   Theorem~\ref{quasiinvariance} with respect to \cite{GLT_21} and \cite{PTV} is that we crucially use dispersive estimates in the analysis.
\\

{  We hope that the approach developed here may allow to identify the Radom-Nikodym derivatives in the quasi-invariance obtained in \cite{PTV}, up to additional difficulties related to complex valued functions in the context of NLS.}
\\

Let us finally comment on the case $k=1$. 
In this case the Gibbs measure which is absolutely continuous with respect to $\mu_1$ is an invariant measure. This is precisely the result obtained in \cite{richards} in the case $p=3$
and \cite{ChKi}  in the general case $p\geq 3$ (see also \cite{ORT} for weak solutions in the case $p>3$).\\

\subsection{An informal conclusion}
The results of this paper and previous works of the second and third authors \cite{DTV,TV13a,TV13b,TV14} can be summarized as follows. 
In the case of integrable models, exact  conservation laws for all Sobolev regularities imply existence of invariant measures; the modified energies we construct in the context of non integrable models imply existence of quasi-invariant measures. Concerning the deterministic behavior of the solutions, exact  conservation laws imply uniform bounds on Sobolev norms of solutions while the modified energies we construct imply polynomial bounds on Sobolev norms of solutions. 
\subsection{Organization of the paper}
In Section~\ref{L6} we get dispersive estimates that play a crucial role in the analysis.  For completeness, the proof of the key nonlinear estimate is presented in the Appendix. In Section \ref{N=infty} we prove Theorem~\ref{basicA} and in Section \ref{sectqi} we prove Theorem~\ref{basicB}.  Section \ref{sectqi} is devoted to the proof of Theorem~\ref{quasiinvariance}. 
\vskip 0.3cm

{\bf Acknowledgement.} The third author is grateful to Yvan Martel for pointing out the reference \cite{MartelAJM} and for interesting discussions about gKdV. The authors also thank referees for many helpful comments and remarks that led to improve the exposition.

\section{Dispersive estimates }\label{L6}

The aim of this section is to collect useful results on the flows associated with 
\eqref{gKdV} and \eqref{gKdVtrunc}. Firstly, global existence and uniqueness of solutions
for the truncated flows follow by a straightforward O.D.E. argument, along with conservation of $L^2$ mass. From now on we assume without further comment 
 existence and uniqueness of global flows $\Phi_M(t)$ for $M\in \N$.
The Cauchy problem associated with \eqref{gKdV} is much more involved. In particular we quote \cite{B1,B2,CKSTT,Staf1} whose analysis implies that for every $s\geq 1$ there exists a unique global solution associated with the initial datum $\varphi\in H^s$; moreover we have continuous dependence on the initial datum. The analysis in \cite{CKSTT} allows to treat local Cauchy theory down to low regularity $H^\frac 12$.
 \\
 
 It will later be important to have a series of uniform bounds with respect to $M$ (in particular suitable $L^6$ bounds), as well as some delicate convergences in suitable topologies of the finite dimensional flows 
to the infinite dimensional one. To the best of our knowledge, those properties do not follow in a straightforward way from the aforementioned works and their proofs require some further arguments. Indeed in our analysis we shall borrow many ideas from references above (in particular \cite{CKSTT}), that in conjunction with new ingredients will imply
several properties for the flows $\Phi_M(t)$ with $M\in \N\cup\{\infty\}$.
\\

In order to provide a precise statement we first define our resolution space.
First recall $X^{s,b}$ spaces, introduced in the fundamental work \cite{B1}.  
For real numbers $s,b$ and a function $u$ on $\R\times\T$, we define the $X^{s,b}$ norm associated with the KdV dispersive relation by
$$
\|u\|_{X^{s,b}}=\|\langle n\rangle^s\langle \tau+n^3\rangle^b \hat{u}(\tau,n)\|_{L^2_{n,\tau}},
$$
where $\hat{u}(\tau,n)$, $\tau\in\R$, $n\in\Z$  is the space-time Fourier transform of $u$.
For $T>0$, we denote $X^{s,b}_{T}$ the restriction space of function on $(-T,T)\times \T$ equipped with the norm 
$$
\|u\|_{X^{s,b}_T}=\inf_{\substack{\tilde u\in X^{s,b}\\\tilde{u}|_{(-T,T)}=u}}  \|\tilde{u}\|_{X^{s,b}}.
$$
In our analysis one needs to take $b=1/2$ and we are led to work in the space $Y^s$ equipped with the norm 
$$
\|u\|_{Y^s}=\|u\|_{X^{s,\frac{1}{2}}}+\|\langle n\rangle^s \hat{u}(\tau,n)\|_{l^2_n L^1_\tau}\,.
$$
One can introduce as above the restriction spaces $Y^s_T$. We recall the embedding
$Y^s_T\subset {\mathcal C}([0,T];H^s)$.

\begin{proposition}\label{Cauchy}
Let $s\geq 1$ and $T>0$ then
\begin{equation}
  \label{exist}
  \forall \varphi \in H^s\,, \quad \exists!\,\, \Phi(t)\varphi\in Y_T^s \text{ solution to } \,\,\, \eqref{gKdV}
\end{equation}
and we  have convergence of the truncated flow, for all compact ${\mathcal K}\subset H^s$,
\begin{equation}\label{vaffpao} 
\sup_{\varphi\in {\mathcal K}} \|\pi_M (\Phi_M(t)\varphi)-\Phi (t)\varphi\|_{L^\infty([0,T];H^s) }\overset{M\rightarrow \infty} \longrightarrow 0\,.
\end{equation}
For every $\varepsilon>0$ and $R>0$ there exists $C>0$ independent of $M\in \N\cup\{\infty\}$ such that:
\begin{align}\label{carb}
\forall \varphi\in H^{s} &\hbox{ s.t. } \|\varphi\|_{H^1}<R\,,\quad \|\pi_M \Phi_M(t) \varphi\|_{H^s}  \leq C\|\varphi\|_{H^s}\,,\\
  \label{unifL6bOU}
 \forall \varphi\in H^{s+\varepsilon} &\hbox{ s.t. } \|\varphi\|_{H^1}<R\,,\quad \|\pi_M (\Phi_M(t)\varphi)\|_{L^6((0,T);W^{s,6})}  \leq C  \|\varphi\|_{H^{s+\epsilon}}\,,\\
  \label{vaffpa}
  \forall \varphi\in H^{s+\varepsilon}&\,,\quad \|\pi_M (\Phi_M(t)\varphi)-\Phi (t)\varphi\|_{L^6((0,T);W^{s,6})}
 \overset{M\rightarrow \infty} \longrightarrow  0\,.
\end{align}
\end{proposition}
We remark that \eqref{vaffpao} (with $\mathcal K=\{\varphi\}$) and \eqref{unifL6bOU} imply \eqref{vaffpa}: using the interpolation inequality, for $\delta>0$,
\begin{equation*}
  \|u\|_{W^{s,6}}\leq C \|u\|_{W^{s+\delta,6}}^{\frac s{s+\delta}} \|u\|_{L^{6}}^{\frac \delta{s+\delta}}\,,
\end{equation*}
we get, by time integration, H\"older inequality in time and Sobolev embedding $H^s\subset L^6$:
\begin{equation}
  \label{GNimpY}
  \|u\|_{L^6 (0,T);W^{s,6})}\leq C \|u\|_{L^6((0,T);W^{s+\delta,6})}^{\frac s{s+\delta}} \|u\|_{L^{\infty}(0,T);H^s)}^{\frac \delta{s+\delta}}\,.
\end{equation}
Next one can choose $u=\pi_M (\Phi_M(t)\varphi)-\Phi (t)\varphi$ and we get \eqref{vaffpa} since the second term on the r.h.s. in \eqref{GNimpY} converges to zero by \eqref{vaffpao} and the 
first term on the r.h.s. is bounded provided we choose $\delta=\frac\varepsilon 2$ and apply \eqref{unifL6bOU}, replacing $s$ by $s+\frac \varepsilon 2$
(there is room to play with $\varepsilon>0$).
\\

Hence we are reduced to proving  \eqref{exist}, \eqref{vaffpao}, \eqref{carb} and \eqref{unifL6bOU}.
The main idea is to perform a gauge transform on gKdV, work on the gauged equation and 
then transfer results back to the original flow. To prove \eqref{vaffpao} it will be of importance to have lemma about continuity of time-dependent translations  for time-dependent functions
(see Lemma~\ref{trans} below).
\subsection{The gauged gKdV equation}
We now present the gauge transform following \cite{CKSTT}.
Set $u_M(t,x)=\pi_M (\Phi_M(t)\varphi) $
and introduce a change of unknown,
$$
v_M(t,x)=u_M\big(t,x+(p+1)\int_{0}^t \int_\T u_M^p dxdt\big)\,,
$$
taming the derivative loss in the nonlinearity. By invariance of the Lebesgue norm by translation,
\begin{equation}\label{utov}
u_M(t,x)=v_M\big(t,x-(p+1)\int_{0}^t \int_\T v_M^p dxdt\big).
\end{equation}
Therefore, $v_M$ is a solution to
$$\begin{cases}
\partial_t v_M - \partial_{x}^3 v_M + \pi_M \partial_x (v_M^{p+1})-(p+1)(\int v_M^p dx) \partial_x v_M= 0, \\
v_M(0,x)=\pi_M \varphi(x).
\end{cases}$$ Let $\Pi$ be the orthogonal projector on the non zero frequencies defined by $\Pi f(x)=f(x)-\int_\T f dx$ for $x\in\T$, and observe that $\pi_M v_M=v_M$. Therefore we can write 
\begin{align*}
\pi_M \partial_x (v_M^{p+1})-(p+1)(\int_\T v_M^p dx) \partial_x v_M
 = &
(p+1) \pi_M ( v_M^p \partial_x v_M ) -(p+1)(\int_\T v_M^p dx) \pi_M \partial_x v_M\\
= & (p+1)\pi_M (   \partial_x v_M \Pi v_M^p )\,.
\end{align*}
Next notice that since $\Pi (\partial_x v_M)=\partial_x v_M$, we can write
\begin{equation*}
\int_\T \partial_x v_M\Pi v_M^p dx=\int_\T  v_M^p  \Pi\partial_x v_M dx=\int_\T v_M^p \partial_x v_M dx=\frac{1}{p+1}\int_\T\partial_x (v_M^{p+1}) dx=0
\end{equation*}
therefore we have $(p+1)\pi_M(\partial_x v_M \Pi v_M^p )=(p+1)\pi_M \Pi (\partial_x v_M \Pi v_M^p)$ and the equation for $v_M$ writes
\begin{equation}\label{vM}
\begin{cases}
\partial_t v_M - \partial_{x}^3 v_M + (p+1)\pi_M\Pi (\partial_x v_M \Pi v_M^p)= 0\\
v_M(0,x)=\pi_M \varphi(x).\end{cases}
\end{equation}
The projector $\Pi$ in the nonlinear term of \eqref{vM} is of fundamental importance because it allows canceling resonant nonlinear interactions in Bourgain spaces.  
Next we denote by $\Phi_M^{\mathcal G} (t)\varphi$ the flow associated with
\begin{equation}\label{vMpiN}
\begin{cases}
\partial_t v - \partial_{x}^3 v + (p+1)\pi_M\Pi ((\partial_x \pi_M v) \Pi (\pi_M v)^p)= 0\\
v(0,x)=\varphi(x)\end{cases}
\end{equation}
 and by $\Phi_\infty^{\mathcal G} (t)\varphi$ 
the flow associated with 
\begin{equation}\label{vMinfty}
\begin{cases}
\partial_t v - \partial_{x}^3 v + (p+1)\Pi (\partial_x v\ \Pi v^p)= 0\\
v(0,x)=\varphi(x).\end{cases}
\end{equation}
We abuse notation, writing $\Phi_\infty^{\mathcal G} (t)=\Phi^{\mathcal G} (t)$. Notice that the solution to \eqref{vM} is provided by $\pi_M(\Phi_M^{\mathcal G} (t) \varphi)$.\\
In order to prove Proposition~\ref{Cauchy}, we work with the flow $\Phi_M^{\mathcal G} (t)$ and then go back to 
the original flow $\Phi_M(t)$.
\begin{proposition}\label{Cauchygauge}
Let $s\geq 1$ and $T>0$, then
\begin{equation}
  \label{existG}
\forall \varphi \in H^s
  \quad \exists!\,\, \Phi^{\mathcal G}(t)\varphi\in Y_T^s \text{ solution to } \quad \eqref{vMinfty}
\end{equation}
For every $R>0$, the map from $H^s\cap \{\varphi\in H^s \hbox{ s.t. } \|\varphi\|_{H^1}<R\}$ to $Y^s_T$
\begin{equation}\label{continuouGauge}
\varphi\rightarrow 
\pi_M \Phi_M^{\mathcal G}(t)\varphi \text{ is uniformly Lipschitz w.r.t.}\, M\in\N\cup\{\infty\}
\end{equation}
and there exists $C>0$, independent of $M\in \N\cup\{\infty\}$ such that
\begin{equation}\label{carbG}
 \forall \varphi\in H^{s} \hbox{ s.t. } \|\varphi\|_{H^1}<R\,,\quad \|\pi_M \Phi_M^{\mathcal G}(t) \varphi\|_{L^\infty([0,T];H^s)}\leq C\|\varphi\|_{H^s}\,.
\end{equation}
We also have the following convergence, for all compact ${\mathcal K}\subset H^s$,
\begin{equation}\label{vaffpaoGauge} 
\sup_{\varphi\in K} \|\pi_M (\Phi^{\mathcal G}_M(t)\varphi)-\Phi^{\mathcal G}(t)\varphi\|_{L^\infty([0,T];H^s) }\overset{M\rightarrow \infty} \longrightarrow 0\,.
\end{equation}
Finally for all $\varepsilon>0$ there exists $C>0$ independent of $M\in \N\cup\{\infty\}$ such that,
\begin{equation}\label{unifL6bOUGauge}
\forall \varphi\in H^{s+\varepsilon} \,\text{ s.t. }\, \|\varphi\|_{H^1}<R\,,\quad \|\pi_M (\Phi^{\mathcal G}_M(t)\varphi)\|_{L^6((0,T);W^{s,6})}\leq C  \|\varphi\|_{H^{s+\epsilon}}\,.
\end{equation}
\end{proposition}
Notice that along the statement of Proposition \ref{Cauchygauge} we claim Lipschitz continuity of the flow. This property will be crucial in order to prove \eqref{vaffpaoGauge}. Once we are done with this proposition,  we can go back to the original flow
$\Phi_M(t)$ by using the following result concerning a time-dependent version of the continuity of the translation operator.
\begin{lemme}\label{trans} Let $s\geq 0$ and ${\mathcal W}\subset {\mathcal C}([0,T]; H^s)$ be compact. Assume that for every $w(t,x)\in {\mathcal W}$ there exists a sequence 
$w_M(t,x)\in {\mathcal C}([0,T];H^s)$ and functions $\tau_M^w, \tau^w \in {\mathcal C}([0,T]; \R)$ 
such that:
\begin{equation}\label{uniformL6G}
\sup_{w\in {\mathcal K}}\|w_M(t,x) - w(t,x)\|_{L^\infty([0,T]; H^s)}\overset{M\rightarrow \infty} \longrightarrow 0
\end{equation}
and
\begin{equation}\label{uniformLinftyL6G}
\sup_{w\in {\mathcal K}}\|\tau_M^w(t) - \tau^w (t)\|_{L^\infty([0,T]; \R)}\overset{M\rightarrow \infty} \longrightarrow 0.
\end{equation}
Then we have
\begin{equation}\label{uniformL6Hs}
\sup_{w\in {\mathcal K}} \|w_M(t,x+\tau_M^w(t)) - w(t,x+\tau^w(t))\|_{L^\infty((0,T); H^{s})}\overset{M\rightarrow \infty} \longrightarrow 0.
\end{equation}
\end{lemme}
\begin{proof}[ Proof of Proposition \ref{Cauchy}] Together Proposition \ref{Cauchygauge} and Lemma \ref{trans} imply Proposition \ref{Cauchy}.
From \eqref{utov},
$$[\pi_M(\Phi_M (t) \varphi)](x)=[\pi_M(\Phi_M^{\mathcal G}(t) \varphi)](x+\tau_M^\varphi(t))$$
where 
\begin{equation}
  \label{paramtrans}
  \tau_M^\varphi(t)=-(p+1)\int_0^t \int_\T [\Phi_M^{\mathcal G}(t) \varphi]^{p+1} dxdt\,,
\end{equation}
hence \eqref{exist} follows by \eqref{existG}.
 As $W^{s,6}$ is translation invariant, \eqref{unifL6bOU}  follows by \eqref{unifL6bOUGauge}, \eqref{carb} follows by \eqref{carbG}.  Moreover, \eqref{vaffpao} follows by \eqref{vaffpaoGauge} in conjunction with
Lemma~\ref{trans}, where we choose ${\mathcal W}=\{\Phi^\mathcal G(t)\varphi\hbox{ s.t. } \varphi\in \mathcal K\}$, and if we denote $w=\Phi^\mathcal G(t)\varphi$  then the corresponding translation parameters $\tau_M^w(t)$ 
that appear in Lemma~\ref{trans} are given by \eqref{paramtrans}.
Notice that the compactness of $\{\Phi^{\mathcal G}(t) \varphi|\varphi\in {\mathcal K}\}$ in $\mathcal C([0,T]; H^s)$, required in Lemma~\ref{trans}, comes as a by-product of continuity of the flow map $\Phi^{\mathcal G}(t)$ (see \eqref{continuouGauge}) along with the embedding $Y^s_T\subset \mathcal C([0,T]; H^s)$.
Uniform convergence of the translation parameters as $M\rightarrow \infty$ follows by \eqref{vaffpaoGauge} in conjunction with Sobolev embedding
$H^s\subset L^{p+1}$.
\end{proof}
\begin{proof}[Proof of Lemma \ref{trans}]
By translation invariance of the $H^s$ norm and by \eqref{uniformL6G}
we get
\begin{equation*}
\sup_{w\in {\mathcal W}} \|w_M(t,x+\tau^w(t)) - w(t,x+\tau^w(t))\|_{L^\infty([0,T]; H^s)}\overset{M\rightarrow \infty} \longrightarrow 0\,
\end{equation*}
hence, it will be enough to prove
\begin{equation}\label{uniformL6tau}
\sup_{w\in {\mathcal W}}\|w_M(t,x+\tau_M^w(t)) - w_M(t,x+\tau^w(t))\|_{L^\infty([0,T]; H^s)}\overset{M\rightarrow \infty} \longrightarrow 0
\end{equation}
in order for \eqref{uniformL6Hs} to hold.
Again by translation invariance of the $H^s$ norm, \eqref{uniformL6tau} is equivalent to 
\begin{equation*}
\sup_{w\in {\mathcal W}} \|w_M(t,x+\tau_M^w(t)-\tau^w(t)) - w_M(t,x)\|_{L^\infty([0,T]; H^s)}\overset{M\rightarrow \infty} \longrightarrow 0.
\end{equation*}
Assume by contradiction there exist $w^k\in {\mathcal W}$ and times $t_{M(k)}^k \in [0,T]$ with $M(k)\overset{k\rightarrow \infty} \longrightarrow \infty$ such that 
\begin{equation*}
\|w_M^k (t_{M(k)}^k,x+\tau_M^{w^k}(t_{M(k)}^k)-\tau^{w_k}(t_{M(k)}^k)) - w_M^k(t_{M(k)}^k ,x)\|_{H^s}
>\epsilon_0>0\,,
\end{equation*}
by \eqref{uniformL6G} we get 
\begin{equation*}\label{uniformL6tauequiv}
\|w^k (t_{M(k)}^k,x+\tau_{M(k)}^{w^k}(t_{M(k)}^k)-\tau^{w^k}(t_{M(k)}^k)) - w^k(t_{M(k)}^k ,x)\|_{H^s}>\frac{\epsilon_0}2>0\,.
\end{equation*}
We claim that this cannot be:  by compactness of ${\mathcal W}$ we can assume that $$w^k\overset{k\rightarrow \infty} \longrightarrow w^*\in {\mathcal W}
\hbox{ in } {\mathcal C}([0,T];H^s)\,,$$ therefore
$$\|w^* (t_{M(k)}^k,x+\tau_{M(k)}^{w^k}(t_{M(k)}^k)-\tau^{w^k}(t_{M(k)}^k)) - w^*(t_{M(k)}^k ,x)\|_{H^s}>\frac{\epsilon_0}4>0.$$
Next, up to a subsequence, $t_{M(k)}^k\overset{k\rightarrow \infty} \longrightarrow t^*$ and by continuity of the function $w^*$ w.r.t. time, we get
\begin{equation}\label{absurtraslnew} \|w^* (t^*,x+\tau_{M(k)}^{w^k}(t_{M(k)}^k)-\tau^{w^k}(t_{M(k)}^k)) - w^*(t^* ,x)\|_{H^s}>\frac{\epsilon_0}8>0.\end{equation}
But by \eqref{uniformLinftyL6G} we have 
$$ |\tau_{M(k)}^{w^k}(t_{M(k)}^k)-\tau^{w_k}(t_{M(k)}^k)|\overset{k\rightarrow \infty}
\longrightarrow 0$$ and \eqref{absurtraslnew} contradicts boundedness of the translation operator in (time independent) $H^s$.\end{proof}
Proposition \ref{Cauchygauge} will follows from the proposition below, which holds for finite time intervals. The defocusing character of our equation is crucial in order to extend to large times these local in time properties.
\begin{proposition}\label{Cauchy_bisG}
Let $s\geq 1$ and $R>0$, then there exists $T>0$ such that:
\begin{equation}\label{YTs}\forall \varphi \in H^s\cap \{\varphi\in H^s \hbox{ s.t. } \|\varphi\|_{H^1}<R\}\,\quad 
  \exists!\,\, \Phi^{\mathcal G}(t)\varphi\in Y_T^s \text{ solution to }\,\,\, \eqref{vMinfty}\,;
\end{equation}
the solution maps from $H^s\cap \{\varphi\in H^s \hbox{ s.t. } \|\varphi\|_{H^1}<R\}$ to $Y^s_T$
\begin{equation}\label{continuouG}
\varphi\rightarrow \pi_M \Phi_M^{\mathcal G}(t)\varphi \text{ are uniformly Lipschitz w.r.t. } M\in\N\cup\{\infty\}\,.
\end{equation}
Moreover there exists a constant $C>0$ independent of $M\in \N\cup\{\infty\}$ such that
\begin{equation}\label{carbGG}
\forall \varphi\in H^{s} \hbox{ s.t. } \|\varphi\|_{H^1}<R\,,\quad \|\pi_M \Phi_M(t) \varphi\|_{L^\infty([0,T];H^s)}\leq C\|\varphi\|_{H^s}\,.
\end{equation}
We also have the following convergence, for all compact $ {\mathcal K}\subset H^s $
\begin{equation}\label{vaffpaoG} 
\sup_{\varphi\in \mathcal K} \|\pi_M (\Phi^{\mathcal G}_M(t)\varphi)-\Phi^{\mathcal G}(t)\varphi\|_{L^\infty([0,T];H^s) }\overset{M\rightarrow \infty} \longrightarrow 0\,,
\end{equation}
and for every $\varepsilon>0$ there exists $C>0$ independent of $M\in \N\cup\{\infty\}$ such that
\begin{equation}\label{unifL6bOUG}
\forall \varphi\in H^{s+\varepsilon} \hbox{ s.t. } \|\varphi\|_{H^1}<R\,,\quad \|\pi_M (\Phi^{\mathcal G}_M(t)\varphi)\|_{L^6((0,T);W^{s,6})}\leq C  \|\varphi\|_{H^{s+\epsilon}} \,.
\end{equation}
\end{proposition}
\begin{proof}[Proof of Proposition \ref{Cauchygauge}] We now prove that Proposition  \ref{Cauchy_bisG} implies Proposition \ref{Cauchygauge}.
The following quantity
\begin{equation}\label{trunchamil}
\frac{1}{2} \int \big (\pi_M (\Phi_M^{\mathcal G}(t)\varphi)\big )^{2}+ \big (\partial_x \pi_M (\Phi_M^{\mathcal G}(t)\varphi)\big )^2 \\+\frac{1}{p+2}\int \big (\pi_M (\Phi_M^{\mathcal G}(t)\varphi)\big )^{p+2}
\end{equation}
is conserved by \eqref{gKdVtrunc} for $M\in \N\cup\{\infty\}$. In particular if we set $M=\infty$ and $R>0$ we get
$$\sup_{\substack{t\in [0, T^{\max}_\varphi)\\\varphi \in H^s
\hbox{ s.t. }  \|\varphi\|_{H^1}<R}} \|\Phi^{\mathcal G}(t)\varphi\|_{H^1}=K(R)<\infty
$$
where $T_\varphi^{\max}$ is the maximal time of existence of a solution in the iteration space $Y_T^s$. The solution $\Phi^{\mathcal G}(t)\varphi$ retains  Sobolev regularity $H^s$ and the corresponding $H^1$ norm stays below $K(R)$ up to $T_\varphi^{max}$.
Hence we have $T_\varphi^{\max}=\infty$ as we can iterate infinitely many times the local existence result
provided by \eqref{YTs} where $R$ is replaced by $K(R)$ and $T$ is the corresponding time of local existence in \eqref{YTs}.
We also have,  by \eqref{trunchamil} 
$$\sup_{\substack{M\in \N\cup\{\infty\}, \hbox{ } t\in [0, \infty)\\\varphi \in H^s
\hbox{ s.t. }  \|\varphi\|_{H^1}<R}} \|\pi_M \Phi_M^{\mathcal G}(t)\varphi\|_{H^1}=\tilde K(R)<\infty
$$
and arguing as above one easily checks that \eqref{continuouGauge} follows by iterations of \eqref{continuouG}, with $R=\tilde K(R)$.
By a similar iteration argument \eqref{unifL6bOUGauge} follows by \eqref{unifL6bOUG} and \eqref{carbG} follows by \eqref{carbGG}. Finally \eqref{vaffpaoG} and \eqref{continuouGauge} 
imply \eqref{vaffpaoGauge}
following a general argument from \cite{sigma}.
\end{proof}

\begin{proof}[Proof of Proposition \ref{Cauchy_bisG}] 
 Denote by $S(t)$ the linear group associated with linear KdV equation, namely $S(t)=e^{t\partial_x^3}$. Then \eqref{vM} rewrites, in integral form,
\begin{equation}\label{fixedpoint}
v_M(t)=S(t)(\pi_M \varphi)+(p+1)\int_{0}^t S(t-\tau) \pi_M\Pi ( \partial_x v_M(\tau) \Pi v_M^p(\tau))d\tau\,.
\end{equation}
The analysis of \cite{CKSTT}, pages 183-186 and pages 197-200  may be used to obtain that for $s\geq 1$,
\begin{equation}\label{tame}
\Big\|\int_{0}^t S(t-\tau) \pi_M\Pi ( \partial_x w(\tau) \Pi w^p(\tau))d\tau\Big\|_{Y^s_T}
\leq CT^{\kappa}\|w\|_{Y^1_T}^p \|w\|_{Y^s_T},
\end{equation}
where $\kappa>0$ and $T\in (0,1)$.  We refer to the appendix for the proof of \eqref{tame}. Notice that \eqref{tame} is a slightly modified version compared with the one available in the literature: we gain a power of $T$, which is very important later. By a similar argument one proves a multi-linear estimate for $s\geq 1$:
\begin{equation}\label{multi_linear}
\Big\|\int_{0}^t S(t-\tau) \pi_M\Pi 
( \partial_x w_{p+1}(\tau)
\Pi (w_1(\tau)\times \dots \times w_p(\tau))
)d\tau\Big\|_{Y^s_T}
\leq CT^{\kappa}
\sum_{i=1}^{p+1}\big( \|w_{i}\|_{Y^s_T} \prod_{\substack{j=1,\dots, p+1\\j\neq i}} \|w_{j}\|_{Y^1_T}
\big)
\end{equation}
and existence and uniqueness follows by a classical fixed point argument in the space $Y_T^s$.

Applying \eqref{tame} with $s=1$, $ w=v_N$ and recalling \eqref{fixedpoint}, we obtain that $\|v_M\|_{Y^1_T}\leq C\|\varphi \|_{H^1}$ provided $T$ is small enough depending only on a bound for $\varphi$ in $H^1$. 
Applying once again \eqref{tame}, we get
$$
\|v_M\|_{Y^s_T}\leq C\|\varphi \|_{H^s}+CT^\kappa (C\|\varphi \|_{H^1})^p \|v_M\|_{Y^s_T}
$$
which implies 
$$
\|v_M\|_{Y^s_T}\leq C\|\varphi \|_{H^s}
$$
by possibly taking $T$ smaller but still depending only on an $H^1$ bound for $\varphi$.   By the embedding $Y_T^s\subset L^\infty([0,T];H^s)$, \eqref{carbGG} follows  and we also get
\begin{equation}\label{BL}
\|v_M\|_{X^{s,\frac{1}{2}}_T}\leq C\|\varphi \|_{H^s}.
\end{equation}
Now we invoke the Strichartz estimate $(8.37)$ of \cite{B1} :
$$
\|S(t)g\|_{L^6((0,T); L^6)}\leq C \|g\|_{H^{\sigma}},\quad \sigma>0
$$
which together with the transfer principle from \cite[Lemma 2.3]{GTV} yields 
\begin{equation}\label{L666}
\|w\|_{L^6((0,T); L^{6}) }\leq C \|w\|_{ X^{\sigma,b}_T}, \quad b>\frac 12.
\end{equation}
Next let $w\in X^{\frac 13, \frac 13}_T$, then we may assume without loss of generality that $w$ is a global space time function such that $ \|w\|_{X^{\frac 13, \frac 13}}\leq 2 \|w\|_{X^{\frac 13, \frac 13}_T}$. By Sobolev embedding $H^\frac 13\subset L^6$ and $S(t)$ being an isometry on $H^s$,
\begin{equation*}
\|w\|_{L^6(\R;L^6(\T))}\leq C \|S(-t) w(t,.)\|_{L^6(\R;H^\frac 13(\T))} 
\leq C  \|\langle D \rangle_x^\frac 13 (S(-t) w(t,.))\|_{L^6(\R;L^2(\T))} 
\end{equation*}
and by Minkowski inequality and Sobolev embedding (that we now exploit w.r.t. the time variable)
\begin{multline*}\dots \leq C \|\langle D \rangle_x^{\frac 13} (S(-t) w(t,.))\|_{L^2(\T; L^6(\R))}
\leq C \|\langle D \rangle_x^{\frac 13} S(-t) w(t,.)\|_{L^2(\T; H^\frac 13 (\R))}\\
=C \| \langle D \rangle_t^{\frac 13}\langle D \rangle_x^\frac 13 (S(-t) w(t,.))\|_{L^2(\R\times\T)} =C\|w\|_{X^{\frac 13,\frac 13}}\leq 2C \|w\|_{X_T^{\frac 13,\frac 13}}
\end{multline*}
so that $\|w\|_{L^6((0,T); L^{6}) }\leq C \|w\|_{X^{\frac 13,\frac 13}_T}$.
Interpolation with \eqref{L666} yields
$$
\forall \,\varepsilon>0\,,\quad \|w\|_{L^6((0,T); L^6)}\leq C \|w\|_{X^{\varepsilon,\frac 12}_T}\,.
$$
By choosing $w=v_M$ and recalling \eqref{BL} where we replace $s$ by $s+\varepsilon$,
$$\|v_M\|_{L^6((0,T); W^{s,6})}\leq C \|v_M\|_{X^{s+\varepsilon,\frac 12}_T}\leq C \|\varphi\|_{H^{s+\varepsilon}}, \quad \forall \varepsilon>0,$$
and we get \eqref{unifL6bOUG}. 
The proof of \eqref{continuouG} follows by \eqref{multi_linear} by considering the difference of two solutions.\\
Finally,
\begin{equation*}
\pi_M\Pi (\partial_x v_M \Pi v_M^p))
-\Pi (\partial_x v \Pi  v^p )
=
\\
\pi_M \Pi
(\partial_x v_M
\Pi(v_M^p-v^p))+ (\partial_x v_M-\partial_x v)
\Pi v^p
)-(1-\pi_M)\Pi ( \partial_x v \Pi v^p)\,,
\end{equation*}
where $v_M, v$ are solutions to \eqref{vM} and \eqref{vMinfty}. Therefore using \eqref{multi_linear}, where we choose $p$ factors $w_i$ equal to either $v_M, v$ and one factor equal to $v-v_M$, writing the fixed point equation solved by $v-v_M$, and recalling \eqref{continuouG},
we get (see e.g. \cite[Proposition~2.7]{sigma} for details), with $\mathcal K$ being a compact in $H^s$,
$$\sup_{\varphi\in {\mathcal K}} \|\pi_M \Phi_M^{\mathcal G}(t)\varphi-\Phi^{\mathcal G} (t)\varphi\|_{Y^s_T}\overset{M\rightarrow \infty} \longrightarrow 0\,.$$
Therefore we get \eqref{vaffpaoG} by using the continuous embedding $Y^s_T\subset L^\infty ([0,T]; H^s)$.
\end{proof}
\section{The energy ${\mathcal E}_k (u)$ and proof of Theorem \ref{basicA}}\label{N=infty}

Recall  $p\in 2\N$ is the integer involved in the nonlinearity in \eqref{gKdV} and the following notations will be used without any further comment.
For any time independent function $u(x)$ we write $\int u=\int_\T u(x) dx$ and similarly for any time dependent function $w(t,x)$
and $T>0$ we write $\int_0^T\int w=\int_0^T\int_\T w(t,x) dx dt$.
\\

We now introduce suitable densities that will be needed to construct the modified energy 
${\mathcal E}_k (u)$. We start with densities that represent the "worse contributions"
that will appear when we compute the time derivative of ${\mathcal E}_k (u)$
along solutions to \eqref{gKdV}. The specific structure of ${\mathcal E}_k (u)$ will allow us to erase the aforementioned "worse" contributions due to
algebraic cancellations.
\begin{definition}
For all $(j_0,j_1,j_2)\in \N^3$, let ${\mathcal I}_{j_0,j_1,j_2}(u)= u^{p-2}  (\partial_x^{j_0} u)^2 \partial_x^{j_1} u  \partial_x^{j_2} u$. For every integer $k>1$ 
we define sets of densities
\begin{equation*}
\Upsilon_k=\{{\mathcal I_{j_0, j_1,j_2}}(u) \hbox{ s.t. } (j_0, j_1, j_2)\in {\mathcal A}^k\}\,\quad \Xi_k=\{{\mathcal I}_{j_0, j_1,j_2}(u) \hbox{ s.t. } (j_0, j_1, j_2)\in {\mathcal B}^k\}
\end{equation*}
where sets of indices ${\mathcal A}^k$, ${\mathcal B}^k$ are defined as
\begin{align*}
  {\mathcal A}^k=\{(j_0,j_1, j_2)\in &\N^3 \hbox{ s.t. } j_2\leq j_1\leq j_0, \quad 2j_0+j_1+j_2=2k-2\}\\
{\mathcal B}^k=\{(j_0,j_1, j_2)\in & \N^3 \hbox{ s.t. } j_2\leq j_1\leq j_0, \quad 2j_0+j_1+j_2=2k+1\}\,.
\end{align*}
\end{definition}
We next introduce "good densities", namely densities that appear once 
we compute the time derivative of ${\mathcal E}_k (u)$
along solutions to \eqref{gKdV}, but are harmless since they can be handled thanks to the dispersive effect of \eqref{gKdV}.
\begin{definition}
Let ${\mathcal J}_{i_0, \dots , i_{m}}(u)=\prod_{l=0}^{m} \partial_x^{i_l} u$. For every { $k>1$} we define 
\begin{align*}
  \Theta_k = & \{{\mathcal J}_{j_0, \dots, j_{2p+1}}(u) \hbox{ s.t. } (j_{0}, \dots, j_{2p+1})\in {\mathcal C}^k\}\\
\Omega_k= & \{{\mathcal J}_{j_0, \dots, j_{p+1}}(u) \hbox{ s.t. } (j_{0}, \dots, j_{p+1})\in {\mathcal D}^k\},
\end{align*}
with 
\begin{align*}{\mathcal C}^k= &\{(j_0,\dots, j_{2p+1})\in \N^{2p+2} \text{ s.t. } j_{2p+1}\leq \dots \leq j_{0}, \quad\sum_{l=0}^{2p+1}j_l\leq 2k-1, \hbox{ } j_2\geq 1, j_0\leq k-1 \}\,\\
{\mathcal D}^k= & \{(j_0,\dots, j_{p+1})\in \N^{p+2} \text{ s.t. } 
j_{p+1}\leq \dots \leq j_{0}, \quad \sum_{l=0}^{p+1}j_l= 2k+1, \quad  j_4\geq 1 \}\,.
\end{align*}
\end{definition}
\begin{remark}
Roughly speaking, densities in $\Theta_k$ are homogeneous of order $2p+2$, the sum of the derivatives on each factor is $2k-1$ and
 derivatives are distributed on at least three factors.
Densities in $\Omega_k$ are homogeneous of order $p+2$, the sum of the derivatives is $2k+1$ and derivatives are shared on at least five factors.
\end{remark}
\begin{definition}\label{startstarstar}
For any $(j_0, j_1, j_2)\in {\mathcal A}^k$
we define
\begin{multline*}{\mathcal I}^*_{j_0, j_1, j_2}(u)=(p-2)  u^{p-3}  \partial_x^3 u (\partial_x^{j_0} u)^2 \partial_x^{j_1} u  \partial_x^{j_2} u
+ 2 u^{p-2}  \partial_x^{j_0} u\partial_x^{j_0+3} u \partial_x^{j_1} u  \partial_x^{j_2} u
\\+ u^{p-2} (\partial_x^{j_0} u)^2 \partial_x^{j_1+3}  u  \partial_x^{j_2} u 
+ u^{p-2} (\partial_x^{j_0} u)^2 \partial_x^{j_1} u  \partial_x^{j_2+3} u
\end{multline*}
and
\begin{multline*}{\mathcal I}^{**}_{j_0, j_1, j_2}(u)=(p-2)  u^{p-3}  \partial_x (u^{p+1}) (\partial_x^{j_0} u)^2 \partial_x^{j_1} u  \partial_x^{j_2} u
\\+ 2 u^{p-2}  \partial_x^{j_0} u\partial_x^{j_0+1} (u^{p+1}) \partial_x^{j_1} u  \partial_x^{j_2} u
+ u^{p-2} (\partial_x^{j_0} u)^2 \partial_x^{j_1+1} (u^{p+1})  u  \partial_x^{j_2} u 
\\+ u^{p-2} (\partial_x^{j_0} u)^2 \partial_x^{j_1} u  \partial_x^{j_2+1} (u^{p+1})\,.
\end{multline*}
\end{definition}
\begin{remark}\label{starstarstar}
The expression ${\mathcal I}^*_{j_0, j_1, j_2}(u)$
is obtained by considering the time derivative of the density ${\mathcal I}_{j_0, j_1, j_2}(u)$
and by replacing $\partial_t u$ with $\partial_x^3 u$.
Similarly ${\mathcal I}^{**}_{j_0, j_1, j_2}(u)$
is constructed 
by considering the time derivative of the density ${\mathcal I}_{j_0, j_1, j_2}(u)$
and by replacing $\partial_t u$ with $\partial_x (u^{p+1})$. Then, if $v(t,x)$ is a solution to \eqref{gKdV},
\begin{equation}\label{basics}\frac d{dt} \int {\mathcal I}_{j_0, j_1, j_2}(v) =\int {\mathcal I}^{*}_{j_0, j_1, j_2}(v) - \int {\mathcal I}^{**}_{j_0, j_1, j_2}(v)\,.
\end{equation}
\end{remark}
The next proposition, on the structure of ${\mathcal E}_k (u)$,  will be key in obtaining Theorem~\ref{basicA}.
{ We shall assume along next proposition that $v(t,x)$ is smooth in order to justify all the necessary computations. Then such smoothness assumption may be removed once we integrate in time \eqref{ddtgkdv} together with a density argument.}
\begin{proposition}\label{key} Let $k>1$. Then
for all $(j_0, j_1, j_2)\in {\mathcal A}^k$ there exists $\lambda_{j_0, j_1, j_2}\in \R$,
for all  $(j_0, \dots , j_{2p+1})\in {\mathcal C}^k$ there exists $\mu_{j_0, \dots, j_{2p+1}}\in \R$
and for all $(j_0, \dots , j_{p+1})\in {\mathcal D}^k$ there exists $\nu_{j_0, \dots, j_{p+1}}\in \R$,
such that, with
\begin{equation}\label{energyEk}
{\mathcal E}_k (u)=\|\partial^{k}_{x}u\|_{L^{2}}^2 - \sum_{(j_0, j_1, j_2)\in {\mathcal A}^k} 
\lambda_{j_0, j_1, j_2} \int {\mathcal I}_{j_0,j_1,j_2}(u)
\end{equation}
the following identity holds for  {  any smooth} solution $v(t,x)$ to \eqref{gKdV}:
\begin{multline}
  \label{ddtgkdv}
  \frac d{dt}  {\mathcal E}_k (v)= \sum_{(j_0, \dots , j_{2p+1})\in {\mathcal C}^k } 
\mu_{j_0, \dots, j_{2p+1}} \int {\mathcal J}_{j_0, \dots , j_{2p+1}}(v)  \\ +
\sum_{(j_0, \dots , j_{p+1})\in {\mathcal D}^k } 
\nu_{j_0, \dots, j_{p+1}} \int {\mathcal J}_{j_0, \dots , j_{p+1}}(v)- \sum_{(j_0, j_1, j_2)\in {\mathcal A}^k} 
\lambda_{j_0, j_1, j_2} \int {\mathcal I}_{j_0,j_1,j_2}^{**}(v)\,.
\end{multline}
\end{proposition}
\begin{remark}
The energy $ {\mathcal E}_k(u)$ has  leading order term $\|\partial^{k}_{x}u\|_{L^2}^2$
plus a lower order remainder that is a linear combination of densities belonging to $\Upsilon_k$, whose coefficients  
$\lambda_{j_0, j_1, j_2}$ are well prepared in such a way that
on the r.h.s. in \eqref{ddtgkdv} we get a linear combination of densities belonging to $\Theta_k$ and $\Omega_k$. In particular, densities ${\mathcal I}_{j_0,j_1,j_2}^{**}(v)$ can be expressed as linear combinations of terms belonging to $\Theta_k$ computed
along $v(t,x)$.
\end{remark}
\begin{proof}[Proof of Theorem \ref{basicA}] We prove how Proposition \ref{key} implies Theorem \ref{basicA}. We first define the quantity
\begin{equation}\label{Rku}{\mathcal R}_k(u)=\sum_{(j_0, j_1, j_2)\in {\mathcal A}^k} 
\lambda_{j_0, j_1, j_2} \int {\mathcal I}_{j_0,j_1,j_2}(u)\end{equation}
and we show \eqref{RestimateA}. 
We have by H\"older and Sobolev embedding:
\begin{align}\label{proofrepeat}
|\int (\partial_x^{j_0} u)^2 \partial_x^{j_1} u \partial_x^{j_2} u u^{p-2}|
 \leq & \|\partial_x^{j_0} u\|_{L^2}^2 \|\partial_x^{j_1} u\|_{L^\infty}  \|\partial_x^{j_2} u\|_{L^\infty} \|u\|^{p-2}_{L^\infty} %
\\\leq  C \|u\|_{H^{j_0}}^2 \|u\|_{H^{j_1+1}}  \| u\|_{H^{j_2+1}}  \|u\|^{p-2}_{H^1}
\leq & C \|u\|_{H^k}^{2\theta_0+\theta_1+\theta_2} \|u\|^{2-2\theta_0-\theta_1-\theta_2+ p}_{H^1}= C \|u\|_{H^k}^{\frac{2k-4}{k-1}}\|u\|^{p+\frac{2}{k-1}}_{H^1}
\end{align}
where we performed interpolation with $\theta_0 k + (1-\theta_0)=j_0$, $\theta_l k + (1-\theta_l)= j_l+1$ for $l=1,2$ and hence $2\theta_0+\theta_1+\theta_2={(2k-4)}/{(k-1)}$.
{ Next we define the functional 
\begin{multline}\label{gknt}
{\mathcal G}_k^T(\varphi) = \sum_{(j_0, \dots , j_{2p+1})\in {\mathcal C}^k } 
\mu_{j_0, \dots, j_{2p+1}} \int_0^T \int {\mathcal J}_{j_0, \dots , j_{2p+1}}(\Phi(\tau)\varphi)  d\tau\\ +
\sum_{(j_0, \dots , j_{p+1})\in {\mathcal D}^k } 
\nu_{j_0, \dots, j_{p+1}} \int_0^T \int {\mathcal J}_{j_0, \dots , j_{p+1}}(\Phi(\tau)\varphi)  d\tau - \sum_{(j_0, j_1, j_2)\in {\mathcal A}^k} 
\lambda_{j_0, j_1, j_2} \int_0^T \int {\mathcal I}_{j_0,j_1,j_2}^{**}(\Phi(\tau)\varphi)  d\tau \end{multline}
and we shall
prove \eqref{GestimateA}. Notice that the identity \eqref{timederA} follows after integration in time of the identity \eqref{ddtgkdv} 
provided that $\varphi$ is regular enough that all computations may be justified.
Then by a straightforward density argument we get the identity \eqref{timederA} for $\varphi\in H^k$ by adapting the proof of \eqref{GestimateA} that we give below.} By looking at the structure of the densities involved
in the definition of \eqref{gknt} we deduce \eqref{GestimateA} provided that
for any $R>0$ and $T>0$ there exists $C>0$ such that for all $\varphi\in H^k$  with $\|\varphi\|_{H^1}<R$, we have
the following bound for $v(t,x)=\Phi(t)\varphi$:
\begin{gather}\label{GestimateDA}
\Big |\int_0^T \int \prod_{l=0}^{p+1} \partial_x^{j_l} v \Big|\leq  C + C\|\varphi\|_{H^{k}}^{(\frac{2k-4}{k-1})^+}, \quad
j_{p+1}\leq \dots \leq j_{0},\,\sum_{l=0}^{p+1}j_l= 2k+1, \, j_4\geq 1\,,\\
\label{GestimateCA}
\Big|\int_0^T \int \prod_{l=0}^{2p+1} \partial_x^{j_l} v \Big|\leq  C 
+ C\|\varphi\|_{H^{k}}^{(\frac{2k-4}{k-1})^+},\quad j_{2p+1}\leq \dots \leq j_{0}, \sum_{l=0}^{2p+1}j_l\leq 2k-1,  j_2\geq 1\,.
\end{gather}
One easily checks that, by integration by parts,  for every $(j_0, j_1, j_2)\in {\mathcal A}_k$,
$\int {\mathcal I}^{**}_{j_0, j_1, j_2}(u)$ can be expressed as the integral of a linear combination of densities belonging to $\Theta_k$: if $(j_0, j_1, j_2)\neq (k-1, 0,0)$ then ${\mathcal I}^{**}_{j_0, j_1, j_2}(u)$ 
is a linear combination of densities with homogeneity $2p+2$ and the top order derivative is at most $k-1$; for $(j_0, j_1, j_2)=(k-1, 0,0)$ we have $${\mathcal I}^{**}_{k-1, 0, 0}(u)=2 \partial_x^{k} (u^{p+1}) \partial_x^{k-1} u u^p
+ p (\partial_x^{k-1} u)^2 \partial_x (u^{p+1}) u^{p-1}\,.$$ 
The second term on the r.h.s. belongs to $\Theta_k$, and the first term on the r.h.s. can be written, by using 
the Leibnitz rule to develop $\partial_x^{k} (u^{p+1})$, as a linear combination of terms belonging to $\Theta_k$ plus a multiple of $\partial_x^{k} u \partial_x^{k-1} u u^{2p}$. We conclude using $\int \partial_x^{k} u \partial_x^{k-1} u u^{2p} 
=\frac 12 \int \partial_x (\partial_x^{k-1} u)^2 u^{2p} =-p \int (\partial_x^{k-1} u)^2 \partial_x u u^{2p-1} $, and the last one is the integral of a density 
in $\Theta_k$.
Hence all terms on the r.h.s. of \eqref{ddtgkdv} can be expressed as linear combinations of terms in $\Theta_k$ and $\Omega_k$
and therefore \eqref{GestimateDA} and \eqref{GestimateCA} are enough to conclude.
\\

Next, we prove \eqref{GestimateDA} and \eqref{GestimateCA} by splitting the proof in four subcases.\\

{\em Proof of \eqref{GestimateDA}, $j_5=0$:} by H\"older inequality we get 
\begin{equation}\label{j5=0}\Big |\int_0^T \int \prod_{l=0}^{p+1} \partial_x^{j_l} v \Big |
\leq  \prod_{l=0}^4 \|\partial_x^{j_l} v\|_{L^6((0,T);L^6)} \|v\|_{L^6((0,T);L^6)} \|v\|_{L^\infty((0,T);L^\infty)}^{p-4}\end{equation}
and by Sobolev embedding, conservation of the Hamiltonian and \eqref{unifL6bOU} we proceed as follows
\begin{equation*}
(\cdots) \leq C \big (\prod_{l=0}^4 \|\varphi\|_{H^{j_l^+}} \big ) \|\varphi\|_{H^1}^{p-3} 
\leq C  \prod_{l=0}^4 ( \|\varphi\|_{H^{k}}^{\theta_l} \|\varphi\|_{H^{1}}^{1-\theta_l})  \|\varphi\|_{H^1}^{p-3}
= C\|\varphi\|_{H^{k}}^{\sum_{l=0}^4\theta_l} \|\varphi\|_{H^{1}}^{p+2-\sum_{l=0}^4 \theta_l} 
\end{equation*}
where 
$\theta_l k + (1-\theta_l)= j_l^+$ for $j=0,\dots, 4$ namely $\theta_l=\frac{j_l^+-1}{k-1}$,
and 
$$\theta_0+\theta_1+\theta_2+\theta_3+\theta_4= (\frac{2k-4 }{k-1})^+.$$
Summarizing we get \eqref{GestimateDA} by using 
the uniform a priori bound on $\|\varphi\|_{H^1}$. 
\\

{\em Proof of \eqref{GestimateDA}, $j_5\neq 0$:} we follow closely the previous argument. By H\"older inequality, Sobolev embedding, \eqref{unifL6bOU} and interpolation, we get
\begin{align}\label{refIHP}
  & \Big |\int_0^T \int \prod_{l=0}^{p+1} \partial_x^{j_l} v \Big |
\leq  \prod_{l=0}^5 \|\partial_x^{j_l} v\|_{L^6((0,T);L^6)}  \prod_{l=6}^{p+1} \|\partial_x^{j_l} v\|_{L^\infty ((0,T);L^\infty)}\\
\leq & C \prod_{l=0}^5 \|\varphi\|_{H^{j_l^+}}  \prod_{l=6}^{p+1} \|v\|_{L^\infty ((0,T); H^{j_l^++1})} \leq C \prod_{l=0}^5 \|\varphi\|_{H^{j_l^+}}  \prod_{l=6}^{p+1} \|\varphi \|_{H^{j_l^++1}}\\
\leq  & C \prod_{l=0}^5 (\|\varphi\|_{H^{1}}^{1-\theta_l} \|\varphi\|_{H^{k}}^{\theta_l})  \prod_{l=6}^{p+1} (\|\varphi\|_{H^{1}}^{1-\eta_l} \|\varphi\|_{H^{k}}^{\eta_l}) 
= C \|\varphi\|_{H^{k}}^{\sum_{l=0}^5 \theta_l +\sum_{l=6}^{p+1} \eta_l}\|\varphi\|_{H^1}^{p+2-\sum_{l=0}^5 \theta_l -\sum_{l=6}^{p+1} \eta_l}
\end{align}
where $\theta_l k + (1-\theta_l)= j_l^+$ for $l=0,\dots, 5$ and $\eta_l k + (1-\eta_l)= j_l^++1$ for $l=6,\dots, p+1$, hence
$$\sum_{l=0}^5 \theta_l +\sum_{l=6}^{p+1} \eta_l=(\frac{2k-5}{k-1})^+.$$
The estimates above along with the uniform bound assumed on $\|\varphi\|_{H^1}$ imply
\begin{equation}\label{tzp}
\Big |\int_0^T \int \prod_{l=0}^{p+1} \partial_x^{j_l} v\Big |
\leq C \|\varphi\|_{H^{k}}^{(\frac{2k-5}{k-1})^+}.\end{equation}
hence we get \eqref{GestimateDA} (indeed, in this subcase we even get better bounds as the corresponding powers are smaller
than the ones that appear in \eqref{GestimateDA}).
\\

{\em Proof of \eqref{GestimateCA}, $j_3=0$: } again, by H\"older inequality we get 
\begin{multline}\label{rmcis}\Big |\int_0^T \int \prod_{l=0}^{2p+1} \partial_x^{j_l} v \Big|
\leq \prod_{l=0}^2 \|\partial_x^{j_l} v\|_{L^4((0,T); L^4)} 
\|v\|_{L^4((0,T); L^4)} \|v\|_{L^\infty((0,T);L^\infty)}^{2p-2}\\
\leq C \prod_{l=0}^2 \|\partial_x^{j_l} v\|_{L^6((0,T); L^6)} 
\|v\|_{L^6((0,T); L^6)} \|v\|_{L^\infty((0,T); L^\infty)}^{2p-2}
\end{multline}
and by Sobolev embedding, conservation of the Hamiltonian and \eqref{unifL6bOU},
\begin{equation*}
  (\dots) \leq C \big( \prod_{l=0}^2 \|\varphi\|_{H^{j_l^+}}\big)  \|\varphi\|_{H^1}^{2p-4} 
\leq C  \prod_{l=0}^2 (\|\varphi\|_{H^{k}}^{\theta_l} \|\varphi\|_{H^{1}}^{1-\theta_l} ) \|\varphi\|_{H^1}^{2p-1}
= C \|\varphi\|_{H^{k}}^{\sum_{l=0}^2\theta_l}\|\varphi\|_{H^{1}}^{2p+2-\sum_{l=0}^2 \theta_l}
\end{equation*}
where 
$\theta_l k + (1-\theta_l)= j_l^+$ for $l=0,1,2$
and hence
$\theta_0+\theta_1+\theta_2= (\frac{2k-4 }{k-1})^+$. Combining above estimates with the uniform bound on $\|\varphi\|_{H^1}$ we get
\begin{equation}\label{ptv3}
\Big |\int_0^T \int \prod_{l=0}^{2p+1} \partial_x^{j_l} v \Big|
\leq C \|\varphi\|_{H^{k}}^{(\frac{2k-4}{k-1})^+}.\end{equation}

{\em Proof of \eqref{GestimateCA}, $j_3\neq 0$: } by H\"older inequality, Sobolev embedding, \eqref{unifL6bOU} and interpolation,
\begin{align}\label{rmsc}
\Big |\int_0^T \int \prod_{l=0}^{2p+1} \partial_x^{j_l} v \Big|
\leq & \prod_{l=0}^3 \|\partial_x^{j_l} v\|_{L^4((0,T);L^4)}  \prod_{l=4}^{2p+1} \|\partial_x^{j_l} v\|_{L^\infty ((0,T);L^\infty)} \\
\leq  & C \prod_{l=0}^3 \|\partial_x^{j_l} v\|_{L^6((0,T);L^6)}  \prod_{l=4}^{2p+1} \|\partial_x^{j_l} v\|_{L^\infty ((0,T);L^\infty)} \\
\leq & C \prod_{l=0}^3 \|\varphi\|_{H^{j_l^+}}  \prod_{l=4}^{2p+1} \|v\|_{L^\infty ((0,T); H^{j_l^++1})} \leq   C \prod_{l=0}^3 \|\varphi\|_{H^{j_l^+}}  \prod_{l=4}^{2p+1} \|\varphi \|_{H^{j_l^++1}}\\
  \leq &  C \prod_{l=0}^3 (\|\varphi\|_{H^{1}}^{1-\theta_l} \|\varphi\|_{H^{k}}^{\theta_l})  \prod_{l=4}^{2p+1} (\|\varphi\|_{H^{1}}^{1-\eta_l} \|\varphi\|_{H^{k}}^{\eta_l})\\
  \leq & C \|\varphi\|_{H^{k}}^{\sum_{l=0}^3 \theta_l +\sum_{l=4}^{2p+1} \eta_l}\|\varphi\|_{H^1}^{2p+2-\sum_{l=0}^5 \theta_l -\sum_{l=6}^{p+1} \eta_l}
\end{align}
where $\theta_l k + (1-\theta_l)= j_l^+$ for $l=0,1,2, 3$ and $\eta_l k + (1-\eta_l)= j_l^++1$ for $l=4,\dots, 2p+1$, then 
$\sum_{l=0}^3 \theta_l +\sum_{l=4}^{2p+1} \eta_l=(\frac{2k-5}{k-1})^+.$ Combining above estimates with the uniform bound on $\|\varphi\|_{H^1}$,
\begin{equation}\label{ptv10} \Big |\int_0^T \int \prod_{l=0}^{2p+1} \partial_x^{j_l} v \Big | \leq C \|\varphi\|_{H^{k}}^{(\frac{2k-5}{k-1})^+}.
\end{equation}
\end{proof}

\subsection{Proof of  Proposition \ref{key} }
We first define an ordering on the sets
${\mathcal A}^k$ and ${\mathcal B}^k$.
\begin{definition}
Given $(i_0, i_1, i_2)$, $(l_0, l_1, l_2)\in {\mathcal A}^k$ or ${\mathcal B}^k$,
we define
\begin{align*}
  (i_0, i_1, i_2)\prec & (l_0, l_1, l_2) \iff \begin{cases} \text{ either } i_0<l_0  \\ \text{ or }
i_0=l_0, i_1<l_1\end{cases}\\
(i_0, i_1, i_2) \preceq & (l_0, l_1, l_2) \iff \begin{cases} \text{ either } (i_0, i_1, i_2)\prec (l_0, l_1, l_2)  \\ \text{ or }
  (i_0, i_1, i_2)=(l_0, l_1, l_2)\end{cases}
  \end{align*}
For any given $(\bar j_0, \bar j_1, \bar j_2)\in {\mathcal A}_k$ we define
$${\mathcal A}_k^{\prec (\bar j_0,  \bar j_1,  \bar j_2)}=\big \{(j_0,  j_1,  j_2)\in {\mathcal A}_k | (j_0,  j_1,  j_2)\prec (\bar j_0,  \bar j_1,  \bar j_2)\big \}$$
as well as ${\mathcal A}_k^{\preceq (\bar j_0,  \bar j_1,  \bar j_2)}$, ${\mathcal A}_k^{\succ (\bar j_0,  \bar j_1,  \bar j_2)}$, ${\mathcal A}_k^{\succeq(\bar j_0,  \bar j_1,  \bar j_2)}$ by modifying the order accordingly.
\end{definition}
The next propositions will be of importance:
\begin{proposition}\label{min}
For every $(j_0, j_1, j_2)\in {\mathcal A}^k$ 
there exists $\delta_{j_0, j_1, j_2}\in \R$ and for every $(j_0, \dots , j_{p+1})\in {\mathcal D}^k $ 
there exists $\gamma_{j_0, \dots, j_{p+1}}\in \R$ such that
\begin{equation*}\label{quadraticprime}
\int \partial_x^k u \partial_x^{k+1} (u^{p+1})= \sum_{(j_0, j_1, j_2)\in {\mathcal A}^k} 
\delta_{j_0, j_1, j_2} \int {\mathcal I}_{j_0+1,j_1+1,j_2}(u)
 +
\sum_{(j_0, \dots , j_{p+1})\in {\mathcal D}^k } 
\gamma_{j_0, \dots, j_{p+1}} \int {\mathcal J}_{j_0, \dots , j_{p+1}}(u).
\end{equation*}
\end{proposition}
\begin{proof}
The map $(j_0, j_1, j_2)\rightarrow (j_0+1, j_1+1, j_2)$ is one to one from $\mathcal A_k$ to $\mathcal B_k$. hence the statement is equivalent to
\begin{equation}\label{quadratic}
\int \partial_x^k u \partial_x^{k+1} (u^{p+1})= \sum_{(j_0, j_1, j_2)\in {\mathcal B}^k} 
\delta_{j_0, j_1, j_2} \int {\mathcal I}_{j_0,j_1,j_2}(u)
 +
\sum_{(j_0, \dots , j_{p+1})\in {\mathcal D}^k } 
\gamma_{j_0, \dots, j_{p+1}} \int {\mathcal J}_{j_0, \dots , j_{p+1}}(u)\,.
\end{equation}
Expanding the $(k+1)$ derivative on the l.h.s. by chain rule, we get integrals of linear combinations of the following densities:
$$\prod_{\substack{j_0\geq j_1\geq \dots \geq j_{p+1}\geq 0\\j_0+\dots +j_{p+1}=2k+1}} \partial_x^{j_l} u.$$
We shall denote by l.o.t. all the densities belonging to $\Omega_k$, namely densities that can be absorbed in the 
second term at the r.h.s. in \eqref{quadratic}. Hence we are reduced to considering only terms like
\begin{equation}\label{termsimp}
\int \partial_x^{j_0} u  \partial_x^{j_1} u \partial_x^{j_2}u \partial_x^{j_3} u u^{p-2}, \quad j_0+j_1+j_2+j_{3}=2k+1,\\
j_0\geq j_1\geq j_2\geq j_3
\end{equation}
and we will prove that, up to l.o.t., they can be expressed as linear combination of terms
$$
\int {\mathcal I}_{i_0,i_1,i_2}(u), \quad (i_0, i_1, i_2)\in {\mathcal B}^k\,.
$$
For every density defined in \eqref{termsimp} we also define $\Delta(j_0, j_1, j_2, j_3)=j_0-j_1$, so that by integration by parts we are done if 
$\Delta(j_0, j_1, j_2, j_3)\in \{0,1\}$. We claim that if $\Delta(j_0, j_1, j_2, j_3)>1$ then,
up to l.o.t., we can express terms in \eqref{termsimp} as linear combination of densities such that $\Delta(j_0, j_1, j_2, j_3)\in \{0, 1\}$.
In order to do so, assume that we have a density like \eqref{termsimp} such that $\Delta(j_0, j_1, j_2, j_3) >1$. Then by integration by parts, more precisely by moving one derivative from the higher order derivative 
$\partial_x^{j_0} u$ on the other factors we get, up to l.o.t., a linear combination of terms 
\begin{multline*}
\int \partial_x^{\tilde j_0} u  \partial_x^{\tilde j_1} u \partial_x^{\tilde j_2}u \partial_x^{\tilde j_3} u u^{p-2}, \quad \tilde j_0+
\tilde j_1+\tilde j_2+\tilde j_{3}=2k+1, \quad \tilde j_0\geq \tilde j_1\geq \tilde j_2\geq \tilde j_3\\
\Delta(\tilde j_0, \tilde j_1, \tilde j_2, \tilde j_3)\in \{\Delta(j_0, j_1, j_2, j_3)-1, \Delta(j_0, j_1, j_2, j_3)-2\}.\end{multline*}
We can further iterate this argument, and conclude after a finite number of steps.
\end{proof}
Our next proposition will be used along the proof of Proposition \ref{key} and we postpone its proof for now.
\begin{proposition}\label{sett}
Let $(j_0, j_1, j_2)\in {\mathcal A}^k$, there exists $\alpha_{j_0, j_1, j_2}\in \R\setminus\{0\}$ such that
\begin{multline}\label{17}
\int {\mathcal I}^*_{j_0, j_1, j_2}(u)= \alpha_{j_0, j_1, j_2} \int {\mathcal I}_{j_0+1, j_1+1, j_2}(u)
+\sum_{\substack{(i_0,i_1 , i_{2})\in {\mathcal B}^k\\ (i_0, i_1, i_2)\prec 
(j_0+1, j_1+1, j_2) }} 
\eta_{i_0, i_1, i_{2}} \int {\mathcal I}_{i_0, i_1, i_{2}}(u) 
\\
+ \sum_{(j_0, \dots , j_{p+1})\in {\mathcal D}^k } 
\varepsilon_{j_0, \dots, j_{p+1}} \int {\mathcal J}_{j_0, \dots , j_{p+1}}(u)
\end{multline}
where $\varepsilon_{j_0, \dots, j_{p+1}}, \eta_{i_0, i_1, i_{2}}\in \R$.
\end{proposition}
\begin{proof}[Proof of Proposition \ref{key}] Propositions \ref{min} and \ref{sett} will imply Proposition \ref{key}. If $v(t,x)$ is solution to \eqref{gKdV} then
\begin{equation*}
\frac d{dt} \|\partial^{k}_{x}v\|_{L^2}^2=2 \int \partial_t \partial_x^k v \partial_x^k v
= 2\int \partial_x^{k+3} v \partial_x^k v - 2 \int \partial_x^{k+1} (v^{p+1}) \partial_x^k v=-2 \int \partial_x^{k+1} (v^{p+1}) \partial_x^k v.
\end{equation*}
hence by Proposition~\ref{min}
\begin{equation}\label{quarda}
\frac d{dt} \|\partial^{k}_{x}v\|_{L^2}^2=  \sum_{(j_0, j_1, j_2)\in {\mathcal A}^k} 
\beta_{j_0, j_1, j_2} \int {\mathcal I}_{j_0+1,j_1+1,j_2}(v) \\
 + \text{l.o.t.}\,,
 \end{equation}
where $\beta_{j_0, j_1, j_2}=-2 \delta_{j_0, j_1, j_2}$ and l.o.t. denotes the integral of a linear combination of densities belonging to $\Theta_k$ and $\Omega_k$.
where $\alpha_{j_0, j_1, j_2}\in \R$. Next we prove hat we can select  the coefficients $\lambda_{j_0, j_1, j_2}$ 
in such a way that
\begin{equation}
  \label{minsett}
  \sum_{(j_0, j_1, j_2)\in {\mathcal A}_k} \beta_{j_0, j_1, j_2}  \int {\mathcal I}_{j_0+1, j_1+1, j_2}(v)
  - \sum_{(j_0, j_1, j_2)\in {\mathcal A}_k} \lambda_{j_0, j_1, j_2} \int {\mathcal I}_{j_0, j_1, j_2}^*(v)=\text{l.o.t}
\end{equation}
where again l.o.t. denotes the integral of a linear combination of densities belonging to $\Theta_k$ and $\Omega_k$.
Of course \eqref{minsett} allows us to conclude the proof thanks to \eqref{basics} and by recalling that terms $\int {\mathcal I}^{**}_{j_0, j_1, j_2}(v) dx$
can be absorbed in l.o.t. (more precisely those densities are a linear combination of terms belonging to $\Theta_k$).

We proceed with \eqref{minsett}: we claim that for every $(j_0, j_1,  j_2)\in {\mathcal A}_k$ we can select a real number $\lambda_{j_0, j_1, j_2}$ such that the following holds: for any given $(\bar j_0,  \bar j_1,  \bar j_2)\in {\mathcal A}_k$ there exists a function
$$\rho^{(\bar j_0,  \bar j_1,  \bar j_2)}: {\mathcal A}_k^{\prec (\bar j_0,  \bar j_1,  \bar j_2)}\ni (j_0,  j_1,  j_2)\rightarrow \rho_{j_0, j_1, j_2}^{(\bar j_0,  \bar j_1,  \bar j_2)}\in \R$$
such that
\begin{multline}\label{part}
\sum_{(j_0, j_1, j_2)\in {\mathcal A}_k^{\succeq (\bar j_0,  \bar j_1,  \bar j_2)}}   \int\big( \beta_{j_0, j_1, j_2} {\mathcal I}_{j_0+1, j_1+1, j_2}(u) - \lambda_{j_0, j_1, j_2} {\mathcal I}_{j_0, j_1, j_2}^*(u)\big)\\=
\sum_{(j_0, j_1, j_2)\in {\mathcal A}_k^{\prec (\bar j_0,  \bar j_1,  \bar j_2)}} \rho_{j_0, j_1, j_2}^{(\bar j_0,  \bar j_1,  \bar j_2)}  \int {\mathcal I}_{j_0+1, j_1+1, j_2}(u)+ \text{l.o.t.}\,.
\end{multline}
If we prove the claim then \eqref{minsett} holds by choosing
$(\bar j_0,  \bar j_1,  \bar j_2)\in {\mathcal A}_k$ such that, for all $(j_0,  j_1,  j_2)\in {\mathcal A}_k$,
$(\bar j_0,  \bar j_1,  \bar j_2)\preceq (j_0,  j_1,  j_2)$.

We first select $\lambda_{(k-1, 0, 0)}$ where $(k-1, 0,0)=\max \big \{(j_0, j_1, j_2)\in {\mathcal A}_k\big \}$; in a second step we define $\lambda_{(j_0', j_1', j_2')}$ where $(j_0', j_1',j_2')=\max \big \{(j_0, j_1, j_2)\in {\mathcal A}_k^{\prec (k-1,0,0)}\big \}$. Then we proceed in a similar way until we get to defining  $\lambda_{(j_0{''}, j_1{''}, j_2{''})}$
where $(j_0{''}, j_1{''},j_2{''})=\min \big \{(j_0, j_1, j_2)\in {\mathcal A}_k\big \}$.
We start by choosing $\lambda_{k-1, 0,0}$.
Notice that by \eqref{17}, if we impose the condition $\lambda_{k-1, 0, 0}\alpha_{k-1, 0, 0}=\beta_{k-1, 0, 0} $, then we get
\begin{equation*}\beta_{k-1, 0, 0}  \int {\mathcal I}_{k, 1, 0}(u)-\lambda_{k-1, 0, 0} \int {\mathcal I}_{k-1, 0, 0}^*(u)
=
\sum_{(j_0, j_1, j_2)\in {\mathcal A}_k^{\prec (k-1,  0,  0)}}  \rho^{(k-1,0,0)}_{j_0, j_1, j_2}  \int {\mathcal I}_{j_0+1, j_1+1, j_2}(u)+\text{l.o.t.}\end{equation*}
Assume that we have selected $\lambda_{j_0, j_1, j_2}$ for every 
$(j_0,  j_1,  j_2)\succeq (\bar j_0, \bar j_1, \bar j_2)$ and assume that \eqref{part} holds. In the next step we select
$\lambda_{\tilde j_0, \tilde j_1, \tilde j_2}$ where $(\tilde j_0, \tilde j_1, \tilde j_2)=\max\big \{(j_0, j_1, j_2)\in {\mathcal A}_k^{\prec (\bar j_0, \bar j_1, \bar j_2)}\big \}$. Then by using the same construction as above we can select
$\lambda_{\tilde j_0, \tilde j_1, \tilde j_2}$, with the condition
$\lambda_{\tilde j_0, \tilde j_1, \tilde j_2}\alpha_{\tilde j_0, \tilde j_1, \tilde j_2}=\rho^{(\bar j_0, \bar j_1, \bar j_2)}_{\tilde j_0, \tilde j_1, \tilde j_2}$. This concludes the proof.
\end{proof}

\begin{proof}[Proof of Proposition \ref{sett}]
To streamline our proof, for a given density $\mathcal D(u)$, we say that
$\int {\mathcal D}(u)\equiv 0$ if and only if ${\mathcal D}(u)$ is
a linear combination of densities in $\Omega_k$ and densities $\mathcal I_{i_0, i_1, i_2}(u)\in \Xi_k$ such that $(i_0, i_1, i_2)\prec 
(j_0+1, j_1+1, j_2)$. If moreover ${\mathcal D}_1(u)$ and ${\mathcal D}_2(u)$ are two densities
then we say $\int {\mathcal D}_1(u)\equiv \int {\mathcal D}_2(u)$ if and only if
$\int ({\mathcal D}_1(u)  - {\mathcal D}_2(u))\equiv 0$. Hence we aim at proving
$\int {\mathcal I}^*_{j_0, j_1, j_2}(u)\equiv \alpha_{j_0, j_1, j_2} \int {\mathcal I}_{j_0+1, j_1+1, j_2}(u)$ where 
$\alpha_{j_0, j_1, j_2} \neq 0$.

We may write 
\begin{equation}\label{ABCDprime}
{\mathcal I}^*_{j_0, j_1, j_2}(u)=2A + B+ C + (p-2) D\end{equation}
with
\begin{equation}\label{ABCD}
  \begin{split}
      A= & \int u^{p-2}  \partial_x^{j_0+3} u \partial_x^{j_0} u \partial_x^{j_1} u  \partial_x^{j_2} u, \quad
B=  \int u^{p-2}  (\partial_x^{j_0} u)^2 \partial_x^{j_1+3} u  \partial_x^{j_2} u ,\\
C= & \int u^{p-2}  (\partial_x^{j_0} u)^2 \partial_x^{j_1} u  \partial_x^{j_2+3} u, \quad
D= \int u^{p-3}  \partial_x^3 u (\partial_x^{j_0} u)^2 \partial_x^{j_1} u  \partial_x^{j_2} u.
  \end{split}
\end{equation}
We split the proof in three cases, each with subcases. Define $\Delta(j_0, j_1, j_2)=j_0-j_1$.

{\bf First case: $\Delta(j_0, j_1, j_2)\geq 2$:} first we prove that
\begin{equation}\label{gui}
B\equiv C\equiv D\equiv 0,\end{equation}
The condition $\Delta(j_0, j_1, j_2)\geq 2$ is equivalent to
$j_1+3\leq j_0+1$, hence we get $B\equiv 0$ as at most one derivative involved in the density of $B$ is of order $j_0+1$ and all the remaining factors have less derivatives.
By a similar argument we get $C\equiv 0$. Regarding $D$,  $\Delta(j_0, j_1, j_2)\geq 2$ implies
$j_0\geq 2$ and hence  $D\equiv 0$ as again at most one derivative involved in the density of $D$ can be of order $j_0+1$ and all the other factors involve less derivatives.

In order to deal with $A$ notice that by integration by parts we can move one derivative from $\partial_x^{j_0+3} u$
to the other factors in the density of $A$:
\begin{multline}\label{firstAprimo}A
=-\int (u^{p-2}  \partial_x^{j_0+2} u \partial_x^{j_0+1} u \partial_x^{j_1} u  \partial_x^{j_2} u 
+ u^{p-2}  \partial_x^{j_0+2} u \partial_x^{j_0} u \partial_x^{j_1+1} u  \partial_x^{j_2} u 
+ u^{p-2}  \partial_x^{j_0+2} u \partial_x^{j_0} u \partial_x^{j_1} u  \partial_x^{j_2+1} u )
\\-(p-2) \int u^{p-3}  \partial_x u \partial_x^{j_0+2} u \partial_x^{j_0} u \partial_x^{j_1} u  \partial_x^{j_2} u 
=(I_A+II_A+III_A)+IV_A.
\end{multline}
As $I_A=-\frac12 \int u^{p-2} \partial_x
(( \partial_x^{j_0+1} u)^2) \partial_x^{j_1} u  \partial_x^{j_2} u$, by integration by part
\begin{multline*}
I_A=\frac 12 \int u^{p-2}  (\partial_x^{j_0+1} u)^2 \partial_x^{j_1+1} u  \partial_x^{j_2} u 
+\frac 12 \int u^{p-2}  (\partial_x^{j_0+1} u)^2 \partial_x^{j_1} u  \partial_x^{j_2+1} u 
\\+\frac{p-2} 2\int u^{p-3}\partial_x u(\partial_x^{j_0+1} u)^2 \partial_x^{j_1} u  \partial_x^{j_2} u .
\end{multline*}
For $II_A$, again by integration by parts we move one derivative from factor $\partial_x^{j_0+2} u$ on the others:
\begin{multline*}
II_A=\int u^{p-2}  (\partial_x^{j_0+1} u)^2 \partial_x^{j_1+1} u  \partial_x^{j_2} u 
+\underbrace{\int u^{p-2}  \partial_x^{j_0+1} u \partial_x^{j_0} u \partial_x^{j_1+2} u  \partial_x^{j_2} u}_{\equiv 0}\\
+\underbrace{ \int u^{p-2}  \partial_x^{j_0+1} u \partial_x^{j_0} u \partial_x^{j_1+1} u  \partial_x^{j_2+1} u}_{\equiv 0}+(p-2)
\underbrace{\int u^{p-3} \partial_x u \partial_x^{j_0+1} u \partial_x^{j_0} u \partial_x^{j_1+1} u  \partial_x^{j_2} u}_{\equiv 0}.
\end{multline*}
Indeed, the second, third and fourth terms on the r.h.s. are equivalent to zero
as, with $\Delta(j_0, j_1, j_2)\geq 2$, in the corresponding densities we have at most one derivative of order $j_0+1$ and all the others are lower order.

 Similarly for the third term $III_A$ we get by integration by parts
\begin{multline*}
III_A=\int u^{p-2}  (\partial_x^{j_0+1} u)^2\partial_x^{j_1} u  \partial_x^{j_2+1} u+\underbrace{\int u^{p-2}  \partial_x^{j_0+1} u \partial_x^{j_0} u \partial_x^{j_1+1} u  \partial_x^{j_2+1} u}_{\equiv 0}\\
+
\underbrace{\int u^{p-2}  \partial_x^{j_0+1} u \partial_x^{j_0} u \partial_x^{j_1} u  \partial_x^{j_2+2} u}_{\equiv 0}
+(p-2) \underbrace{ \int u^{p-3} \partial_x u \partial_x^{j_0+1} u \partial_x^{j_0} u \partial_x^{j_1} u  \partial_x^{j_2+1} u}_{\equiv 0}
\end{multline*}
where similarly to $II_A$ we get that second, third and fourth term on the r.h.s. are equivalent to zero from $\Delta(j_0, j_1, j_2)\geq 2$.

Summarizing we get from \eqref{firstAprimo} the following equivalence:
\begin{multline}\label{icamo}A\equiv \frac 32 \int u^{p-2}  (\partial_x^{j_0+1} u)^2 \partial_x^{j_1+1} u  \partial_x^{j_2} u 
+\frac 32 \int u^{p-2}  (\partial_x^{j_0+1} u)^2 \partial_x^{j_1} u  \partial_x^{j_2+1} u \\+\frac{p-2} 2\int u^{p-3}\partial_x u(\partial_x^{j_0+1} u)^2 \partial_x^{j_1} u  \partial_x^{j_2} u
-(p-2) \int u^{p-3}  \partial_x u \partial_x^{j_0+2} u \partial_x^{j_0} u \partial_x^{j_1} u  \partial_x^{j_2} u.
\end{multline}
Notice that we no integration by parts was performed  on $IV_A$ as its contribution will depend on the value of $j_2$. Next we consider two subcases.

{\em First subcase: $\Delta(j_0, j_1, j_2)\geq 2$, $j_2>0$:} we have $\int u^{p-3}  \partial_x u \partial_x^{j_0+2} u \partial_x^{j_0} u \partial_x^{j_1} u  \partial_x^{j_2} u\equiv 0$ as the density under the integral has five factors with a nontrivial derivative. Similarly we have
$\int u^{p-3}\partial_x u(\partial_x^{j_0+1} u)^2 \partial_x^{j_1} u  \partial_x^{j_2} u\equiv 0$, hence by \eqref{icamo} we get
\begin{equation*}A\equiv \frac 32 \int u^{p-2}  (\partial_x^{j_0+1} u)^2 \partial_x^{j_1+1} u  \partial_x^{j_2} u
\\+\frac 32 \int u^{p-2}  (\partial_x^{j_0+1} u)^2 \partial_x^{j_1} u  \partial_x^{j_2+1} u .
\end{equation*}
Summarizing we obtain that, for $j_1=j_2$,
\begin{equation}\label{ddd}A\equiv 3 \int {\mathcal I}_{j_0+1, j_1+1, j_2}(u)
\end{equation}
while for  $j_1>j_2$,
\begin{equation}\label{aaa}A\equiv \frac 32 \int {\mathcal I}_{j_0+1, j_1+1, j_2}(u)\,.
\end{equation}
In any case by \eqref{ddd}, \eqref{aaa}, \eqref{gui} and \eqref{ABCDprime} we get the desired conclusion in this subcase.

{\em Second subcase: $\Delta(j_0, j_1, j_2)\geq 2$, $j_2=0$:} by \eqref{icamo}, we get
\begin{multline}\label{secondAsec}A\equiv \frac 32 \int u^{p-1}  (\partial_x^{j_0+1} u)^2 \partial_x^{j_1+1} u  
+\frac 32 \int u^{p-2}  (\partial_x^{j_0+1} u)^2 \partial_x^{j_1} u  \partial_x u \\ +\frac{p-2} 2\int u^{p-2}\partial_x u(\partial_x^{j_0+1} u)^2 \partial_x^{j_1} u  
-(p-2) \int u^{p-2}  \partial_x u \partial_x^{j_0+2} u \partial_x^{j_0} u \partial_x^{j_1} u \\
=I_A+II_A+III_A+IV_A.
\end{multline}
For the term $IV_A$, we move one derivative from factor $\partial_x^{j_0+2} u$ on other factors,
\begin{multline*}
IV_A=
(p-2) \int u^{p-2}  \partial_x u (\partial_x^{j_0+1} u)^2 \partial_x^{j_1} u 
+(p-2) \underbrace{ \int u^{p-2}  \partial_x u \partial_x^{j_0+1} u \partial_x^{j_0} u \partial_x^{j_1+1} u }_{\equiv 0}\\
+(p-2) \underbrace{ \int u^{p-2}  \partial_x^2 u \partial_x^{j_0+1} u \partial_x^{j_0} u \partial_x^{j_1} u }_{\equiv 0}
+(p-2)^2 \underbrace{\int u^{p-3}  (\partial_x u)^2 \partial_x^{j_0+1} u \partial_x^{j_0} u \partial_x^{j_1} u }_{\equiv 0}.
\end{multline*}
Again, the second, third and fourth terms on the r.h.s. are equivalent to zero. Indeed, at most one derivative involved in the densities is of order $j_0+1$ and all the other are lower order: from $\Delta(j_0, j_1, j_2)\geq 2$ we get $j_0+1\geq 3$ and that settles the third term on the r.h.s.;  for the fourth term the argument is even easier and for the second term, $\Delta(j_0, j_1, j_2)\geq 2$ implies $j_1+1<j_0+1$.
Hence we get from \eqref{secondAsec}
\begin{equation}\label{icam}A\equiv
\frac 32 \int u^{p-1}  (\partial_x^{j_0+1} u)^2 \partial_x^{j_1+1} u  
+\frac 32 \int u^{p-2}  (\partial_x^{j_0+1} u)^2 \partial_x^{j_1} u  \partial_x u
+\frac 32(p-2) \int u^{p-2}  \partial_x u (\partial_x^{j_0+1} u)^2 \partial_x^{j_1} u\,.
\end{equation}
Notice that if $j_1>0$ then the second and third term on the r.h.s. in \eqref{secondAsec} are equivalent to zero:
 the corresponding densities involve two derivatives of order $j_0+1$ but the remaining derivatives have necessarily order below
$j_1+1$, and we conclude that
\begin{equation}
  \label{eee}A\equiv \frac 32 \int {\mathcal I}_{j_0+1, j_1+1, j_2}(u)\,.
\end{equation}
If $j_1=0$, densities are multiple of $ \int u^{p-1}  (\partial_x^{j_0+1} u)^2 \partial_x u$. By computing the sum of the coefficients we get
\begin{equation}\label{fff}A\equiv \frac 32p \int {\mathcal I}_{j_0+1, j_1+1, j_2}(u).\end{equation}
Combining \eqref{eee}, \eqref{fff}, \eqref{gui} and \eqref{ABCDprime} we get the desired conclusion in this subcase.

{\bf  Second case: $\Delta(j_0, j_1, j_2)=1$:} as $j_0=j_1+1$ we have by integration by parts 
(see \eqref{ABCD} for the definition of $B$) 
\begin{multline*}B=\int u^{p-2}  (\partial_x^{j_0} u)^2 \partial_x^{j_0+2} u  \partial_x^{j_2} u=
-\underbrace{\int u^{p-2}  (\partial_x^{j_0} u)^2 \partial_x^{j_0+1} u  \partial_x^{j_2+1} u}_{\equiv 0}-2\int u^{p-2}  \partial_x^{j_0} u (\partial_x^{j_0+1} u)^2  \partial_x^{j_2} u\\
-(p-2) \underbrace{ \int u^{p-3} \partial_x u  (\partial_x^{j_0} u)^2 \partial_x^{j_0+1} u  \partial_x^{j_2} u}_{\equiv 0}\end{multline*}
where the first and third terms on the r.h.s. are equivalent to zero, as one easily checks that only one derivative involved in the corresponding 
densities has order $j_0+1$ and all the others are lower order.
Then
\begin{equation}
  \label{maccB}
  B\equiv -2\int {\mathcal I}_{j_0+1, j_1+1, j_2}(u)\,.
\end{equation}
Moreover we have (see \eqref{ABCD}) $A=\int u^{p-2}  \partial_x^{j_0+3} u \partial_x^{j_0} u \partial_x^{j_0-1} u  \partial_x^{j_2} u$, and if we move one derivative  by integration by parts, from $\partial_x^{j_0+3} u$ on other factors we get:
\begin{multline}\label{43A} A=-\int u^{p-2} ( \partial_x^{j_0+2} u \partial_x^{j_0+1} u \partial_x^{j_0-1} u  \partial_x^{j_2} u
+ u^{p-2}  \partial_x^{j_0+2} u (\partial_x^{j_0} u)^2  \partial_x^{j_2} u
+ u^{p-2}  \partial_x^{j_0+2} u \partial_x^{j_0} u \partial_x^{j_0-1} u  \partial_x^{j_2+1} u )
\\-(p-2)\int u^{p-3}  \partial_x u \partial_x^{j_0+2} u \partial_x^{j_0} u \partial_x^{j_0-1} u  \partial_x^{j_2} u 
=(I_A+II_A+III_A)+IV_A\,.
\end{multline}
We have $I_A=-\frac 12 \int u^{p-2}  \partial_x ((\partial_x^{j_0+1} u)^2) \partial_x^{j_0-1} u  \partial_x^{j_2} u$
and by integration by parts
\begin{multline*}I_A = \frac 12 \int u^{p-2} (\partial_x^{j_0+1} u)^2 \partial_x^{j_0} u  \partial_x^{j_2} u 
+\frac 12 \int u^{p-2}  (\partial_x^{j_0+1} u)^2 \partial_x^{j_0-1} u  \partial_x^{j_2+1} u 
\\+\frac{p-2}2 \int u^{p-3} \partial_x u (\partial_x^{j_0+1} u)^2 \partial_x^{j_0-1} u  \partial_x^{j_2} u\,.
\end{multline*}
Regarding $II_A$, integration by parts moves one derivative from $\partial_x^{j_0+2} u$ on other factors and we get
\[
II_A=2 \int u^{p-2}  (\partial_x^{j_0+1} u )^2\partial_x^{j_0} u  \partial_x^{j_2} u 
+\underbrace{\int u^{p-2}  \partial_x^{j_0+1} u (\partial_x^{j_0} u)^2  \partial_x^{j_2+1} u }_{\equiv 0}
+(p-2) \underbrace{ \int u^{p-3} \partial_x u \partial_x^{j_0+1} u (\partial_x^{j_0} u)^2  \partial_x^{j_2} u}_{\equiv 0}\,.
\]
Again, the second and third terms  on the r.h.s. are equivalent to zero, as  $\Delta(j_0, j_1, j_2)=1$ implies at most one derivative in the corresponding densities has order $j_0+1$ and all the other are lower order.

Similarly by integration by parts we rewrite $III_A$ as follows
\begin{multline*}
{III_A=\int u^{p-2}  (\partial_x^{j_0+1} u )^2\partial_x^{j_0-1} u  \partial_x^{j_2+1} u 
+\underbrace{ \int u^{p-2}  \partial_x^{j_0+1} u (\partial_x^{j_0} u)^2 \partial_x^{j_2+1} u}_{\equiv 0}}\\
+\int u^{p-2}  \partial_x^{j_0+1} u \partial_x^{j_0} u \partial_x^{j_0-1} u  \partial_x^{j_2+2} u 
+(p-2) \underbrace{ \int u^{p-3} \partial_x u \partial_x^{j_0+1} u \partial_x^{j_0} u \partial_x^{j_0-1} u  \partial_x^{j_2+1} u}_{\equiv 0}.\end{multline*}
Again, on the r.h.s. we have terms equivalent to zero as, exactly as above, the corresponding densities only one derivative has order $j_0+1$ and all the other are lower order.

By integration by parts, moving  one derivative from $\partial_x^{j_0+2} u$ on other factors, we get
\begin{multline*}
IV_A=(p-2)\int u^{p-3}  \partial_x u( \partial_x^{j_0+1} u )^2 \partial_x^{j_0-1} u  \partial_x^{j_2} u
+(p-2) \underbrace{\int u^{p-3}  \partial_x u \partial_x^{j_0+1} u (\partial_x^{j_0} u)^2  \partial_x^{j_2} u}_{\equiv 0}\\
+(p-2)\underbrace{\int u^{p-3} \partial_x u \partial_x^{j_0+1} u \partial_x^{j_0} u \partial_x^{j_0-1} u \partial_x^{j_2+1} u}_{\equiv 0}
+(p-2)\int u^{p-3}  \partial_x^2 u \partial_x^{j_0+1} u \partial_x^{j_0} u \partial_x^{j_0-1} u  \partial_x^{j_2} u 
\\+(p-2)(p-3) \underbrace{ \int u^{p-4} (\partial_x u)^2 \partial_x^{j_0+1} u \partial_x^{j_0} u \partial_x^{j_0-1} u  \partial_x^{j_2} u}_{\equiv 0}\,.
\end{multline*}
On the r.h.s. we get several terms equivalent to zero as, once again,
 corresponding densities have one derivative of order $j_0+1$ and all others are lower order. 
Summarizing we get from \eqref{43A} the following equivalence:
\begin{multline}\label{macc}A\equiv \frac 12 \int u^{p-2} (\partial_x^{j_0+1} u)^2 \partial_x^{j_0} u  \partial_x^{j_2} u 
+\frac 12 \int u^{p-2}  (\partial_x^{j_0+1} u)^2 \partial_x^{j_0-1} u  \partial_x^{j_2+1} u 
\\+\frac{p-2}2 \int u^{p-3} \partial_x u (\partial_x^{j_0+1} u)^2 \partial_x^{j_0-1} u  \partial_x^{j_2} u 
+2 \int u^{p-2}  (\partial_x^{j_0+1} u )^2\partial_x^{j_0} u  \partial_x^{j_2} u 
\\\!\!\!\!+\int u^{p-2}  (\partial_x^{j_0+1} u )^2\partial_x^{j_0-1} u  \partial_x^{j_2+1} u 
+\int u^{p-2}  \partial_x^{j_0+1} u \partial_x^{j_0} u \partial_x^{j_0-1} u  \partial_x^{j_2+2} u\\+(p-2)\int u^{p-3}  \partial_x u( \partial_x^{j_0+1} u )^2 \partial_x^{j_0-1} u  \partial_x^{j_2} u
+(p-2)\int u^{p-3}  \partial_x^2 u \partial_x^{j_0+1} u \partial_x^{j_0} u \partial_x^{j_0-1} u  \partial_x^{j_2} u\,.
\end{multline}
Moreover by \eqref{ABCD} and integration by parts,
\begin{align*}\label{maccD}
  D = & \int u^{p-3}  \partial_x^3 u (\partial_x^{j_0} u)^2 \partial_x^{j_0-1} u  \partial_x^{j_2} u\\
= &  -2\int u^{p-3}  \partial_x^2 u \partial_x^{j_0+1} u  \partial_x^{j_0} u  \partial_x^{j_0-1} u  \partial_x^{j_2} u
-\underbrace{ \int u^{p-3}  \partial_x^2 u (\partial_x^{j_0} u)^3   \partial_x^{j_2} u}_{\equiv 0}\\
& \quad\quad\quad- \underbrace{ \int u^{p-3}  \partial_x^2 u (\partial_x^{j_0} u)^2 \partial_x^{j_0-1} u  \partial_x^{j_2+1} u}_{\equiv 0} -(p-3) \underbrace{ \int u^{p-4} \partial_xu   \partial_x^2 u (\partial_x^{j_0} u)^2 \partial_x^{j_0-1} u  \partial_x^{j_2} u}_{\equiv 0}
  \end{align*}
where three terms on the r.h.s. are equivalent to zero as corresponding densities have at most one derivative of order $j_0+1$ and all the others are lower order. Hence we get
\begin{equation}\label{maccD}
D\equiv -2\int u^{p-3}  \partial_x^2 u \partial_x^{j_0+1} u  \partial_x^{j_0} u  \partial_x^{j_0-1} u  \partial_x^{j_2} u\,.
\end{equation}
Next we split in three subcases.\\
{\em First subcase: $\Delta(j_0, j_1, j_2)=1, j_1>j_2$:} here all terms except the first and the fourth on the r.h.s. in \eqref{macc} are equivalent to zero
and 
\begin{equation}\label{brun}
A\equiv \frac 52 \int {\mathcal I}_{j_0+1, j_1+1, j_2}(u).\end{equation}
In fact, densities involved at the second, third, fifth and seventh terms on the r.h.s. in \eqref{macc}
involve two derivatives of order $j_0+1$ but all the remaining factors involve derivatives of order less than $j_0$.
Notice also that sixth and eighth terms on the r.h.s. in \eqref{macc} are equivalent to zero as, in corresponding densities, only one derivative has order $j_0+1$ and all the remaining are lower order. Moreover, $j_2+3\leq j_0+1$ and \eqref{ABCD} yield
\begin{equation}\label{bruno}C\equiv 0\end{equation}
as at most one derivative in the corresponding density is of order $j_0+1$ and all others have lower order. Notice also that
\begin{equation}\label{brunor}
D\equiv 0\end{equation} from \eqref{maccD},  as necessarily $j_0>1$ hence only one
derivative in the density is of order $j_0+1$ and all others are lower order. We conclude by combining \eqref{brun}, \eqref{bruno}, \eqref{brunor}, \eqref{maccB} and \eqref{ABCDprime}.

{\em Second subcase: $\Delta(j_0, j_1, j_2)=1, j_1=j_2>0$:}  all terms on the r.h.s. in  \eqref{macc} are equivalent to zero except the first, second, fourth, fifth, sixth and hence
\begin{equation}\label{mel}A\equiv 5 \int {\mathcal I}_{j_0+1, j_1+1, j_2}(u) .
\end{equation}
In fact one can check that remaining terms on the r.h.s. in \eqref{macc} are equivalent to zero as either corresponding densities involve at most one derivative
of order $j_0+1$ and others are lower order or two derivatives are of order $j_0+1$ but other derivatives are of order less than $j_0$. Moreover we have from \eqref{ABCD},
\begin{align*}
C= &\int u^{p-2}  (\partial_x^{j_0} u)^2 \partial_x^{j_0-1} u  \partial_x^{j_0+2} u \\
= & -\underbrace{\int u^{p-2}  (\partial_x^{j_0} u)^2 \partial_x^{j_0} u  \partial_x^{j_0+1} u }_{\equiv 0 }
-2\int u^{p-2}  \partial_x^{j_0} u \partial_x^{j_0-1} u ( \partial_x^{j_0+1} u)^2 
    -(p-2)\underbrace{\int u^{p-3} \partial_x u (\partial_x^{j_0} u)^2 \partial_x^{j_0-1} u  \partial_x^{j_0+1} u}_{\equiv 0}
\end{align*}
where we used integration by parts and noticed that the first and third terms on the r.h.s. are equivalent to zero 
as corresponding densities have only one derivative of order $j_0+1$ and others are lower order. Hence,
\begin{equation}\label{mele}C\equiv -2 \int {\mathcal I}_{j_0+1, j_1+1, j_2}(u).\end{equation}
Notice also that by \eqref{maccD} and $j_0>1$, we get
\begin{equation}\label{mele1}D\equiv 0\end{equation}
as in the corresponding density at most one derivative is of order $j_0+1$. We conclude this subcase by combining \eqref{mel}, \eqref{mele}, \eqref{mele1}, \eqref{maccB} and \eqref{ABCDprime}.

{\em Third subcase: $\Delta(j_0, j_1, j_2)=1, j_1=j_2=0$:}  on the r.h.s in \eqref{macc} no term is equivalent to zero, indeed all of them are equivalent to $\int {\mathcal I}_{j_0+1, j_1+1, j_2}(u) dx$,
hence
\begin{equation}\label{petr}
A\equiv \frac 52 p \int {\mathcal I}_{j_0+1, j_1+1, j_2}(u).
\end{equation}
Notice also that
in this subcase by \eqref{ABCD} we get by integration by parts
\begin{equation*}
  C=\int u^{p-1}  (\partial_x u)^2  \partial_x^{3} u\\=-2\int u^{p-1}  \partial_x u  (\partial_x^{2} u)^2-(p-1) \underbrace{\int u^{p-2}
    (\partial_x u)^3  \partial_x^{2} u}_{\equiv 0}
\end{equation*}
where the second term on the r.h.s. is equivalent to zero as only one derivative has order $j_0+1$ and all other factors involve less derivatives. 
Hence we get
\begin{equation}\label{petro}C
\equiv -2  \int {\mathcal I}_{j_0+1, j_1+1, j_2}(u).\end{equation}
Moreover by \eqref{maccD}
we get 
\begin{equation}\label{petron}D\equiv -2\int u^{p-1} ( \partial_x^2 u)^2 \partial_x u  dx=-2\int {\mathcal I}_{j_0+1, j_1+1, j_2}(u).
\end{equation}
We conclude this subcase by combining \eqref{petr}, \eqref{petro}, \eqref{petron}, \eqref{maccB} and \eqref{ABCDprime}.

{\bf Third case: $\Delta(j_0, j_1, j_2)=0$:} first of all, we have by \eqref{ABCD}
$$D=\int u^{p-3}  \partial_x^3 u (\partial_x^{j_0} u)^2 \partial_x^{j_0} u  \partial_x^{j_2} u.$$
For $j_2>0$ we get $D\equiv 0$ as the density on the r.h.s.  involves five factors with non trivial derivatives. For $j_2=0$ we have,
using $2j_0+j_1+j_2=2k-2$, that $j_1$ is an even number and as  $j_{1}=j_0$ we have necessarily $j_0\geq 2$, and hence
$$D=\int u^{p-2}  \partial_x^3 u (\partial_x^{j_0} u)^2 \partial_x^{j_0} u \equiv 0$$
where in the density above at most one derivative has order $j_0+1$ due to the fact that $j_0\geq 2$ and all the other factors have less derivatives.
Summarizing we have 
\begin{equation}\label{gozz}
D\equiv 0.
\end{equation}
Notice also that in this case by \eqref{ABCD} we have 
\begin{equation}\label{aequivb}
A\equiv B
\end{equation}
hence we focus on $A$. By definition of $A$ (see \eqref{ABCD}) and integration by parts we get
\begin{align}\label{poststru}
A= &\int u^{p-2}  \partial_x^{j_0+3} u (\partial_x^{j_0} u)^2  \partial_x^{j_2} u \\
\nonumber = &    \begin{multlined}[t]
   -2 \int u^{p-2}  \partial_x^{j_0+2} u \partial_x^{j_0+1} u \partial_x^{j_0} u  \partial_x^{j_2} u 
-\int u^{p-2}  \partial_x^{j_0+2} u (\partial_x^{j_0} u)^2  \partial_x^{j_2+1} u 
\\- (p-2) \int u^{p-3} \partial_x u \partial_x^{j_0+2} u (\partial_x^{j_0} u)^2  \partial_x^{j_2} u 
=I_A+II_A+III_A.    \end{multlined}
\end{align}
We have
$I_A=-\int u^{p-2}  \partial_x( \partial_x^{j_0+1} u)^2) \partial_x^{j_0} u  \partial_x^{j_2} u $
and by integration by parts,
\[
I_A= \int u^{p-2}  (\partial_x^{j_0+1} u)^2 \partial_x^{j_0+1} u  \partial_x^{j_2} u +\int u^{p-2}  (\partial_x^{j_0+1} u)^2 \partial_x^{j_0} u  \partial_x^{j_2+1} u 
+(p-2)\underbrace{\int u^{p-3}  \partial_x u (\partial_x^{j_0+1} u)^2 \partial_x^{j_0} u  \partial_x^{j_2} u}_{\equiv 0}.
\]
The third term on the right hand side
is equivalent to zero as the density involves two factors with derivatives of order $j_0+1$ and other factors involve less that $j_1+1=j_0+1$ derivatives.
By integration by parts we move one derivative from $\partial_x^{j_0+2} u$ on other factors and  write he second term on the r.h.s in \eqref{poststru}
as 
\begin{multline*}
II_A= 2 \int u^{p-2}  (\partial_x^{j_0+1} u)^2\partial_x^{j_0} u  \partial_x^{j_2+1} u 
+\int u^{p-2}  \partial_x^{j_0+1} u (\partial_x^{j_0} u)^2  \partial_x^{j_2+2} u \\+ (p-2)\underbrace{\int u^{p-3} \partial_x u \partial_x^{j_0+1} u (\partial_x^{j_0} u)^2  \partial_x^{j_2+1} u }_{\equiv 0}
\end{multline*}
where the third term on the r.h.s. is equivalent to zero as the corresponding density involves five factors with nontrivial derivatives.
Similarly by integration by parts we write the third term on the r.h.s in \eqref{poststru} as 
\begin{multline*}
III_A=2 (p-2) \underbrace{\int u^{p-3} \partial_x u (\partial_x^{j_0+1} u)^2 \partial_x^{j_0} u  \partial_x^{j_2} u}_{\equiv 0}+(p-2) \underbrace{\int u^{p-3} \partial_x u \partial_x^{j_0+1} u (\partial_x^{j_0} u)^2  \partial_x^{j_2+1} u}_{\equiv 0}
\\+(p-2) \underbrace{\int u^{p-3} \partial_x^2 u \partial_x^{j_0+1} u (\partial_x^{j_0} u)^2  \partial_x^{j_2} u}_{\equiv 0}
+(p-2)(p-3) \underbrace{\int u^{p-4} (\partial_x u)^2 \partial_x^{j_0+1} u (\partial_x^{j_0} u)^2  \partial_x^{j_2} u}_{\equiv 0}.
\end{multline*}
The first term on the r.h.s. is equivalent to zero as it is equal to the third term on the r.h.s. of $I_A$ above. Concerning the second and fourth terms on the r.h.s., they are equivalent to zero as in the corresponding densities there are at least five factors involving
nontrivial derivatives. Finally, the third term on the r.h.s. is zero as at most two factors may have $j_0+1$ derivatives 
and all other terms have less than $j_0+1=j_1+1$ derivatives. 
Summarizing we get from \eqref{poststru}
\begin{equation}\label{gozzoA}
A\equiv  \int {\mathcal I}_{j_0+1, j_1+1, j_2}(u)
+ 3 \int u^{p-2}  (\partial_x^{j_0+1} u)^2\partial_x^{j_0} u  \partial_x^{j_2+1} u
+\int u^{p-2}  \partial_x^{j_0+1} u (\partial_x^{j_0} u)^2  \partial_x^{j_2+2} u.
\end{equation}
Next we split in two subcases.

{\em First subcase: $\Delta(j_0, j_1, j_2)=0, j_0=j_1=j_2$:} we have from \eqref{ABCD} 
\begin{equation}\label{sargo}A=B=C=\int u^{p-2}  \partial_x^{j_0+3} u \partial_x^{j_0} u \partial_x^{j_0} u  \partial_x^{j_0} u
\end{equation}
and moreover by \eqref{gozzoA} we get
\begin{align*}
A\equiv & \int {\mathcal I}_{j_0+1, j_1+1, j_2}(u) 
+ 3 \int u^{p-2}  (\partial_x^{j_0+1} u)^3\partial_x^{j_0} u +\int u^{p-2}  \partial_x^{j_0+1} u (\partial_x^{j_0} u)^2  \partial_x^{j_0+2} u \\
= & 4 \int {\mathcal I}_{j_0+1, j_1+1, j_2}(u) 
+\frac 12 \int u^{p-2}  (\partial_x^{j_0} u)^2  \partial_x ((\partial_x^{j_0+1} u)^2)\\
= & 4 \int {\mathcal I}_{j_0+1, j_1+1, j_2}(u)  -
\int u^{p-2}  \partial_x^{j_0} u  (\partial_x^{j_0+1} u)^3 - \frac {p-2}2 \underbrace{\int u^{p-3}  \partial_x u (\partial_x^{j_0} u)^2  (\partial_x^{j_0+1} u)^2 }_{\equiv 0}
\end{align*}
where we used again integration by parts and the third term on the r.h.s. is equivalent to zero
as the corresponding density involves five factors with nontrivial derivatives.
Summarizing we get
\begin{equation}\label{gozzoAA}
A\equiv  3 \int {\mathcal I}_{j_0+1, j_1+1, j_2}(u).
\end{equation}
We conclude by \eqref{gozzoAA}, \eqref{gozz}, \eqref{sargo} and \eqref{ABCDprime}.
\\

{\em Second subcase: $j_0=j_1>j_2$:} from
\eqref{gozzoA} we get
\[
A\equiv  \int {\mathcal I}_{j_0+1, j_1+1, j_2}(u)
+ 3 \underbrace{\int u^{p-2}  (\partial_x^{j_0+1} u)^2\partial_x^{j_0} u  \partial_x^{j_2+1} u}_{\equiv 0} 
+\underbrace{\int u^{p-2}  \partial_x^{j_0+1} u (\partial_x^{j_0} u)^2  \partial_x^{j_2+2} u}_{\equiv 0}
\]
where the second term on the r.h.s. is equivalent to zero as the density involves two derivatives of order $j_0+1$ and all other derivatives have order less than $j_0+1$,
while the density of the third term on the r.h.s. involves at most two derivative of order $j_0+1$ and all other derivatives have order at most $j_0$.
Summarizing we get 
\begin{equation}\label{gozzol}
A\equiv  \int {\mathcal I}_{j_0+1, j_1+1, j_2}(u).
\end{equation}
Next, $2j_0+j_1+j_2=2k-2$ implies that $j_1+j_2$ is even and as $j_1>j_2$ we get
\begin{equation}\label{conditmonot}j_2+2\leq j_1=j_0.\end{equation} 
On the other hand by \eqref{ABCD} we have
\begin{equation}\label{ozzC}C=\int u^{p-2}  (\partial_x^{j_0} u)^2 \partial_x^{j_0} u  \partial_x^{j_2+3} u\equiv 0\end{equation}
where on the r.h.s. the density involves at most one derivative of order $j_0+1$ (see \eqref{conditmonot})
and all other factors have less than $j_0$ derivatives. We conclude by \eqref{gozz}, \eqref{aequivb}, \eqref{gozzol}, \eqref{ozzC} and \eqref{ABCDprime}.
\end{proof}

\section{Proof of Theorem \ref{basicB}}\label{sectqil}

{ The proof heavily relies on Proposition \ref{key}. We set $\bar {\mathcal R}_k(u)$ to be the density from the r.h.s. of \eqref{Rku}
and $\bar {\mathcal G}_{k,\infty}^T(\varphi)$ to be the r.h.s. in \eqref{gknt}. First,  we will prove \eqref{RestimateB}. For $M=\infty$, the identity \eqref{timeder} follows from integrating the identity from Proposition \ref{key} and a density argument like the one we used to prove \eqref{timederA}. Second,  we will prove \eqref{Gestimate} in the case $M=\infty$. In a further step,  we will define, for $M\in \N$, the functionals $\bar {\mathcal G}_{k,M}^T(\varphi)$ in such a way that \eqref{timeder} occurs and we will prove the bound
\eqref{Gestimate} for $M\in \N$. Constructing
 $\bar {\mathcal G}_{k,M}^T(\varphi)$ will require a slight modification of $\bar {\mathcal G}_{k,\infty}^T(\varphi)$.
Finally we will prove \eqref{Gestimate} for $M\in \N$ to conclude with the proof of \eqref{Gestimateprime}.}
\subsection{Proof of \eqref{RestimateB}}
The proof of \eqref{RestimateB} follows the same ideas used along the proof of 
\eqref{RestimateA}. Recall that the structure of $\bar {\mathcal R}_k(u)$
is provided in \eqref{Rku} and hence by repeating the same computations as in
\eqref{proofrepeat}, except a modification of the interpolation estimates at the last step, we get:
\begin{equation*}|\int (\partial_x^{j_0} u)^2 \partial_x^{j_1} u \partial_x^{j_2} u u^{p-2}|
\leq  C \|u\|_{H^{j_0}}^2 \|u\|_{H^{j_1+1}}  \| u\|_{H^{j_2+1}}  \|u\|^{p-2}_{H^1}\leq 
C \|u\|_{H^{(k-\frac 12)^-}}^{2\theta_0+\theta_1+\theta_2} \|u\|^{2-2\theta_0-\theta_1-\theta_2+ p}_{H^1}
\end{equation*}
where $\theta_0 (k-\frac 12)^- + (1-\theta_0)=j_0$, $\theta_l (k-\frac 12)^- + (1-\theta_l)= j_l+1$ for $l=1,2$. We conclude \eqref{RestimateB}
since we have
$$
2\theta_0+\theta_1+\theta_2=(\frac{4k-8}{2k-3})^+\,.
$$
\subsection{Proof of \eqref{Gestimate} in the case $M=\infty$} Since the expression $\bar {\mathcal G}_{k,\infty}^T(\varphi)$ is defined by the r.h.s. in \eqref{gknt}
we can adapt the proof of \eqref{GestimateA} in order to get \eqref{Gestimate} in the case $M=\infty$.
In fact it is sufficient to prove the following version of \eqref{GestimateDA}
and \eqref{GestimateCA}
for any $R,T>0$ and for all $\varphi\in H^k$  such that $\|\varphi\|_{H^1}<R$:
\begin{gather}\label{GestimateDB}
\Big |\int_0^T \int \prod_{l=0}^{p+1} \partial_x^{j_l} v \Big|\leq  C + C\|\varphi\|_{H^{(k-\frac 12)^-}}^{(\frac{4k-8}{2k-3})^+}, \quad
j_{p+1}\leq \dots \leq j_{0},\,\sum_{l=0}^{p+1}j_l= 2k+1, \, j_4\geq 1\,,\\
\label{GestimateCB}
\Big|\int_0^T \int \prod_{l=0}^{2p+1} \partial_x^{j_l} v \Big|\leq  C 
+ C\|\varphi\|_{H^{(k-\frac 12)^-}}^{(\frac{4k-8}{2k-3})^+},\quad j_{2p+1}\leq \dots \leq j_{0}, \sum_{l=0}^{2p+1}j_l\leq 2k-1,  j_2\geq 1\,.
\end{gather}
The proof of \eqref{GestimateDB} follows, exactly as for the proof of \eqref{GestimateDA}, by considering two subcases: $j_5=0$ and $j_5\neq 0$.
In the first case we start from the estimate 
\eqref{j5=0} and the subsequent computation which imply
\begin{equation*}\Big |\int_0^T \int \prod_{l=0}^{p+1} \partial_x^{j_l} v \Big |
\leq C \big (\prod_{l=0}^4 \|\varphi\|_{H^{j_l^+}} \big ) \|\varphi\|_{H^1}^{p-3} 
\end{equation*}
and hence by interpolation we can continue the estimate as follows
$$(\cdots) \leq C  \big (\prod_{l=0}^4 \|\varphi\|_{H^{(k-\frac 12)^-}}^{\sum_{l=0}^4\theta_l}\big ) \|\varphi\|_{H^{1}}^{p+2-\sum_{l=0}^4 \theta_l} $$
where 
$\theta_l (k-\frac 12)^- + (1-\theta_l)= j_l^+$ for $l=0, \dots, 4$ namely 
$\theta_l=\frac{j_l^+-1}{k-\frac 32}$ and
$$
\theta_0+\theta_1+\theta_2+\theta_3+\theta_4= (\frac{4k-8 }{2k-3})^+.
$$
Hence we get \eqref{GestimateDB} in the case $j_5=0$.
In order to establish \eqref{GestimateDB} in the case $j_5\neq 0$ we argue as in the proof
of \eqref{GestimateDA} for $j_5\neq 0$ and hence following \eqref{refIHP} we get
\begin{equation*}
  \Big |\int_0^T \int \prod_{l=0}^{p+1} \partial_x^{j_l} v \Big |
\leq C \prod_{l=0}^5 \|\varphi\|_{H^{j_l^+}}  \prod_{l=6}^{p+1} \|v\|_{L^\infty ((0,T); H^{j_l^++1})} \leq C \prod_{l=0}^5 \|\varphi\|_{H^{j_l^+}}  \prod_{l=6}^{p+1} \|\varphi \|_{H^{j_l^++1}}
\end{equation*}
and we can continue the estimate as follows by interpolation
\begin{equation}
(\cdots)\leq  C \prod_{l=0}^5 (\|\varphi\|_{H^{1}}^{1-\theta_l} \|\varphi\|_{H^{(k-\frac12)^{-}}}^{\theta_l})  \prod_{l=6}^{p+1} (\|\varphi\|_{H^{1}}^{1-\eta_l} \|\varphi\|_{H^{(k-\frac 12)^-}}^{\eta_l}) 
= C \|\varphi\|_{H^{(k-\frac 12)^-}}^{\sum_{l=0}^5 \theta_l +\sum_{l=6}^{p+1} \eta_l}\|\varphi\|_{H^1}^{p+2-\sum_{l=0}^5 \theta_l -\sum_{l=6}^{p+1} \eta_l}
\end{equation} 
where
$\theta_l (k-\frac 12)^- + (1-\theta_l)= j_l^+$ for $l=0, \dots, 5$, $\eta_l (k-\frac 12)^- + (1-\eta_l)= j_l^++1$ for $l=6,\dots, p+1$. We conclude
\eqref{GestimateDB} for $j_5\neq 0$
since 
$$\sum_{l=0}^5 \theta_l +\sum_{l=6}^{p+1} \eta_l=(\frac{4k-10 }{2k-3})^+$$ 
(notice that in fact we get in this case a stronger version of \eqref{GestimateDB} since we get on the r.h.s. an exponent even smaller than
$(\frac{4k-8}{2k-3})$).

In order to establish \eqref{GestimateCB} we argue as in the proof of \eqref{GestimateCA}
and hence we split in the following subcases: $j_3=0$ and $j_3\neq 0$.
In the case $j_3=0$ we can argue as in \eqref{rmcis} and the subsequent computations in order to get
\begin{equation*}
\Big |\int_0^T \int \prod_{l=0}^{2p+1} \partial_x^{j_l} v \Big|
\leq C \big( \prod_{l=0}^2 \|\varphi\|_{H^{j_l^+}}\big)  \|\varphi\|_{H^1}^{2p-4} 
\end{equation*}
and hence by interpolation we can continue the estimate as follows
\begin{equation*}(\cdots)\leq C  \prod_{l=0}^2 (\|\varphi\|_{H^{(k-\frac 12)^-}}^{\theta_l} \|\varphi\|_{H^{1}}^{1-\theta_l} ) \|\varphi\|_{H^1}^{2p-1}
= C \|\varphi\|_{H^{(k-\frac 12)^-}}^{\sum_{l=0}^2\theta_l}\|\varphi\|_{H^{1}}^{2p+2-\sum_{l=0}^2 \theta_l}
\end{equation*}
where $\theta_l (k-\frac 12)^- + (1-\theta_l)= j_l^+$ for $l=0,1,2$. We get \eqref{GestimateCB} in the case $j_3=0$ since
$\theta_0+\theta_1+\theta_2=(\frac{4k-8 }{2k-3})^+.$
In order to prove \eqref{GestimateCB} in the case $j_3\neq 0$ we can argue as in \eqref{rmsc} and we get
\begin{equation*}
\Big |\int_0^T \int \prod_{l=0}^{2p+1} \partial_x^{j_l} v \Big|
\leq   C \prod_{l=0}^3 \|\varphi\|_{H^{j_l^+}}  \prod_{l=4}^{2p+1} \|\varphi \|_{H^{j_l^++1}}
\end{equation*}
and we can continue by interpolation as follows
\begin{align*}(\cdots) \leq &  C \prod_{l=0}^3 (\|\varphi\|_{H^{1}}^{1-\theta_l} \|\varphi\|_{H^{(k-\frac 12)^-}}^{\theta_l})  \prod_{l=4}^{2p+1} (\|\varphi\|_{H^{1}}^{1-\eta_l} \|\varphi\|_{H^{(k-\frac 12)^-}}^{\eta_l})\\
  \leq & C \|\varphi\|_{H^{(k-\frac 12)^-}}^{\sum_{l=0}^3 \theta_l +\sum_{l=4}^{2p+1} \eta_l}\|\varphi\|_{H^1}^{2p+2-\sum_{l=0}^5 \theta_l -\sum_{l=6}^{p+1} \eta_l}
\end{align*}
where $\theta_l (k-\frac 12)^- + (1-\theta_l)= j_l^+$ for $l=0,1,2, 3$ and $\eta_l (k-\frac 12)^- + (1-\eta_l)= j_l^++1$ for $l=4,\dots, 2p+1$, then 
$\sum_{l=0}^3 \theta_l +\sum_{l=4}^{2p+1} \eta_l=(\frac{4k-10}{2k-3})^+.$ Then we conclude \eqref{GestimateCB} 
in the case $j_3\neq 0$ (in fact we get an even strong version since we have $(\frac{4k-10}{2k-3})^+<(\frac{4k-8}{2k-3})$.
\subsection{Definition of $\bar {\mathcal G}^T_{k,M}$ for $M\in \N$ and proof of \eqref{Gestimate} with constant uniform w.r.t. $M$}
We first introduce the energies $\bar {\mathcal G}_{k, M}^T$ in such a way that
\eqref{timeder} occurs, and we shall also establish \eqref{Gestimate} for $M\in \N$. We set $v_M(t,x)=\pi_M (v(t,x))$ where $v(t,x)$ is the unique solution to \eqref{gKdVtrunc} for a given $M\in \N$. We retain notations from Section~\ref{N=infty}, except we need to alter ${\mathcal I}^{**}_{j_0, j_1, j_2}(u)$
as follows:
\begin{definition}
For any $(j_0, j_1, j_2)\in {\mathcal A}^k$
and for every $M\in \N$
we define
\begin{multline*}{\mathcal I}^{**}_{j_0, j_1, j_2,M}(u)=(p-2)  u^{p-3}   \pi_M(\partial_x (u^{p+1})) (\partial_x^{j_0} u)^2 \partial_x^{j_1} u  \partial_x^{j_2} u
+ 2 u^{p-2}  \partial_x^{j_0} u  \pi_M(\partial_x^{j_0+1} (u^{p+1})) \partial_x^{j_1} u  \partial_x^{j_2} u
\\+ u^{p-2} (\partial_x^{j_0} u)^2  \pi_M(\partial_x^{j_1+1} (u^{p+1}))  u  \partial_x^{j_2} u 
+ u^{p-2} (\partial_x^{j_0} u)^2 \partial_x^{j_1} u   \pi_M(\partial_x^{j_2+1} (u^{p+1}))
\end{multline*}
\end{definition}
The main difference between  ${\mathcal I}^{**}_{j_0, j_1, j_2}(u)$ and ${\mathcal I}^{**}_{j_0, j_1, j_2,M}(u)$ is the projector $\pi_M$.
The reason of this modification is that we have the following identity (compare with \eqref{basics})
\begin{equation}\label{basicsN}\frac d{dt} \int {\mathcal I}_{j_0, j_1, j_2}(v_M)=\int {\mathcal I}^{*}_{j_0, j_1, j_2}(v_M) - \int {\mathcal I}^{**}_{j_0, j_1, j_2,M}(v_M)
\end{equation}
where $v_M(t,x)=\pi_M (v(t,x))$ and $v(t,x)$ is a solution to \eqref{gKdVtrunc}.
Next we notice that we have the following version of \eqref{quarda} where the solution $v(t,x)$ to \eqref{gKdV}
is replaced by $v_M(t,x)$:
\begin{equation}\label{quardaM}
\frac d{dt} \|\partial^{k}_{x}v_M\|_{L^2}^2=  \sum_{(j_0, j_1, j_2)\in {\mathcal A}^k} 
\beta_{j_0, j_1, j_2} \int {\mathcal I}_{j_0+1,j_1+1,j_2}(v_M)\\
 + \text{l.o.t.}
 \end{equation}
where l.o.t. denote the integral of a linear combination of densities belonging to $\Theta_k$ and $\Omega_k$
computed along the solution $v_M(t,x)$.
In fact we have
\[
\frac d{dt} \|\partial^{k}_{x}v_M\|_{L^2}^2=2 \int \partial_t \partial_x^k v_M \partial_x^k v_M
= 2\int \partial_x^{k+3} v_M \partial_x^k v_M - 2 \int \partial_x^{k+1} \pi_M (v_M^{p+1}) \partial_x^k v_M=-2 \int \partial_x^{k+1} (v_M^{p+1}) \partial_x^k v_M
\]
where we used that $\pi_M$ is symmetric and $\pi_M (v_M)=v_M$.
Hence the proof  of \eqref{quardaM} follows that of \eqref{quarda}. Next recall that by \eqref{minsett}, which is available for any generic function $u(x)$, we get
\[
  \sum_{(j_0, j_1, j_2)\in {\mathcal A}_k} \beta_{j_0, j_1, j_2}  \int {\mathcal I}_{j_0+1, j_1+1, j_2}(v_M) - 
\sum_{(j_0, j_1, j_2)\in {\mathcal A}_k} \lambda_{j_0, j_1, j_2} \int {\mathcal I}_{j_0, j_1, j_2}^*(v_M)=\text{l.o.t.}
\]
where l.o.t. denotes the integral of a linear combination of densities belonging to $\Theta_k$ and $\Omega_k$.
By combining this identity with \eqref{quardaM} and \eqref{basicsN}
we get (by defining the energy $\bar {\mathcal E}_k$ as in \eqref{energyEk})
\begin{multline}\label{adaptversion}
\frac{d}{dt} \bar {\mathcal E}_k(v_M)
= \sum_{(j_0, \dots , j_{2p+1})\in {\mathcal C}^k } 
\mu_{j_0, \dots, j_{2p+1}} \int {\mathcal J}_{j_0, \dots , j_{2p+1}}(v_M) \\ +
\sum_{(j_0, \dots , j_{p+1})\in {\mathcal D}^k } 
\nu_{j_0, \dots, j_{p+1}} \int {\mathcal J}_{j_0, \dots , j_{p+1}}(v_M)- \sum_{(j_0, j_1, j_2)\in {\mathcal A}^k} 
\lambda_{j_0, j_1, j_2} \int {\mathcal I}_{j_0,j_1,j_2,M}^{**}(v_M) 
\end{multline}
where on the r.h.s. we have the integral of a linear combination of densities belonging to $\Theta_k$ and $\Omega_k$,
computed along the solution $v_M(t,x)$ plus the integral of a linear combination of the expressions 
${\mathcal I}_{j_0,j_1,j_2,M}^{**}(v_M(t,x))$ where $(j_0, j_1, j_2)\in {\mathcal A}_k$.
Notice that we cannot expand  $\int {\mathcal I}_{j_0,j_1,j_2,M}^{**}(v_M(t,x)) dx$
as the integral of linear combination of terms belonging to $\Theta_k$ computed along $v_M(t,x)$  because of $\pi_M$ appearing in ${\mathcal I}_{j_0,j_1,j_2,M}^{**}(u)$.\\
Based on the r.h.s. in \eqref{adaptversion} we define 
\begin{multline}\label{mend}
\bar {\mathcal G}_{k,M}^T(\varphi)=
\sum_{(j_0, \dots , j_{2p+1})\in {\mathcal C}^k } 
\mu_{j_0, \dots, j_{2p+1}} \int_0^T\int {\mathcal J}_{j_0, \dots , j_{2p+1}}(\pi_M(\Phi_M(t)\varphi) \\ +
\sum_{(j_0, \dots , j_{p+1})\in {\mathcal D}^k } 
\nu_{j_0, \dots, j_{p+1}} \int_0^T \int {\mathcal J}_{j_0, \dots , j_{p+1}}(\pi_M(\Phi_M(t)\varphi) \\- \sum_{(j_0, j_1, j_2)\in {\mathcal A}^k} 
\lambda_{j_0, j_1, j_2} \int_0^T \int {\mathcal I}_{j_0,j_1,j_2,M}^{**}(\pi_M(\Phi_M(t)\varphi)
\end{multline}
and we prove \eqref{Gestimate} for $M\in \N$ with constant uniform w.r.t. $M$. On the r.h.s. of \eqref{mend} we have three type of densities.
The first two groups of densities on the r.h.s. computed along $v_M(t,x)$ can be estimated by using \eqref{GestimateDB} and \eqref{GestimateCB} where $v(t,x)$ is  replaced by $v_M(t,x)$. Indeed the proof can be done {\em mutatis mutandis} by replacing $v(t,x)$ with $v_M(t,x)$.\\
Hence we only need to estimate the third group of densities on the r.h.s. in \eqref{mend} as follows
\begin{equation}\label{lecmer}|\int_0^T \int {\mathcal I}_{j_0,j_1,j_2,M}^{**}(v_M(t,x))|\leq C 
+ C\|\varphi\|_{H^{(k-\frac 12)^-}}^{(\frac{4k-8}{2k-3})^+}, \quad (j_0, j_1, j_2)\in {\mathcal A}^k
\end{equation} with a constant uniform w.r.t. $M$, in order to conclude \eqref{Gestimate}.
Notice that
${\mathcal I}_{j_0,j_1,j_2,M}^{**}(v_M)$ with $(j_0, j_1, j_2)\in {\mathcal A}^k$ can be expressed as a linear combination
of densities belonging to $\Theta_k$, except that the projector $\pi_M$ can appear in front of a group of factors.
The proof of this fact is obvious if $(j_0, j_1, j_2)\neq (k-1,0,0)$. Assuming $(j_0, j_1, j_2)= (k-1,0,0)$, we have 
\begin{equation}\label{IMsta}\int {\mathcal I}_{k-1,0,0,M}^{**}(v_M)
=2\int \partial_x^{k}  \pi_M (v_M^{p+1})   \partial_x^{k-1} v_M  v_M^p
+ p \int 
 (\partial_x^{k-1} v_M)^2 v_M^{p-1} \pi_M \partial_x (v_M^{p+1}) \,.
 \end{equation}
 The density of the second integral on the r.h.s. belongs to $\Theta_k$, with factors $v_M$ up to the projector $\pi_M$ .
 For the first term on the r.h.s.,  recall that by Leibnitz rule we have
$$(p+1) \partial_x^{k-1} u u^p=\partial_x^{k-1} (u^{p+1}) + \sum_{\substack{(\alpha_1,\dots, \alpha_{p+1})\\\alpha_1+\dots +\alpha_{p+1}=k-1\\\alpha_j<k-1}}
c_{\alpha_1,\dots, \alpha_{p+1}}\prod_{j=1}^{p+1} \partial_x^{\alpha_j} u$$
hence the first term on the r.h.s. in \eqref{IMsta} can be expressed as a multiple of
\begin{multline}\label{Lenvis}\int \partial_x^{k} \pi_M (v_M^{p+1})    \partial_x^{k-1} (v_M^{p+1}) 
+ \sum_{\substack{(\alpha_1,\dots, \alpha_{p+1})\\\alpha_1+\dots +\alpha_{p+1}=k-1\\\alpha_j<k-1}}
c_{\alpha_1, \dots, \alpha_{p+1}} \int
\partial_x^{k} \pi_M (v_M^{p+1})    \prod_{j=1}^{p+1} \partial_x^{\alpha_j} v_M\\
= \int \partial_x^{k} \pi_M (v_M^{p+1})    \partial_x^{k-1} \pi_M (v_M^{p+1}) 
- \sum_{\substack{(\alpha_1,\dots, \alpha_{p+1})\\\alpha_1+\dots +\alpha_{p+1}=k-1\\\alpha_j<k-1}}
c_{\alpha_1, \dots, \alpha_{p+1}} \int
\partial_x^{k-1} \pi_M (v_M^{p+1})    \partial_x\Big (\prod_{j=1}^{p+1} \partial_x^{\alpha_j} v_M\Big ) \end{multline}
where we used that $\pi_M$ is a projector. Now, the first term on the r.h.s. in \eqref{Lenvis} is zero, as an integral of an exact derivative and the other terms, after expanding the 
derivative of a product, can be expressed as the integral of densities belonging to $\Theta_k$ up to the projector $\pi_M$.
Summarizing in order to get \eqref{lecmer} it is sufficient to estimate
\begin{equation}\label{prodproj}\Big |\int_0^T \int \Big( \prod_{j=0}^m \partial_x^{\gamma_j} v_M\Big)  \pi_M \Big(\prod_{i=m+1}^{2p+1} \partial_x^{\beta_i} v_M\Big )\Big |
\leq  C 
+ C\|\varphi\|_{H^{(k-\frac 12)^-}}^{(\frac{4k-8}{2k-3})^+}\end{equation}
where $\sum_{j=0}^m \gamma_j+ \sum_{i=m+1}^{2p+1} \beta_i=2k-1$ and at least three numbers in $\{\gamma_j, \beta_i\}$ are non-zero. We can assume the following alternative:
the four factors with higher derivatives belong to the first product outside the projector, the four factors with higher derivatives belong to the second product on which
we have the action of the projector or the four factors with higher derivatives are distributed between the two groups.
Hence, depending on the structure of the density appearing on the the l.h.s. in \eqref{prodproj} and by assuming that the numbers $\gamma_j$ are decreasing
as well as the numbers $\beta_i$, we can associate with the density on the l.h.s. in \eqref{prodproj} a number $J\in \{0,1,2,3,4\}$ defined as follows:
\begin{align*}
 & J=  4  \hbox{ in case }  \{\gamma_k, k=0,1,2,3\}\geq \max \{\gamma_j, \beta_i| j=4,\dots, m, i=m+1,\dots, 2p+1\}\\
& J=0  \hbox{ in case }  \min \{\beta_k, k=m+1,m+2, m+3, m+4\}\geq \max \{\gamma_j, \beta_i| j=0,\dots, m, i=m+5,\dots, 2p+1\}\\
& 0<J<4 \hbox{ in case }  \begin{multlined}[t]
 \min \{\gamma_h, \beta_k| h=0,\dots, J-1, k=m+1,\dots , m+4-J\}\ \geq \\
\max \{\gamma_j, \beta_i| j=J,\dots, m, i=m+4-J,\dots, 2p+1\}.\end{multlined}
      \end{align*}
In the case $J=4$ we estimate
\[
\Big|\int_0^T \int \Big (\prod_{j=0}^m \partial_x^{\gamma_j} v_M\Big ) \pi_M \Big (\prod_{i=m+1}^{2p+1} \partial_x^{\beta_i} v_M\Big )\Big|
\leq \Big  \| \prod_{j=0}^m \partial_x^{\gamma_j} v_M \Big  \|_{L^1((0,T);L^1)}
\Big \| \pi_M (\prod_{i=m+1}^{2p+1} \partial_x^{\beta_i} v_M) \Big \|_{L^\infty((0,T); L^\infty)}
\]
and  by Sobolev embedding and H\"older inequality, we proceed with
\begin{align*}
(\dots) \leq & C \prod_{j=0}^3 \|\partial_x^{\gamma_j} v_M \|_{L^4((0,T);L^4)}
 \prod_{j=4}^m \|\partial_x^{\gamma_j} v_M \|_{L^\infty((0,T);L^\infty)} \Big  \| \pi_M \prod_{i=m+1}^{2p+1} \partial_x^{\beta_i} v_M \Big \|_{L^\infty((0,T); H^1)}\\
\leq  & C \prod_{j=0}^3 \|\partial_x^{\gamma_j} v_M \|_{L^4((0,T);L^4)}
 \prod_{j=4}^m \|\partial_x^{\gamma_j} v_M \|_{L^\infty((0,T);H^1))}  \prod_{i=m+1}^{2p+1} \|\partial_x^{\beta_i} v_M\|_{L^\infty((0,T); H^1)}
\end{align*}
where we used that $H^1$ is an algebra and continuity of $\pi_M$ on $H^1$. The proof
of \eqref{prodproj} can be now completed exactly following the proof of \eqref{GestimateCB}. A similar argument works for  $J=0$. In the case $0<J<4$, we alter the argument above as follows
\[
\Big|\int_0^T \int \Big  (\prod_{j=0}^m \partial_x^{\gamma_j} v_M \Big ) \pi_M \Big (\prod_{i=m+1}^{2p+1} \partial_x^{\beta_i} v_M\Big ) \Big|
\leq\Big  \| \prod_{j=0}^m \partial_x^{\gamma_j} v_M \Big \|_{L^\frac 4J((0,T);L^\frac 4J)}
\Big \| \pi_M (\prod_{i=m+1}^{2p+1} \partial_x^{\beta_i} v_M) \Big \|_{L^\frac 4{4-J}((0,T);L^\frac 4{4-J})}.
\]
By continuity of $\pi_M$ on $L^p$ with $p\in(1, \infty)$ we can remove the projector,
and then we can apply H\"older inequality twice in order to place  the four higher derivatives in $L^4$ and  the other ones in 
$L^\infty$. Then we proceed exactly following the proof of \eqref{GestimateCB}.

\subsection{Proof of \eqref{Gestimateprime}}
Due to the structure of the functionals $\bar {\mathcal G}_{k,M}^T$ we first have to prove that, if we consider a given density belonging to $\Omega_k$ then the corresponding space-time integrals computed along $v_M(t,x)$
converge to the same integral computed along $v(t,x)$. This follows from the expression under consideration being multilinear. Indeed by following the same computations we did to get \eqref{GestimateDB}, we estimate the difference of the two expressions we are interested in, by the product of several factors, where each factor involves norms of $v_{M}(t,x)$, $v(t,x)$ in either  $H^\sigma, \sigma<k-\frac 12$ or $L^6((0,T);W^{s,6}),  s<k-\frac 12$) and one factor involves the norm of $v-v_{M}$ in either $H^\sigma,
\sigma<k-\frac 12$ or $L^6((0,T);W^{s,6}),  s<k-\frac 12$. Then we conclude that the difference converges to zero by the result in Section~\ref{L6}.
The proof of \eqref{Gestimateprime} will follow provided that we prove the same convergence property as above if we now have a density belonging to $\Theta_k$
with the extra property that in the density, when computed along $v_M(t,x)$, the projector $\pi_M$ appears in front the product of a group of derivatives of $v_M(t,x)$. By using the identity $Id=\pi_M+\pi_{>M}$,  we remove the projector $\pi_M$ at the expense of a remainder. Without the remainder, convergence for the corresponding terms follows the same lines as above
in the case of a density in $\Omega_k$, as in the proof of \eqref{GestimateCB}. We now deal with a remainder 
given by the density belonging to $\Theta_k$ computed along $v_M(t,x)$ with the projector $\pi_{>M}$ in front of a product of derivatives of $v_M(t,x)$.
In order to show that this remainder goes to zero as $M\rightarrow \infty$ notice that every factor involved in the product on which the operator $\pi_{>M}$ acts, can itself be decomposed
by using the identity $Id=\pi_{cM}+\pi_{>cM}$ where $c$ is a suitable small constant. Once the
decomposition of each factor is done following this last identity, then we develop the product and notice that we get non trivial contributions after the application of the projector $\pi_{>M}$  only for the terms where at least one factor involves $\pi_{>cM} v_M(t,x)$. Then we conclude following the same chain of inequalities needed in order to get
\eqref{GestimateCB}. All factors that we get will be bounded according to estimates in Section~\ref{L6}, except one that appears with a projector
$\pi_{>cM}$ and needs to be computed in one of the norms  $H^\sigma,\sigma<k-\frac 12$ or $L^6((0,T);W^{s,6}),  s<k-\frac 12$. As we are allowed to loose $\epsilon$ derivatives in our estimates, we deduce by a straightforward argument
that we are converging to zero as $M\rightarrow \infty$.
\section{Proof of Theorem \ref{quasiinvariance}}\label{sectqi}
We denote by $d\mu_{k}$ (resp. $d\mu_{k,M}^{\perp})$ the Gaussian measure induced by the random series
\begin{equation*}
\omega\longmapsto \sum_{n\in\Z}\frac{g_n(\omega)}{(1+n^2)^{k/2}}\, e^{inx} \,\,\,\Big(\text{resp.}\,\,  \sum_{n\in\Z, |n|>M}\frac{g_n(\omega)}{(1+n^2)^{k/2}}\, e^{inx}\Big) \,,
\end{equation*}
and define the following two measures: 
$$d\rho_{k}=\chi_R({\mathcal H}(u)) d\mu_{k}, \quad d\rho_{k,M}=\chi_R({\mathcal H}(\pi_M u))d\mu_{k}.$$
Next, we shall also use the following representation
\begin{equation}\label{decompmuk}
d\mu_{k}=\gamma_M \exp{(-\|\pi_M u\|_{H^k}^2)}du_1\dots du_M \times d\mu_{k,M}^\perp\end{equation}
where $\gamma_M$ is a suitable renormalization constant and $du_1\dots du_M$
is the Lebesgue measure on $\C^M$. Next, changing variable,
\begin{align}\label{impcha}
\rho_{k, M}{(\Phi_M(T)A)}= & \gamma_M \int_{\Phi_M(t)A} \chi_R({\mathcal H}(\pi_M u))e^{-\|\pi_M u\|_{H^k}^2} du_1\dots du_M \times d\mu_{k,M}^\perp\\
\nonumber = & \gamma_M \int_{A} \chi_R({\mathcal H}(\pi_M(\Phi_M (T)u)))e^{-(\|\pi_M(\Phi_M (t)(u))\|_{H^k}^2} du_1\dots du_M \times d\mu_{k,M}^\perp\\
\nonumber = & \int_A \chi_R({\mathcal H}(\pi_M(\Phi_M (T)u)))e^{{\|\pi_M u\|_{ H^k}^2}-\|\pi_M (\Phi_M(T)u)\|_{H^k}^2} d\mu_{k}
\end{align}
where $A$ is a Borel subset of $H^{(k-\frac 12)^-}$ and we used \eqref{decompmuk} in the last identity; hence
\begin{align}\label{impcha1}
\,
\rho_{k, M}{(\Phi_M(T)A)}= &
\int_A \chi_R({\mathcal H}(\pi_M(\Phi_M (T)u))) e^{-\bar {\mathcal R}_k (\pi_M u)+\bar {\mathcal R}_k (\pi_M (\Phi_M(T)u))}
e^{\bar {\mathcal E}_k (\pi_M u)-\bar {\mathcal E}_k (\pi_M (\Phi_M(T)u)} d\mu_k\\
\nonumber= & \int_A \chi_R({\mathcal H}(\pi_M(\Phi_M (T)u)))e^{-\bar {\mathcal R}_k (\pi_M u)+\bar {\mathcal R}_k (\pi_M (\Phi_M(T)u)}
e^{- \bar {\mathcal G}_{k,M}^T (\pi_M u)} d\mu_k
\end{align}
where $\bar {\mathcal E}_k$,  $\bar {\mathcal G}_{k,M}^T$ and $\bar {\mathcal R}_k$ are the functionals in Theorem \ref{basicB}. Set
\begin{align*}
  f_{T, M}(u)= & \chi_R({\mathcal H}((\pi_M(\Phi_M (T)u))) e^{-\bar {\mathcal R}_k (\pi_M u)+\bar {\mathcal R}_k (\pi_M (\Phi_M(T)u)}
e^{-\bar {\mathcal G}^T_{k,M} (\pi_M u)}\\
  f_T(u)= & \chi_R({\mathcal H}(\Phi (T)u)) e^{-\bar {\mathcal R}_k (u)+\bar {\mathcal R}_k ((\Phi(T)u)}
e^{-\bar {\mathcal G}_{k,\infty}^T(u)}
\end{align*}
{ then by \eqref{Gestimateprime} and continuity of $\bar {\mathcal R}_k$ on $H^{(k-\frac 12)^-}$ (which follows by minor modifications of the proof of
\eqref{RestimateB}), we get $f_{T, M}(u)\rightarrow f_T(u)$ almost surely w.r.t $\mu_k$, when $M\rightarrow \infty$.}
 Moreover we claim that, for all $q\in [1, \infty)$,
\begin{equation}\label{claim}
  \sup_M \|f_{T, M}(u)\|_{L^q(\mu_k)}<\infty\,,
\end{equation}
and therefore, as $\mu_k$ is a finite measure and by using Egoroff theorem, we can upgrade to convergence in $L^q$   namely
\begin{equation}\label{Lpconv}
f_{T,M}(u) \overset{L^q(\mu_k)}\longrightarrow f_T(u)\in L^q(\mu_k) \hbox{ as } M\rightarrow \infty
\end{equation}
and we conclude, by the conservation of the Hamiltonian $\mathcal H$, $
\chi_R({\mathcal H}(u))g_T(u)=f_T(u)\in L^q(\mu_k)$.
Next we prove \eqref{claim}. {  From conservation of the Hamiltonian, ${\mathcal H}((\pi_M(\Phi_M (T)u))={\mathcal H}(\pi_M u)$ and due to $\chi_R$ there exists $C>0$, uniform w.r.t. $M$, such that
the support of $f_{T,M}$ is contained in $\{u\in H^{(k-\frac 12)^-}|\|\pi_M u\|_{H^1}<C\}$. By combining this fact with \eqref{carb} (where we choose $s=(k-\frac 12)^-$ and $\varphi=\pi_M u$)
we get that, with another constant $C>0$ uniform w.r.t. $M$, for every $u$ in the support of $f_{T,M}$,
\[
\sup_{\tau\in [0, T]}\|\pi_M \Phi_M (\tau)u\|_{H^{(k-\frac 12)^-}}\leq C \|\pi_M u\|_{H^{(k-\frac 12)^-}}\,.
\]
Gathering all together and recalling  \eqref{RestimateB} and \eqref{Gestimate} (where we choose $\varphi=\pi_M u$) we get, with $C>0$ depending on $R$, $T>0$ and uniform w.r.t. $M$
\[
  f_{T,M}(u)\leq \exp({C+C\|u \|_{H^{(k-\frac 12)^-}}^{(\frac{4k-8}{2k-3})^+}})\,.
\] }
As $(\frac{4k-8}{2k-3})<2$, we can apply classical Gaussian bounds in order to get
$$\exp({C\|u \|_{H^{(k-\frac 12)^-}}^{(\frac{4k-8}{2k-3})^+}})\in L^q(\mu_k )$$ and we conclude \eqref{claim}.
Going back to \eqref{impcha} and \eqref{impcha1} we get
\begin{equation}\label{limM}\rho_{k, M}{(\Phi_M(T)A)}=\int_A f_{T, M}(u) d\mu_k\end{equation}
and hence by passing to the limit as $M\rightarrow  \infty$ (at least formally) we get
\begin{equation}\label{limnoM}\rho_{k}{(\Phi(T)A)}=\int_A f_{T}(u) d\mu_k\end{equation}
where $f_T(u)\in L^q(\mu_k)$ and we could conclude the proof.
Proving \eqref{limM} $\Rightarrow$ \eqref{limnoM} requires some work: we will prove the following inequality 
\begin{equation}\label{limnoMless}\rho_{k}{(\Phi(T)A)}\leq \int_A f_{T}(u) d\mu_k\end{equation}
and the reversed one
\begin{equation}\label{limnoMmore}\rho_{k}{(\Phi(T)A)}\geq \int_A f_{T}(u) d\mu_k\end{equation}
which in turn imply \eqref{limnoM}.
Assume in the sequel that $A$ is a compact set in $H^{(k-\frac 12)^-}$, by classical approximation arguments this will be sufficient
to deal with any generic measurable set. In order to prove \eqref{limM} $\Rightarrow$ \eqref{limnoMless} with $A$ compact, we use the following property of the flows $\Phi_M(t)$, whose proof follows by combining Proposition
\ref{Cauchy} and the arguments in \cite{sigma}:
\begin{equation}
\forall A\subset H^{(k-\frac 12)^-} \hbox{ compact}, \forall \varepsilon>0 \quad \exists M_0\in \N \hbox{ s.t. } 
\Phi(T)A\subset \Phi_M(T) \big (A+B_\varepsilon^{(k-\frac 12)^-}\big) 
\quad \forall M>M_0,
\end{equation}
where $B_\varepsilon^s$ denotes the ball of radius $\varepsilon>0$ in $H^s$.
Then if $A$ compact  and $\varepsilon>0$ are given, we get
\[
  \rho_{k,M}{(\Phi(T)A)}\leq \rho_{k,M} \Big (\Phi_M(T) (A+B_\varepsilon^{(k-\frac 12)^-})\Big) \\=
  \int_{A+B_\varepsilon^{(k-\frac 12)^-}} f_{T, M}(u) d\mu_k \quad \forall M>M_0
\]
where we used \eqref{limM}. By using \eqref{Lpconv} we can pass to the limit as $M\rightarrow \infty$ and we get
$$\limsup_{M\rightarrow \infty} \rho_{k,M}{(\Phi(T)A)}\leq \int_{A+B_\varepsilon^{(k-\frac 12)^-}} f_{T}(u) d\mu_k, \quad \forall \varepsilon>0.$$
On the other hand one easily check that 
$\limsup_{M\rightarrow \infty} \rho_{k,M}{(\Phi(T)A)}= \rho_k (\Phi(T)A)$
and also, with $A$ compact, one can prove that 
$$\lim_{\varepsilon\rightarrow 0} \int_{A+B_\varepsilon^{(k-\frac 12)^-}} f_{T}(u) d\mu_k= \int_{A} f_{T}(u) d\mu_k$$
and \eqref{limnoMless} follows. In order to prove \eqref{limnoMmore} we take advantage of a property of the flows $\Phi_M(t)$ that follows from Proposition \ref{Cauchy}:
\begin{equation}\label{newvappr}
\forall A\subset H^{(k-\frac 12)^-} \hbox{ compact}, \forall \varepsilon>0 \quad \exists M_0\in \N \hbox{ s.t. }\,\, \forall M>M_0\,,\, \,
\Phi_M(T)A \subset \Phi(T)A+ B_\varepsilon^{(k-\frac 12)^-}\,.
\end{equation}
By combining this fact with \eqref{limM} we get for any compact set $A$:
$$\rho_{k,M}\big (\Phi(T)A+ B_\varepsilon^{(k-\frac 12)^-}\big)\geq \rho_{k,M} (\Phi_M(T)A) =
\int_{A} f_{T, M}(u) d\mu_k, \quad \forall M>M_0$$
and by passing to the limit as $M\rightarrow \infty$ on the l.h.s. and r.h.s. we get
$$\rho_{k}\big (\Phi(T)A+ B_\varepsilon^{(k-\frac 12)^-}\big)\geq 
\int_{A} f_{T}(u) d\mu_k.$$
From compactness of $A$, we easily get $\lim_{\varepsilon\rightarrow 0} \rho_{k}\big (\Phi(T)A+ B_\varepsilon^{(k-\frac 12)^-}\big )=\rho_{k}(\Phi(T)A)$
and \eqref{limnoMmore} follows.
\section{Appendix}
The aim of this appendix is to provide a proof of \eqref{tame}.
\subsection{A first reduction} 
Let $s\geq 1$ and $p\geq 3$ be integers. For the sake of completeness, we prove
\begin{equation*}
\Big\|\int_{0}^t S(t-\tau) \pi_M\Pi \Big(\Pi\big((v(\tau))^p\big)\partial_x v(\tau)\Big)d\tau\Big\|_{Y^s_T}
\leq CT^{\kappa}\|v\|_{Y^1_T}^p \|v\|_{Y^s_T},\quad \kappa>0,
\end{equation*}
for $T\leq 1$ and $C>0$ independent of $M$ and $\kappa$. The arguments that we will perform below are standard. 
Our only goal is to provide a complete argument for a reader unfamiliar with the $X^{s,b}$ machinery, as well as proving how to gain the positive power of $T$
at the r.h.s. (which was of importance for our analysis in Section~\ref{L6}). At the best of our knowledge the estimate above written in this form is not readily available in the literature, even if we closely follow \cite{CKSTT}.
As $\pi_M$ is bounded on $Y^s$, it suffices to prove 
\begin{equation*}
\Big\|\int_{0}^t S(t-\tau) \Pi (\partial_x v(\tau) \Pi  v^p(\tau))d\tau\Big\|_{Y^s_T}
\leq CT^{\kappa}\|v\|_{Y^1_T}^p \|v\|_{Y^s_T},\quad \kappa>0\,.
\end{equation*}
Recalling the definition of restriction spaces, proving the global in time estimate will be sufficient:
\begin{equation*}
\Big\|\psi(t) \int_{0}^t S(t-\tau) \Pi (\partial_x v(\tau) \Pi v^p(\tau))d\tau\Big\|_{Y^s}
\leq CT^{\kappa}\|v\|_{Y^1_T}^p \|v\|_{Y^s_T},\quad \kappa>0,
\end{equation*}
where $\psi\in C^\infty_0(\R)$ is such that $\psi\equiv 1$ on $[-1,1]$.
Using \cite[Lemma~3.1]{CKSTT}, we obtain that
\begin{equation*}
\Big\|\psi(t) \int_{0}^t S(t-\tau) \Pi (\partial_x v(\tau) \Pi v^p(\tau))d\tau\Big\|_{Y^s}
\leq C
\big\| \Pi (\partial_x v \Pi v^p )\big\|_{Z^s}\, ,
\end{equation*}
where by definition $\|u\|_{Z^s}=\|u\|_{X^{s,-\frac{1}{2}}}+\| \langle \tau+ n^3\rangle^{-1}\,\langle n\rangle^s \hat{u}(\tau,n)\|_{L^2_n L^1_\tau}$.
Therefore, proving 
\begin{equation}\label{MNMN}
\big\| \Pi ( \partial_x v \Pi v^p )\big\|_{Z^s}\leq CT^{\kappa}\|v\|_{Y^1_T}^p \|v\|_{Y^s_T},\quad \kappa>0
\end{equation}
will be enough. Its proof follows by combining the following propositions. The next statement is a slightly modified version of \cite[Theorem~3]{CKSTT}.
\begin{proposition}\label{tao1}
For  $s\geq 1$ there exists a constant $C>0$ such that
\begin{equation}\label{strongweak}
\|u^p\|_{X^{s-1,\frac{1}{2}}}\leq C\|u\|_{Y^s}\, 
\|u\|_{Y^1}^{p-1}\,.
\end{equation}
\end{proposition}
We shall also need the following bilinear estimate. 
\begin{proposition}\label{tao2}
For  $s\geq 1$ there exists $C>0$ such that, for $T\in (0,1)$,
\[
\big\| \Pi (\Pi u_1 \Pi u_2 )\|_{Z^s}
\leq
CT^\kappa \big(
\|u_1\|_{X^{s-1,\frac{1}{2}}_T}\|u_2\|_{X^{0,\frac{1}{2}}_T}
+
\|u_1\|_{X^{0,\frac{1}{2}}_T}\|u_2\|_{X^{s-1,\frac{1}{2}}_T}
\big), \quad \kappa>0 \,.
\]
\end{proposition}
Let us see how \eqref{MNMN} is a consequence of both propositions: from Proposition~\ref{tao2} we get
\[
\big\| \Pi (\partial_x v \Pi v^p)\big\|_{Z^s}
=
\big\| \Pi( \Pi \partial_x v \Pi v^p)\big\|_{Z^s}
\leq
CT^\kappa \big(
\|v^p\|_{X^{s-1,\frac{1}{2}}}\|\partial_x v\|_{X^{0,\frac{1}{2}}}
+
\| v^p\|_{X^{0,\frac{1}{2}}}\|\partial_x v\|_{X^{s-1,\frac{1}{2}}}
\big).
\]
Next, using Proposition~\ref{tao1}, we get
$$
\|v^p\|_{X^{s-1,\frac{1}{2}}}\leq C\|v\|_{Y^s}\, 
\|v\|_{Y^1}^{p-1}\, ,\quad
\| v^p\|_{X^{0,\frac{1}{2}}}
\leq C\|v\|_{Y^1}^p\,
$$
and we get \eqref{MNMN}.
\begin{proof}[Proof of Proposition~\ref{tao1}]
Write 
\begin{equation}\label{fouri}
{\mathcal F}(u^p)(\tau,n)= \int_{\tau=\tau_1+\cdots  +\tau_p}\,\,\,
\sum_{n=n_1+\cdots +n_p}\,\,
\prod_{k=1}^p \hat{u}(\tau_k,n_k)\,,
\end{equation}
where ${\mathcal F}$ and $\hat u$ denote the space time Fourier transform (continuous in time and discrete in space).
\begin{equation}\label{obj}
\|u^p\|_{X^{s-1,\frac{1}{2}}}^2=\int_{\R}\sum_{n\in\Z }\langle n\rangle^{2(s-1)} \langle \tau+n^3\rangle \, |{\mathcal F}(u^p)(\tau,n)|^2\, d\tau\,.
\end{equation}
Notice that the r.h.s. in \eqref{obj} may be bounded with
\begin{equation}\label{obj1}
\int_{\R}\sum_{n\in\Z }\langle n\rangle^{2(s-1)} \langle \tau+n^3\rangle \big (\int_{\tau=\tau_1+\cdots  +\tau_p}\,\,
\sum_{n=n_1+\cdots +n_p}\,\,
\prod_{k=1}^p |\widehat{u}(\tau_k,n_k)|\big )^2\, d\tau.
\end{equation}
Hence if we define $w(t,x)$ by $\hat{w}(\tau,n)=|\hat{u}(\tau,n)|$ we get $\|u\|_{X^{s,b}}=\|w\|_{X^{s,b}}$, $\|u\|_{Y^s}=\|w\|_{Y^s}$, and we are reduced to estimate 
\begin{equation}\label{obj2}
\int_{\R}\sum_{n\in\Z }\langle n\rangle^{2(s-1)} \langle \tau+n^3\rangle \, 
\big (\int_{\tau=\tau_1+\cdots  +\tau_p}\,\,
\sum_{n=n_1+\cdots +n_p}\,\,
\prod_{k=1}^p \widehat{w}(\tau_k,n_k)\big )^2 d\tau\,\,\,.
\end{equation}
Next we split the domain of integration and we consider first the contribution to \eqref{obj2} in the region
\begin{equation}\label{susreg1}
 |\tau+n^3|\leq 10p  |\tau_1+n_1^3|.
 \end{equation}
If we define $w_1$ by $\widehat{w_1}(\tau,n)=\langle \tau+n^3\rangle^{\frac{1}{2}}\, \widehat{w}(\tau,n)$,
then the contribution to \eqref{obj2} in the region \eqref{susreg1} can be controlled 
in the physical space variables as follows
\begin{align*}
C\|w_1 w^{p-1}\|_{L^2(\R;H^{s-1})}^2
\leq & C\big(\|w_1\|_{L^2(\R; H^{s-1})}^2 \|w^{p-1}\|_{L^\infty(\R; L^\infty)}^2+\|w_1\|_{L^2(\R; L^\infty)}^2\|w^{p-1}\|_{L^\infty(\R; H^{s-1})}^2\big)\\
\leq & C \big(\|w\|_{X^{s-1, \frac 12}}^2 \|w\|_{L^\infty(\R; H^1)}^{2(p-1)} + \|w_1\|_{L^2(\R; H^1)}^2\|w\|_{L^\infty(\R; H^{s-1})}^2\|w\|_{L^\infty(\R; H^{1})}^{2(p-2)}\big)
\end{align*}
where we have used standard product rules and Sobolev embedding $H^1\subset L^\infty$.
We proceed with
\[
  (\dots) \leq C \big(\|w\|_{X^{s-1, \frac 12}}^2 \|w\|_{Y^1}^{2(p-1)} +  \|w_1\|_{X^{1,\frac 12}}^2\|w\|_{Y^{s-1}}^2\|w\|_{Y^1}^{2(p-2)} \big)
  \]
where we used $Y^1\subset L^\infty(\R; H^1)$, $Y^{s-1}\subset L^\infty(\R; H^{s-1})$.
 Notice that we have a better estimate, when compared with 
\eqref{strongweak},
in the region \eqref{susreg1}.
Similarly, we can evaluate the contributions to \eqref{obj2} of the regions
$$
| \tau+n^3|\leq 10 p| \tau_k+n_k^3|,\quad 2\leq k\leq p\,.
$$  
Therefore, we may assume that the summation and the integration in \eqref{obj2} is performed in the region 
\begin{equation}\label{suslastreg}
\max_{1\leq k\leq p}|\tau_k+n_k^3|\leq  \frac{1}{10p} |\tau+n^3|\,.
\end{equation}
Write 
$$
(\tau+n^3)-\sum_{k=1}^p(\tau_k+n_k^3)=\Big(\sum_{k=1}^p n_k\Big)^3-\sum_{k=1}^p n_k^3\,,
$$
therefore in the region \eqref{suslastreg} we have
$$
\Big|\Big(\sum_{k=1}^p n_k\Big)^3-\sum_{k=1}^p n_k^3\Big|\geq |\tau+n^3|-\sum_{k=1}^p |\tau_k+n_k^3|\geq \frac{9}{10}|\tau+n^3|
$$
hence 
$$
\langle \tau+n^3\rangle \leq C\Big|\Big(\sum_{k=1}^p n_k\Big)^3-\sum_{k=1}^p n_k^3\Big|\,.
$$
By symmetry we can assume $|n_1|\geq |n_2|\geq \cdots \geq |n_k|$ and by using \cite[Lemma~4.1]{CKSTT}, we obtain that 
$$
\Big|\Big(\sum_{k=1}^p n_k\Big)^3-\sum_{k=1}^p n_k^3\Big|\leq C |n_1|^2 |n_2|.
$$
Consequently in the region \eqref{suslastreg} we get $\langle \tau+n^3\rangle \leq C \langle n_1\rangle^2 \langle n_2\rangle$, and the corresponding contribution to \eqref{obj2} can be estimated as
\begin{equation}\label{CC1}
C\, \int_{\R}\sum_{n\in\Z }
\,
\big (\int_{\tau=\tau_1+\cdots  +\tau_p}\,\,
\sum_{n=n_1+\cdots +n_p}
\,
\langle n_1\rangle^{s} \langle n_2\rangle^\frac 12 \, 
\prod_{k=1}^p (\widehat{w}(\tau_k,n_k)\big )^2 \, d\tau
\end{equation}
If we define $w_1$, $w_2$ by $\widehat{w_1}(\tau,n)=\langle n\rangle^{s} \widehat{w}(\tau,n)$,
 $\widehat{w_2}(\tau,n)=\langle n\rangle^{\frac{1}{2}} \widehat{w}(\tau,n)$, going back to physical space variables, we estimate \eqref{CC1} as
\begin{align*}
C\|w_1  w_2  w^{p-2}\|_{L^2(\R; L^2)}^2 \leq &
C\|w_1\|_{L^\infty(\R; L^2)}^2\|w_2\|_{L^4(\R; L^\infty)}^2\|w\|_{L^4(\R; L^\infty)}^2\|w\|_{L^\infty(\R; L^\infty)}^{2(p-3)}\\
\leq & C \|w\|_{L^\infty(\R; H^s)}^2\|w_2\|_{L^4(\R; W^{\frac 12,4})}^2\|w\|_{L^4(\R;W^{1,4})}^2\|w\|_{L^\infty(\R; H^1)}^{2(p-3)}.
\end{align*}
Hence by using $Y^1\subset L^\infty(\R;H^1)$ and $Y^s\subset L^\infty(\R;H^s)$,
along with the estimate
\begin{equation}\label{FUND}
\|u\|_{L^4(\R; L^4)}\leq C\|u\|_{X^{0,\frac{1}{3}}}
\end{equation}
established in the fundamental work \cite{B1}, we proceed with
\begin{equation*}
(\dots) \leq C \|w\|_{Y^s}^2\|w\|_{X^{1,\frac 13}}^2\|w\|_{X^{1,\frac 13}}^2\|w\|_{Y^1}^{2(p-3)}
\end{equation*}
and this concludes the proof.
\end{proof}
\begin{proof}[Proof of Proposition~\ref{tao2}]
We start with proving
\begin{equation}\label{KPV12}
\big\| \Pi (\Pi u_1\Pi u_2))\|_{X^{s,-\frac{1}{2}}}
\leq
CT^\kappa \big(
\|u_1\|_{X^{s-1,\frac{1}{2}}_T}\|u_2\|_{X^{0,\frac{1}{2}}_T}
+
\|u_1\|_{X^{0,\frac{1}{2}}_T}\|u_2\|_{X^{s-1,\frac{1}{2}}_T}
\big).
\end{equation}
As a first step we prove an estimate for global in time functions:
\begin{multline}\label{KPV1}
\big\| \Pi (\Pi u_1\Pi u_2))\|_{X^{s,-\frac{1}{2}}}
\leq
C\big(
\|u_1\|_{X^{s-1,\frac{1}{2}}}\times \|u_2\|_{X^{0,\frac{1}{3}}}+ \|u_1\|_{X^{s-1, \frac 13}} \times \|u_2\|_{X^{0,\frac 12}}
\\+
\|u_1\|_{X^{0,\frac{1}{3}}}\times \|u_2\|_{X^{s-1,\frac{1}{2}}}+ \|u_1\|_{X^{0,\frac 12}} \times \|u_2\|_{X^{s-1, \frac 13}}
\big).
\end{multline}
Notice that by comparing \eqref{KPV12} and \eqref{KPV1} in the second estimate we gain derivatives but we loose a positive power of the time $T$.
We will prove toward the end how to go from \eqref{KPV1} to \eqref{KPV12}.

The square of the left hand-side of \eqref{KPV1} may be written as 
$$
\int_{\R}\sum_{n\neq 0}\langle n\rangle^{2s}\langle \tau+n^3\rangle^{-1}|{\mathcal F}(\Pi u_1\Pi u_2)(\tau,n)|^2 \,d\tau\,,
$$
and moreover we easily have
$$
|{\mathcal F}(\Pi u_1\Pi u_2)(\tau,n)|
\leq 
\sum_{n_1\neq 0,n}\int_{\R} |\hat {u}_1(\tau_1,n_1)||\hat {u}_2(\tau-\tau_1,n-n_1)|d\tau_1\,.
$$
For $j=1,2$, define $w_j(t,x)$ with $\hat{w}_j(\tau,n)=|\hat{u}_j(\tau,n)|$. Then, we estimate the left hand-side of \eqref{KPV1} as
$$
\int_{\R}\sum_{n\neq 0}\frac{\langle n\rangle^{2s}}{\langle \tau+n^3\rangle}
\Big(
\sum_{n_1\neq 0,n}\int_{\R} \hat{w}_1(\tau_1,n_1)\hat{w}_2(\tau-\tau_1,n-n_1)d\tau_1
\Big)^2
d\tau\,
$$
which in turn by a duality argument is bounded with
$$
\sup_{\|v\|_{L^2_{t,x}}\leq 1}
\int_\R\sum_{n\neq 0}\frac{\langle n\rangle^{s}}{\langle \tau+n^3\rangle^{\frac{1}{2}}}
\sum_{n_1\neq 0,n}\int_{\R} \hat{w}_1(\tau_1,n_1)\hat{w}_2(\tau-\tau_1,n-n_1)
|\hat{v}(\tau,n)|
d\tau_1
d\tau\,.
$$
For $n\neq 0$ and $n_1\neq 0,n$, we have $\langle n\rangle^s\leq C
|n_1|^{\frac{1}{2}}
|n-n_1|^{\frac{1}{2}}
|n |^{\frac{1}{2}}
\big(
\langle n_1\rangle^{s-1}+\langle n-n_1\rangle^{s-1}
\big)$. Therefore, by using a symmetry argument, it suffices to evaluate
\begin{equation}\label{Lambda}
\sup_{\|u\|_{L^2_{t,x}}\leq 1}\int_{\R^2}
\sum_{n\neq 0,n_1\neq 0,n}
\frac{|n_1|^{\frac{1}{2}}
|n-n_1|^{\frac{1}{2}}
|n |^{\frac{1}{2}}}{ \langle\tau+n^3\rangle^{\frac{1}{2}}}
\,
\big(\langle n_1\rangle^{s-1} \hat{w}_1(\tau_1,n_1)\big) \hat{w}_2(\tau-\tau_1,n-n_1)
|\hat{v}(\tau,n)|
d\tau_1
d\tau\,.
\end{equation}
The key property for smoothing is the elementary bound 
$$
\max
\Big(\langle\tau+n^3\rangle^{\frac{1}{2}},
\langle\tau_1+n_1^3\rangle^{\frac{1}{2}},
\langle \tau-\tau_1+(n-n_1)^3\rangle^{\frac{1}{2}}\Big)\geq C
|n_1|^{\frac{1}{2}}
|n-n_1|^{\frac{1}{2}}
|n |^{\frac{1}{2}}\,.
$$
We will consider a splitting of the expression in \eqref{Lambda} in three contributions taking into account which term is the maximum in the above elementary bound. 
Notice that the contribution of \eqref{Lambda} in the following region
\begin{equation}\label{Lambdaprime}
\langle\tau+n^3\rangle^{\frac{1}{2}}\geq C|n_1|^{\frac{1}{2}}
|n-n_1|^{\frac{1}{2}}
|n |^{\frac{1}{2}}\,
\end{equation}
may be estimated as
$$
C\|w_2 \times v_1 \times \langle D_x\rangle^{s-1} w_1 \|_{L^1(\R;L^1)}\leq
C\|v\|_{L^2_{t,x}}\|\langle D_x\rangle^{s-1}w_1\|_{L^4(\R;L^4)}\|w_2\|_{L^4(\R;L^4)}\,
$$
where  $v_1(t,x)$ is defined with $\hat{v}_1(\tau,n)=|\hat{v}(\tau,n)|$. Now, using \eqref{FUND}, we write
\begin{align*}
  \|\langle D_x\rangle^{s-1}w_1\|_{L^4(\R;L^4)}
\leq & C\|w_1\|_{X^{s-1,\frac{1}{3}}}=C\|u_1\|_{X^{s-1,\frac{1}{3}}}\\
\|w_1\|_{L^4(\R;L^4)}
\leq  & C\|w_1\|_{X^{0,\frac{1}{3}}}=C\|u_1\|_{X^{0,\frac{1}{3}}}\,.
\end{align*}
Hence we can estimate the contribution to \eqref{Lambda} in the region \eqref{Lambdaprime}
by $$\|u_1\|_{X^{s-1,\frac{1}{3}}}\times \|u_2\|_{X^{0,\frac{1}{3}}}$$
up to a multiplicative constant.
Next we consider the contribution to \eqref{Lambda} in the region 
\begin{equation}\label{regiontwo}
\langle\tau_1+n_1^3\rangle^{\frac{1}{2}}\geq C|n_1|^{\frac{1}{2}}
|n-n_1|^{\frac{1}{2}}
|n |^{\frac{1}{2}}\,.
\end{equation}
Let $v_1(t,x)$ be defined by $$\widehat{v}_1(\tau,n)=\langle \tau+n^3\rangle^{-\frac{1}{2}} |\hat{v}(\tau,n)|$$ and let $w_{11}(t,x)$ be defined as
$$
\hat{w}_{11}(\tau,n)=\langle n\rangle^{s-1} \langle \tau+n^3\rangle^{\frac{1}{2}} \hat{w}_1(\tau,n)\,,
$$ 
then we can estimate the contribution of\eqref{Lambda} in the region \eqref{regiontwo} by
$$
C\|w_{11} \times w_2 \times v_1\|_{L^1(\R;L^1)}
\leq C\|w_{11}\|_{L^2(\R;L^2)}\|w_2\|_{L^4(\R;L^4)}\|v_1\|_{L^4(\R;L^4)}
$$
Using again \eqref{FUND} we obtain that
\begin{align*}
  \|v_1\|_{L^4(\R;L^4)}\leq & C\|v\|_{X^{0,-\frac{1}{6}}}\leq C \|v\|_{L^2(\R;L^2)}\,,\\
\|w_2\|_{L^4(\R;L^4)}\leq & C\|w_2\|_{X^{0,\frac{1}{3}}}\leq C \|u_2\|_{X^{0,\frac{1}{3}}}\,.
\end{align*}
Moreover we have
$$
\|w_{11}\|_{L^2(\R;L^2)}=\|w_1\|_{X^{s-1,\frac{1}{2}}}=\|u_1\|_{X^{s-1,\frac{1}{2}}}
$$
and summarizing we control the contribution to \eqref{Lambda} in the region \eqref{regiontwo} by 
$\|u_1\|_{X^{s-1,\frac{1}{2}}}\times \|u_2\|_{X^{0,\frac{1}{3}}}$
up to a multiplicative factor. Finally consider the contribution to \eqref{Lambda} on the third region 
\begin{equation}\label{thirdregion}
\langle\tau-\tau_1+(n-n_1)^3\rangle^{\frac{1}{2}}\geq C|n_1|^{\frac{1}{2}}
|n-n_1|^{\frac{1}{2}}
|n |^{\frac{1}{2}}\,.
\end{equation}
Retain $v_1(t,x)$ with $\hat{v}_1(\tau,n)=\langle \tau+n^3\rangle^{-\frac{1}{2}} |\hat{v}(\tau,n)|$ and let $w_{21}(t,x)$ be defined with
$\hat{w}_{21}(\tau,n)= \langle \tau+n^3\rangle^{\frac{1}{2}} \hat{w}_2(\tau,n)$. Then we can control the contribution to \eqref{Lambda} in the region \eqref{thirdregion} by 
\begin{align*}
C \| w_{21}\times  v_1 \times \langle D_x\rangle^{s-1} w_1\|_{L^1(\R;L^1)}
\leq & C
\|\langle D_x\rangle^{s-1} w_1\|_{L^4(\R;L^4)}
\|w_{21}\|_{L^2(\R;L^2)}\|v_1\|_{L^4(\R;L^4)}\,\\
\leq & C\|u_1\|_{X^{s-1, \frac 13}} \|u_2\|_{X^{0,\frac 12}}\|v\|_{X^{0, -\frac 16}} 
\end{align*}
where we have used again \eqref{FUND}. Hence the contribution
of \eqref{Lambda} in the region \eqref{thirdregion} can be estimated, up to a multiplicative constant, by 
$\|u_1\|_{X^{s-1, \frac 13}} \times \|u_2\|_{X^{0,\frac 12}}$. Summarizing, we estimate \eqref{Lambda} by 
$$\|u_1\|_{X^{s-1,\frac{1}{2}}}\times \|u_2\|_{X^{0,\frac{1}{3}}}+ \|u_1\|_{X^{s-1, \frac 13}} \times \|u_2\|_{X^{0,\frac 12}}$$
for functions $u_1$ and $u_2$ which are not localized in time.
Recall that by symmetry, in order to estimate the l.h.s. in \eqref{KPV1}, we need to add further terms where the role of $u_1$ and $u_2$ has exchanged.
Hence we have established \eqref{KPV1} which, as already said, in some sense is stronger than \eqref{KPV12}, since less conormal derivatives of $u_1$ and $u_2$ are involved, but 
it is weaker than \eqref{KPV12} since no gain of positive power of $T$ has been obtained in \eqref{KPV1}.
We finally deal with this issue: as a consequence of \cite[Lemma~2.11]{Tao}  (see also \cite[Lemma~3.2]{CO}) there exists $\kappa>0$ such that
\[
\|w\|_{X^{s-1, \frac 13}_T}\leq C T^\kappa  \|w\|_{X^{s-1, \frac 12}_T}, \quad 
\|w\|_{X^{0, \frac 13}_T}\leq C T^\kappa  \|w\|_{X^{0, \frac 12}_T}, 
\]
and we complete the proof of \eqref{KPV12} by combining the estimates above with \eqref{KPV1}, along with a suitable choice of extensions
for $u_1$ and $u_2$, which a priori are defined on the strip of time $(-T,T)$, on the whole space time with a global norm of the extension which is at most twice the corresponding localized norm.
\end{proof}

\end{document}